%% file: fig8_kappa.tex
\documentclass[a4paper,dvipdfmx]{amsart}
\usepackage{amssymb}
\usepackage{mathrsfs}
\usepackage{stmaryrd}
\usepackage{verbatim}
\usepackage{epic}
\usepackage{eepic}
\usepackage{hyperref}
\usepackage{graphicx}
\usepackage{tikz}
\usetikzlibrary{intersections,calc,arrows.meta,patterns,cd}
\definecolor{refkey}{rgb}{1,0,0}
\definecolor{labelkey}{rgb}{0,0,1}
\theoremstyle{plain}
  \newtheorem{thm}{Theorem}[section]
  \newtheorem{lem}[thm]{Lemma}
  \newtheorem{cor}[thm]{Corollary}
  \newtheorem{prop}[thm]{Proposition}
  \newtheorem{conj}[thm]{Conjecture}

  \newtheorem*{obs*}{Observation}
\theoremstyle{definition}
  \newtheorem{defn}[thm]{Definition}

\theoremstyle{remark}
  \newtheorem{rem}[thm]{Remark}

\newcommand{\Z}{\mathbb{Z}}
\newcommand{\C}{\mathbb{C}}
\newcommand{\R}{\mathbb{R}}
\newcommand{\Q}{\mathbb{Q}}
\newcommand{\Vol}{\operatorname{Vol}}

\newcommand{\CS}{\operatorname{CS}}
\newcommand{\Li}{\operatorname{Li}}
\newcommand{\Hom}{\operatorname{Hom}}
\renewcommand{\L}{\mathcal{L}}
\newcommand{\cs}{\mathbf{cs}}
\newcommand{\cv}{\operatorname{cv}}

\newcommand{\SL}{\rm{SL}}

\newcommand{\Khyp}{\mathscr{H}}
\newcommand{\arccosh}{\operatorname{arccosh}}

\newcommand{\floor}[1]{\lfloor#1\rfloor}
\newcommand{\pic}[2]{\raisebox{-0.5\height}{\includegraphics[scale=#1]{#2.eps}}}
\renewcommand{\Re}{\operatorname{Re}}
\renewcommand{\Im}{\operatorname{Im}}

\newcommand{\Int}{\operatorname{Int}}
\newcommand{\FE}{\mathscr{E}}
\newcommand{\St}{\mathscr{S}}
\renewcommand{\i}{\sqrt{-1}}
\newcommand{\saddle}{\sigma}
\makeatletter
\newcommand{\oset}[3][0ex]{%
  \mathrel{\mathop{#3}\limits^{
    \vbox to#1{\kern-1\ex@
    \hbox{$\scriptstyle#2$}\vss}}}}
\makeatother
\newcommand{\Rpath}{\oset{\frown}{\mathbb{R}}}
\renewcommand{\theenumi}{\roman{enumi}}
\numberwithin{equation}{section}
\allowdisplaybreaks
\begin{document}
\title[The asymptotic behaviors of the colored Jones polynomials]
{The asymptotic behaviors of the colored Jones polynomials of the figure eight-knot, and an affine representation}
\author{Hitoshi Murakami}
\address{
Graduate School of Information Sciences,
Tohoku University,
Aramaki-aza-Aoba 6-3-09, Aoba-ku,
Sendai 980-8579, Japan}
\email{hitoshi@tohoku.ac.jp}
\date{\today}
\begin{abstract}
We study the asymptotic behavior of the $N$-dimensional colored Jones polynomial of the figure-eight knot evaluated at $\exp\bigl((\kappa+2p\pi\i/N\bigr)$, where $\kappa:=\arccosh(3/2)$ and $p$ is a positive integer.
We can prove that it grows exponentially with growth rate determined by the Chern--Simons invariant of an affine representation from the fundamental group of the knot complement to the Lie group $\SL(2;\C)$.
\end{abstract}
\thanks{The author is supported by JSPS KAKENHI Grant Numbers JP20K03601, JP22H01117, JP20K03931.}
\maketitle
%%%%%%%%%%%%%%%%%%%%%%%%%%%%%%%%%%%%%%%%%%%%%%%%%%%%%%%%%%%%%%%%%%%
\input{intro.tex}
\input{dilog.tex}
\input{sum.tex}
\input{Poisson.tex}
\input{saddle}
\input{CS.tex}
\appendix
\input{Poisson_proof}
\input{saddle_proof}
\input{stevedore.tex}

%%%%%%%%%%%%%%%%%%%%%%%%%%%%%%%%%%%%%%%%%%%%%%%%%%%%%%%%%%%%%%%%%%
\bibliography{mrabbrev,hitoshi}
\bibliographystyle{amsplain}
\end{document}

%% file: intro.tex
\section{Introduction}\label{sec:intro}
For a knot $K$ in the three-dimensional sphere $S^3$ and a positive integer $N$, let $J_N(K;q)$ be the $N$-dimensional colored Jones polynomial associated with the $N$-dimensional irreducible representation of the Lie algebra $\mathfrak{sl}(2;\C)$, where we normalize it as $J_N(U;q)=1$ for the unknot $U$, and when $N=2$, it satisfies the following skein relation:
\begin{equation*}
  qJ_2\left(\raisebox{0.5\height}{\pic{0.1}{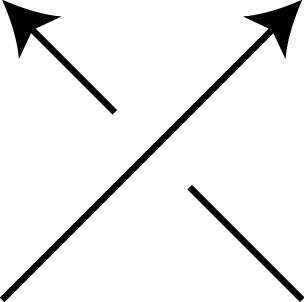}};q\right)
  -
  q^{-1}J_2\left(\raisebox{0.5\height}{\pic{0.1}{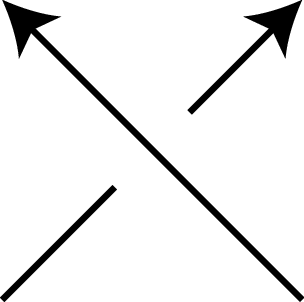}};q\right)
  =
  (q^{1/2}-q^{-1/2})J_2\left(\raisebox{0.5\height}{\pic{0.1}{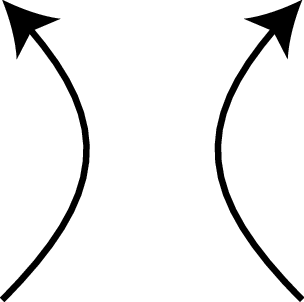}};q\right).
\end{equation*}
\par
If we replace $q$ with $e^{2\pi\i/N}$, we obtain a complex number $J_N(K;e^{2\pi\i/N})$, which is known as the Kashaev invariant \cite{Kashaev:MODPLA95,Murakami/Murakami:ACTAM12001}.
The volume conjecture \cite{Kashaev:LETMP97,Murakami/Murakami:ACTAM12001} states that the series $\{J_N(K;e^{2\pi\i/N})\}_{N=1,2,3,\dots}$ would grow exponentially with growth rate proportional to the simplicial volume $\Vol(K)$ of $S^3\setminus{K}$.
Here the simplicial volume is also known as the Gromov norm \cite{Gromov:INSHE82}.
It coincides with the hyperbolic volume divided by the volume $v_3$ of the ideal regular hyperbolic tetrahedron if the knot is hyperbolic, that is, its complement $S^3\setminus{K}$ processes a complete hyperbolic structure with finite volume.
If the knot is not hyperbolic, then the simplicial volume times $v_3$ is the sum of the hyperbolic volumes of the hyperbolic pieces of $S^3\setminus{K}$ after the Jaco--Shalen--Johannson decomposition \cite{Jaco/Shalen:MEMAM1979,Johannson:1979}.
\begin{conj}[Volume conjecture]
For any knot $K$ in $S^3$, we have
\begin{equation*}
  \lim_{N\to\infty}
  \frac{\log\left|J_N(K;e^{2\pi\i/N})\right|}{N}
  =
  \frac{\Vol(K)}{2\pi}.
\end{equation*}
\end{conj}
The volume conjecture has been generalized in various ways.
It can be complexified as follows \cite{Murakami/Murakami/Okamoto/Takata/Yokota:EXPMA02}.
Let $\Khyp\subset S^3$ be a hyperbolic knot, and $\cv(\Khyp):=\i\Vol(\Khyp)-2\pi^2\CS^{\rm{SO(3)}}(\Khyp)$ be the complex volume of $S^{3}\setminus\Khyp$, where $\CS^{\rm{SO(3)}}(\Khyp)\pmod{\pi^2}$ is the Chern--Simons invariant of the Levi--Civita connection of $S^3\setminus\Khyp$ associated with the complete hyperbolic structure.
\begin{conj}[Complexification of the volume conjecture]
For any hyperbolic knot $\Khyp$ in $S^3$, we have
\begin{equation*}
  \lim_{N\to\infty}
  \frac{\log J_N(\Khyp;e^{2\pi\i/N})}{N}
  =
  \frac{\cv(\Khyp)}{2\pi\i}.
\end{equation*}
\end{conj}
The volume conjecture and its complexification can be refined as follows \cite{Gukov:COMMP2005} (see also \cite{Gukov/Murakami:FIC2008,Dimofte/Gukov/Lenells/Zagier:CNTP2010,Ohtsuki:QT2016}):
\begin{conj}[Refined volume conjecture]
Let $\Khyp\subset S^3$ be a hyperbolic knot.
Then the following asymptotic equivalence holds.
\begin{equation*}
  J_N(\Khyp;e^{2\pi\i/N})
  \underset{N\to\infty}{\sim}
  \left(\frac{T(\Khyp)}{2\i}\right)^{1/2}N^{3/2}
  \exp\left(\frac{\cv(\Khyp)}{2\pi\i}N\right),
\end{equation*}
where $F(N)\underset{N\to\infty}{\sim}G(N)$ means $\lim_{N\to\infty}\frac{F(N)}{G(N)}=1$, and $T(\Khyp)$ is the adjoint {\rm(}cohomological{\rm)} Reidemeister torsion twisted by the holonomy representation $\rho_0\colon\pi_1(S^3\setminus{\Khyp})\to\SL(2;\C)$.
\end{conj}
The refined volume conjecture has been proved for the figure-eight knot \cite{Andersen/Hansen:JKNOT2006}, and hyperbolic knots with as most seven crossings \cite{Ohtsuki:QT2016,Ohtsuki/Yokota:MATPC2018,Ohtsuki:INTJM62017}.
\par
We can also generalize the refined volume conjecture by replacing $2\pi\i$ in $e^{2\pi\i/N}$ with a complex number.
\par
Let $\rho_u\colon\pi_1(S^3\setminus\Khyp)\to\SL(2;\C)$ be an irreducible representation, which is a small deformation of the holonomy representation $\rho_0$.
Then it defines an incomplete hyperbolic structure of $S^3\setminus\Khyp$.
Up to conjugation, we can assume that $\rho_u$ sends the meridian (preferred longitude, respectively) of $\Khyp$ to $\begin{pmatrix}e^{u/2}&\ast\\0&e^{-u/2}\end{pmatrix}$ ($\begin{pmatrix}e^{v(u)/2}&\ast\\0&e^{-v(u)/2}\end{pmatrix}$, respectively).
Then we can define the cohomological adjoint Reidemeister torsion $T_{u}(\Khyp)$ \cite{Porti:MAMCAU1997}, and the Chern--Simons invariant $\CS_{u,v(u)}(\rho_u)$ \cite{Kirk/Klassen:COMMP1993}.
\par
The following conjecture was proposed in \cite{Murakami:JTOP2013} (see also \cite{Gukov/Murakami:FIC2008,Dimofte/Gukov:Columbia}).
\begin{conj}[Generalized volume conjecture]\label{conj:GVC}
For a hyperbolic knot $\Khyp$, there exists a neighbourhood $U\in\C$ of $0$ such that if $u\in U\setminus\pi\i\Q$, then the following asymptotic equivalence holds.
\begin{multline*}
  J_N(\Khyp;e^{(u+2\pi\i)/N})
  \\
  \underset{N\to\infty}{\sim}
  \frac{\sqrt{-\pi}}{2\sinh(u/2)}T_{u}(\Khyp)^{1/2}
  \left(\frac{N}{u+2\pi\i}\right)^{1/2}
  \exp\left(\frac{S_{u}(\Khyp)}{u+2\pi\i}N\right),
\end{multline*}
where $S_{u}(\Khyp):=\CS_{u,v(u)}(\rho_u)+u\pi\i+uv(u)/4$.
\end{conj}
The generalized volume has been proved just for the figure-eight knot \cite{Murakami/Yokota:JREIA2007}.
The asymptotic equivalence in Conjecture~\ref{conj:GVC} was also proved in the case where $0<u<\kappa:=\arccosh(3/2)$ in \cite{Murakami:JTOP2013}.
\par
In the previous paper \cite{Murakami:CANJM2023}, the author proved the following theorem generalizing the result in \cite{Murakami:JTOP2013}.
\begin{thm}\label{thm:u_p}
Let $\FE$ be the figure-eight knot.
For a real number $u$ with $0<u<\kappa$ and a positive integer $p$, we have
\begin{equation*}
\begin{split}
  &J_N(\FE;e^{(u+2p\pi\i)/N})
  \\
  \underset{N\to\infty}{\sim}&
  J_{p}(\FE;e^{4N\pi^2/(u+2p\pi\i)})
  \\
  &\times
  \frac{\sqrt{-\pi}}{2\sinh(u/2)}T_{u}(\FE)^{1/2}
  \left(\frac{N}{u+2p\pi\i}\right)^{1/2}
  \exp\left(\frac{S_{u}(\FE)}{u+2p\pi\i}N\right).
\end{split}
\end{equation*}
\end{thm}
Note that in the case of the figure-eight knot, we have
\begin{align*}
  T_{u}(\FE)
  &=
  \frac{2}{\sqrt{(2\cosh{u}+1)(2\cosh{u}-3)}},
  \\
  S_{u}(\FE)
  &=
  \Li_2\left(e^{-u-\varphi(u)}\right)
  -
  \Li_2\left(e^{-u+\varphi(u)}\right)
  +
  u\bigl(\varphi(u)+2\pi\i\bigr),
\end{align*}
where we put
\begin{equation*}
  \varphi(u)
  :=
  \log
  \left(
    \cosh{u}-\frac{1}{2}
    -
    \frac{1}{2}\sqrt{(2\cosh{u}+1)(2\cosh{u}-3)}
  \right)
\end{equation*}
and $\Li_2(z):=-\int_{0}^{z}\frac{\log(1-x)}{x}\,dx$ is the dilogarithm function.
\par
So it is impossible to extend Theorem~\ref{thm:u_p} to the case where $u=\kappa$ because $T_u(\FE)$ is not defined.
Topologically/geometrically speaking, the corresponding hyperbolic structure of the figure-eight knot complement collapses at $u=\kappa$.
\par
On the other hand, for the figure-eight knot, we have the following theorems.
\begin{thm}[\cite{Murakami:JPJGT2007}]
If $\zeta\in\C$ satisfies the inequality $|\cosh\zeta-1|<1/2$ and $|\Im\zeta|<\pi/3$, then we have
\begin{equation*}
  \lim_{N\to\infty}
  J_N(\FE;e^{\zeta/N})
  =
  \frac{1}{\Delta(\FE;e^{\zeta})},
\end{equation*}
where $\Delta(K;t)$ is the Alexander polynomial of a knot $K$.
\end{thm}
\begin{thm}[\cite{Hikami/Murakami:COMCM2008}]
If $\zeta=\kappa$, then the colored Jones polynomial $J_N(\FE;e^{\kappa/N})$ grows polynomially.
More precisely, we have
\begin{equation*}
  J_N\left(\FE;e^{\kappa/N}\right)
  \underset{N\to\infty}{\sim}
  \frac{\Gamma(1/3)}{3^{2/3}}\left(\frac{N}{\kappa}\right)^{2/3},
\end{equation*}
where $\Gamma(x)$ is the gamma function.
\end{thm}
In this paper, we will extend Theorem~\ref{thm:u_p} to the case $u=\kappa$.
\begin{thm}\label{thm:main}
Let $\FE$ be the figure-eight knot, and $\xi:=\kappa+2p\pi\i$ with $\kappa:=\arccosh(3/2)$ and $p$ a positive integer.
Then we have the following asymptotic equivalence.
\begin{equation*}
  J_N\left(\FE;e^{\xi/N}\right)
  \underset{N\to\infty}{\sim}
  J_p\left(\FE;e^{4\pi^2N/\xi}\right)
  \frac{\Gamma(1/3)e^{\pi\i/6}}{3^{1/6}}
  \left(\frac{N}{\xi}\right)^{2/3}\exp\left(\frac{S_{\kappa}(\FE)}{\xi}N\right),
\end{equation*}
where $S_{\kappa}(\FE):=2\kappa\pi\i$, and we put $\xi^{1/3}:=|\xi|^{1/3}e^{\arctan(2p\pi/\kappa)\i/3}$.
\end{thm}
As a corollary, we obtain a similar result for $J_N\left(\FE;e^{\xi'/N}\right)$ with $\xi':=-\kappa+2p\pi\i$.
\begin{cor}\label{cor:conjugate}
We have
\begin{equation*}
  J_N\left(\FE;e^{\xi'/N}\right)
  \underset{N\to\infty}{\sim}
  J_p\left(\FE;e^{4\pi^2N/\xi'}\right)
  \frac{\Gamma(1/3)e^{\pi\i/6}}{3^{1/6}}
  \left(\frac{N}{\xi'}\right)^{2/3}
  \exp\left(\frac{S_{-\kappa}(\FE)}{\xi'}N\right),
\end{equation*}
where we put $S_{-\kappa}(\FE):=-2\kappa\pi\i$.
\end{cor}
See \S~\ref{sec:CS} for a topological interpretation of $S_{u}(\FE)$ for $|u|\le\kappa$.
It is defined to be $\CS_{u,v(u)}(\rho_{u})+u\pi\i+uv(u)/4$, where $\CS_{u,v(u)}(\rho_{u})$ is the Chern--Simons invariant of a non-Abelian representation $\rho_{u}\colon\pi_1(S^3\setminus\FE)\to\SL(2;\C)$.
\begin{rem}
Since the highest degree term of the Laurent polynomial $J_p(\FE;q)$ is $q^{p(p-1)}$, we have $J_p\left(\FE;e^{4\pi^2N/\xi}\right)\underset{N\to\infty}{\sim}e^{4p(p-1)\pi^2N/\xi}$.
So we also have
\begin{equation*}
  J_N\left(\FE;e^{\xi/N}\right)
  \underset{N\to\infty}{\sim}
  \frac{\Gamma(1/3)e^{\pi\i/6}}{3^{1/6}}
  \left(\frac{N}{\xi}\right)^{2/3}\exp\left(\frac{4p^2\pi^2}{\xi}N\right)
\end{equation*}
because $2\kappa\pi\i/\xi+4p(p-1)\pi^2/\xi=4p^2\pi^2/\xi+2\pi\i$.
A similar result holds for $\xi'$.
\end{rem}
There are two difficulties to prove Theorem~\ref{thm:main}.
\par
The first one is that when we apply the saddle point method to the integral that approximates $J_N(\FE;e^{\xi/N})$, the saddle point is of order two, that is, it looks like the saddle point of $\Re{z^3}$.
See Figure~\ref{fig:z^3_z^2}.
\begin{figure}[h]
\pic{0.6}{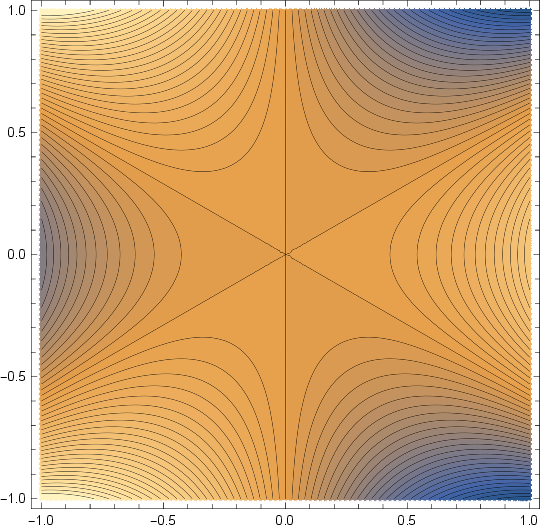}\qquad\pic{0.6}{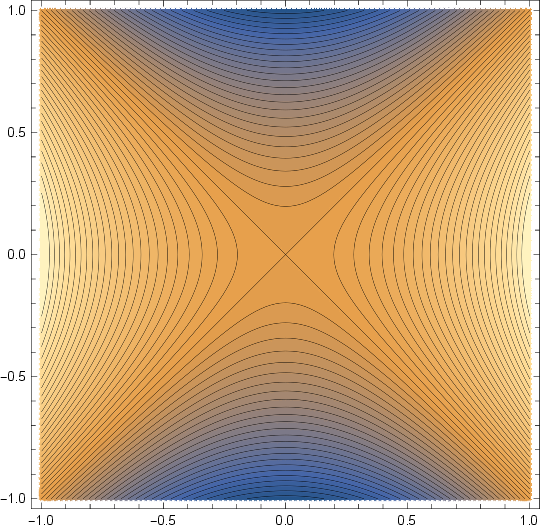}
\caption{Contour plots of the functions $\Re{z^3}$ (left) and $\Re{z^2}$ (right) around their saddle points.
The saddle point $O$ of $\Re{z^3}$ is of order two, and that of $\Re{z^2}$ is of order one.}
\label{fig:z^3_z^2}
\end{figure}
\par
To approximate the colored Jones polynomial by an integral as above, we use a quantum dilogarithm function, which converges to a function described by the dilogarithm function.
However, the second difficulty is that our saddle point is on the boundary of the region of convergence.
So we need to extend the domain of definition of the quantum dilogarithm slightly by using a functional identity.
\par
The paper is organized as follows.
In Section~\ref{sec:dilog}, we define the quantum dilogarithm and extend it as we require.
In Section~\ref{sec:sum}, we express the colored Jones polynomial as a sum of the terms described by the quantum dilogarithm.
To approximate the sum by an integral, we use the Poisson summation formula in Section~\ref{sec:Poisson}.
Then, in Section~\ref{sec:saddle} we use the saddle point method to obtain the asymptotic formula, proving Theorem~\ref{thm:main}.
Appendices~\ref{sec:Poisson_proof} and \ref{sec:saddle_proof} are devoted to proofs of the Poisson summation formula and the saddle point method, respectively.
In Appendix~\ref{sec:stevedore}, we give some computer calculations about the asymptotic behavior of  $J_N(\St;e^{(\pm\tilde{\kappa}+2\pi\i)/N})$ for the stevedore knot $\St$, where $\tilde{\kappa}:=\log{2}$.
Since we know that $e^{\pm\kappa}$ ($e^{\pm\tilde{\kappa}}$, respectively) are zeros of the Alexander polynomial of the figure-eight knot (the stevedore knot, respectively), we try to generalize Theorem~\ref{thm:main} to another knot in vain.

%% file: dilog.tex
\section{Quantum dilogarithm}\label{sec:dilog}
In this section, we fix a complex number $\gamma$ with $\Re{\gamma}>0$ and $\Im{\gamma}<0$.
We will introduce a quantum dilogarithm \cite{Faddeev:LETMP1995}.
See also \cite{Kashaev:LETMP97,Andersen/Hansen:JKNOT2006,Ohtsuki:QT2016}.
\par
We put
\begin{equation}\label{eq:def_TN}
  T_{N}(z)
  :=
  \frac{1}{4}\int_{\Rpath}\frac{e^{(2z-1)x}}{x\sinh(x)\sinh(\gamma x/N)}\,dx
\end{equation}
for an integer $N>|\gamma|/\pi$, where $\Rpath:=(-\infty,-1]\cup\{w\in\C\mid|w|=1,\Im{w}\ge0\}\cup[1,\infty)$ with orientation from $-\infty$ to $\infty$.
Note that $\Rpath$ avoids the poles of the integrand.
We can prove that the integral above converges if $-\frac{\Re\gamma}{2N}<\Re{z}<1+\frac{\Re\gamma}{2N}$.
\begin{lem}\label{lem:TN_convergence}
The integral in the right-hand side of \eqref{eq:def_TN} converges if $-\frac{\Re\gamma}{2N}<\Re{z}<1+\frac{\Re\gamma}{2N}$.
\end{lem}
%%%%%%%%%%%%%%%%%%%%%%%%%%%%%%%%%%%%%%%%%%%%%%%%%%%%%%%%%%%%%%%%%%%%%
\begin{proof}
First note that
\begin{equation*}
  \sinh(as)\underset{s\to\infty}{\sim}\frac{e^{as}}{2},
  \quad\text{and}\quad
  \sinh(as)\underset{s\to-\infty}{\sim}\frac{-e^{-as}}{2},
\end{equation*}
for a complex number $a$ with $\Re{a}>0$.
So we have
\begin{align*}
  \frac{e^{(2z-1)x}}{x\sinh(x)\sinh(\gamma x/N)}
  &\underset{x\to\infty}{\sim}
  \frac{4}{x}\exp\bigl((2z-2-\gamma/N)x\bigr),
  \\
  \frac{e^{(2z-1)x}}{x\sinh(x)\sinh(\gamma x/N)}
  &\underset{x\to-\infty}{\sim}
  -\frac{4}{x}\exp\bigl((2z+\gamma/N)x\bigr)
\end{align*}
since $\Re{\gamma}>0$.
\par
Therefore if $-\frac{\Re\gamma}{2N}<\Re{z}<1+\frac{\Re\gamma}{2N}$ the integral converges.
\end{proof}
Thus $T_{N}(z)$ is a holomorphic function in the region $\left\{z\in\C\mid-\frac{\Re\gamma}{2N}<\Re{z}<1+\frac{\Re\gamma}{2N}\right\}$.
\par
We will study properties of $T_{N}(z)$.
\par
First of all we introduce three related functions:
\begin{defn}
For a complex number $z$ with $0<\Re{z}<1$, we put
\begin{align*}
  \L_0(z)
  &:=
  \int_{\Rpath}\frac{e^{(2z-1)x}}{\sinh(x)}\,dx,
  \\
  \L_1(z)
  &:=
  -\frac{1}{2}\int_{\Rpath}\frac{e^{(2z-1)x}}{x\sinh(x)}\,dx,
  \\
  \L_2(z)
  &:=
  \frac{\pi\i}{2}\int_{\Rpath}\frac{e^{(2z-1)x}}{x^2\sinh(x)}\,dx.
\end{align*}
\end{defn}
In a similar way to the proof of Lemma~\ref{lem:TN_convergence}, we can prove that the three integral above converge if $0<\Re{z}<1$.
The functions above can be expressed in terms of well-known functions.
\begin{lem}[{\cite[Lemma~2.5]{Murakami/Tran:2021}}]\label{lem:L0_L1_L2}
We have the following formulas.
\begin{align}
  \L_0(z)
  &=
  \frac{-2\pi\i}{1-e^{-2\pi\i z}},
  \label{eq:L_0}
  \\
  \L_1(z)
  &=
  \begin{cases}
    \log\left(1-e^{2\pi\i z}\right)&\quad\text{if $\Im{z}\ge0$},
    \\
    \pi\i(2z-1)+\log\left(1-e^{-2\pi\i z}\right)&\quad\text{if $\Im{z}<0$},
  \end{cases}
  \label{eq:L_1}
  \\
  \L_2(z)
  &=
  \begin{cases}
    \Li_2\left(e^{2\pi\i z}\right)&\quad\text{if $\Im{z}\ge0$},
    \\
    \pi^2\left(2z^2-2z+\frac{1}{3}\right)-\Li_2\left(e^{-2\pi\i z}\right)
    &\quad\text{if $\Im{z}<0$}.
  \end{cases}
  \label{eq:L_2}
\end{align}
Here the branch cuts of $\log$ and $\Li_2$ are $(-\infty,0]$ and $[1,\infty)$, respectively.
\end{lem}
The proof is similar to that of \cite[Lemma~2.5]{Murakami/Tran:2021}, and so we omit it.
\par
The function $\L_0(z)$ can be extended to the whole complex plane $\C$ except for integers.
The functions $\L_1(z)$ and $\L_2(z)$ can be extended to holomorphic functions on $\C\setminus\bigl((-\infty,0]\cup[1,\infty)\bigr)$ as follows.
\begin{defn}\label{defn:L1_L2}
For a complex number $z$ in $\C\setminus\bigl((-\infty,0]\cup[1,\infty)\bigr)$, we put
\begin{align}
  \L_1(z)
  &=
  \begin{cases}
    \log\left(1-e^{2\pi\i z}\right)&\quad\text{if $\Im{z}\ge0$},
    \\
    \pi\i(2z-1)+\log\left(1-e^{-2\pi\i z}\right)&\quad\text{if $\Im{z}<0$},
  \end{cases}
  \label{eq:L1_log}
  \\
  \L_2(z)
  &=
  \begin{cases}
    \Li_2\left(e^{2\pi\i z}\right)&\quad\text{if $\Im{z}\ge0$},
    \\
    \pi^2\left(2z^2-2z+\frac{1}{3}\right)-\Li_2\left(e^{-2\pi\i z}\right)
    &\quad\text{if $\Im{z}<0$}.
  \end{cases}
  \label{eq:L2_dilog}
\end{align}
\end{defn}
\begin{lem}\label{lem:L2_Li2}
When $\Im{z}<0$, the functions $\L_1(z)$ and $\L_2(z)$ can also be written as
\begin{align*}
  \L_1(z)
  =&
  \log\left(1-e^{2\pi\i z}\right)+2\floor{\Re{z}}\pi\i,
  \\
  \L_2(z)
  =&
  \Li_2\left(e^{2\pi\i z}\right)
  -2\pi^2\floor{\Re{z}}\left(\floor{\Re{z}}-2z+1\right),
\end{align*}
where $\floor{x}$ is the greatest integer that does not exceed $x$.
\end{lem}
\begin{proof}
For $\L_1(z)$, we have
\begin{align*}
  &\log\left(1-e^{-2\pi\i z}\right)
  \\
  =&
  \log\left[\left(1-e^{2\pi\i z}\right)e^{-2\pi\i z+\pi\i}\right]
  \\
  =&
  \log\left(1-e^{2\pi\i z}\right)
  -2\pi\i z+\pi\i+2\floor{\Re{z}}\pi\i.
\end{align*}
The last equality follows because
\begin{itemize}
\item
if $0<\Re{z}-\floor{\Re{z}}<1/2$, then $-\pi<\arg\left(1-e^{2\pi\i z}\right)<0$,
\item
if $1/2\le\Re{z}-\floor{\Re{z}}<1$, then $0\le\arg\left(1-e^{2\pi\i z}\right)<\pi$,
\end{itemize}
and so the imaginary part of the rightmost side is between $-\pi$ and $\pi$.
Thus we obtain $\L_1(z)=\log\left(1-e^{2\pi\i z}\right)+2\floor{\Re{z}}\pi\i$ from \eqref{eq:L_2}.
\par
For $L_2(z)$, from the well known formula
\begin{equation}\label{eq:dilog_inversion}
  \Li_2(w^{-1})
  =
  -\Li_2(w)
  -\frac{\pi^2}{6}
  -\frac{1}{2}\bigl(\log(-w)\bigr)^2,
\end{equation}
we have
\begin{align*}
  &\Li_2(e^{-2\pi\i z})
  \\
  =&
  -\Li_2(e^{2\pi\i z})
  -\frac{\pi^2}{6}
  -\frac{1}{2}\bigl(2\pi\i z-(2\pi\floor{\Re{z}}+\pi)\i)\bigr)^2
  \\
  =&
  -\Li_2(e^{2\pi\i z})
  +\pi^2\left(2z^2-2z+\frac{1}{3}\right)
  +2\pi^2\floor{\Re{z}}^2-4\pi^2\floor{\Re{z}}z+2\pi^2\floor{\Re{z}}
\end{align*}
and the result follows.
\end{proof}
As a corollary, we have
\begin{cor}\label{cor:L1_L2_Im_negative}
If $\Im{z}<0$, then we have $\L_1(z+1)-\L_1(z)=2\pi\i$ and $\L_2(z+1)-\L_2(z)=4\pi^2z$.
\end{cor}
\begin{lem}\label{lem:der_L1_L2}
The derivatives of $\L_i(z)$ \rm{(}$i=1,2$\rm{)} are given as follows:
\begin{align}
  \frac{d}{d\,z}\L_2(z)
  &=
  -2\pi\i\L_1(z),\label{eq:der_L2}
  \\
  \frac{d}{d\,z}\L_1(z)
  &=
  -\L_0(z)
  =
  \frac{2\pi\i}{1-e^{-2\pi\i z}}.\label{eq:der_L1}
\end{align}
\end{lem}
\begin{proof}
The first equality follows from the well-known equality $\frac{d}{d\,w}\Li_2(w)=-\frac{\log(1-w)}{w}$.
The second one also follows easily.
\end{proof}
Now we will show three identities expressing the difference $T_N(z+a)-T_N(z)$ in terms of $\L_1$.
\par
The first one is as follows:
\begin{lem}\label{lem:TN_1}
If $|\Re{z}|<\frac{\Re{\gamma}}{2N}$, then we have
\begin{equation}\label{eq:TN_1}
  T_{N}(z)-T_{N}(z+1)
  =
  \L_1\left(\frac{N}{\gamma}z+\frac{1}{2}\right).
\end{equation}
\end{lem}
\begin{rem}
Note that since $-\frac{\Re\gamma}{2N}<\Re{z}<\frac{\Re\gamma}{2N}$ and $1-\frac{\Re\gamma}{2N}<\Re(z+1)<1+\frac{\Re\gamma}{2N}$, both $z$ and $z+1$ are in the domain of $T_N$.
\par
We will check that $Nz/\gamma+1/2$ is in $\C\setminus\bigl((-\infty,0]\cup[1,\infty)\bigr)$, the domain of $\L_1$.
\par
If not, then $\frac{N}{\gamma}z+\frac{1}{2}=s$ for $s\le0$ or $s\ge1$.
Thus, we have $z=\frac{\gamma}{N}s$ with $|s|\ge\frac{1}{2}$, which implies $|\Re{z}|\ge\frac{\Re\gamma}{2N}$, a contradiction.
\end{rem}
\begin{proof}
By definition, we have
\begin{align*}
  T_{N}(z)-T_{N}(z+1)
  =&
  \frac{1}{4}\int_{\Rpath}\frac{e^{(2z-1)x}-e^{(2z+1)x}}{x\sinh{x}\sinh(\gamma x/N)}\,dx
  \\
  =&
  -\frac{1}{2}\int_{\Rpath}\frac{e^{2zx}}{x\sinh(\gamma x/N)}\,dx
  \\
  &(\text{$y:=\gamma x/N$})
  \\
  =&
  -\frac{1}{2}\int_{\Rpath'}\frac{e^{2Nzy/\gamma}}{y\sinh{y}}\,dy
  =
  -\frac{1}{2}\int_{\Rpath}\frac{e^{2Nzy/\gamma}}{y\sinh{y}}\,dy
  =
  \L_1\left(\frac{N}{\gamma}z+\frac{1}{2}\right),
\end{align*}
where $\Rpath'$ is obtained from $\Rpath$ by multiplying $\gamma/N$.
The last equality follows since there are no poles of $\frac{1}{y\sinh{y}}$, that is, integer multiples of $\pi\i$ between $\Rpath$ and $\Rpath'$.
\end{proof}
Here is the second one:
\begin{lem}\label{lem:TN_gamma}
If $0<\Re{z}<1$, then we have
\begin{equation}\label{eq:TN_gamma}
  T_{N}\left(z-\frac{\gamma}{2N}\right)-T_{N}\left(z+\frac{\gamma}{2N}\right)
  =
  \L_1(z).
\end{equation}
\end{lem}
\begin{proof}
From the definition, we have
\begin{align*}
  T_{N}\left(z-\frac{\gamma}{2N}\right)
  -
  T_{N}\left(z+\frac{\gamma}{2N}\right)
  =&
  \frac{1}{4}
  \int_{\Rpath}\frac{e^{(2z-\gamma/N-1)x}-e^{(2z+\gamma/N-1)x}}{x\sinh{x}\sinh(\gamma x/N)}\,dx
  \\
  =&
  -\frac{1}{2}\int_{\Rpath}\frac{e^{(2z-1)x}}{x\sinh{x}}\,dx
  =
  \L_1(z).
\end{align*}
\end{proof}
The third one is a little tricky.
\begin{lem}\label{lem:TN_gamma_1}
If $|\Re{z}|<\frac{\Re\gamma}{N}$, then we have
\begin{multline}\label{eq:TN_gamma_1}
  T_{N}\left(z+1-\frac{\gamma}{2N}\right)-T_{N}\left(z+\frac{\gamma}{2N}\right)
  \\
  =
  \begin{cases}
    \L_1(z)-\L_1\left(\frac{N}{\gamma}z\right)
    &\quad\text{if $\Re{z}\ge0$ and $z\ne0$},
    \\
    \L_1(z+1)-\L_1\left(\frac{N}{\gamma}z+1\right)
    &\quad\text{if $\Re{z}<0$},
    \\
    \log\left(\frac{\gamma}{N}\right)
    &\quad\text{if $z=0$}.
  \end{cases}
\end{multline}
\end{lem}
\begin{rem}
If $|\Re{z}|<\frac{\Re\gamma}{N}$, then $1-\frac{3\Re\gamma}{2N}<\Re\bigl(z+1-\frac{\gamma}{2N}\bigr)<1+\frac{\Re\gamma}{2N}$ and $-\frac{\Re\gamma}{2N}<\Re\bigl(z+\frac{\gamma}{2N}\bigr)<\frac{3\Re\gamma}{2N}$, and so both $z+1-\frac{\gamma}{2N}$ and $z+\frac{\gamma}{2N}$ are in the domain of $T_N$.
\par
We will check that the arguments in the right-hand side are in $\C\setminus\bigl((-\infty,0]\cup[1,\infty)\bigr)$, the domain of $\L_1$.
\begin{itemize}
\item
The case where $0\le\Re{z}$ and $z\ne0$.
Since $\Re{z}<\Re\gamma/N$, $z$ is in the domain of $\L_1$ if $N$ is sufficiently large.
Suppose for a contradiction that $\frac{N}{\gamma}z$ is not in the domain of $\L_1$.
Then we see that $\frac{N}{\gamma}z\in(-\infty,0]\cup[1,\infty)$ and so $\frac{N}{\gamma}z=s$ for $s\ge1$ or $s\le0$.
If $s\le0$, then $\Re{z}=s\Re\gamma/N\le0$ and so $z=s=0$, which is a contradiction.
If $s\ge1$, then $\Re{z}\ge s\Re\gamma/N\ge\Re\gamma/N$, which is also a contradiction.
\item
The case where $\Re{z}<0$.
Then $z+1$ is in the domain of $\L_1$ because $1-\Re\gamma/N<\Re(z+1)<1$.
Suppose for a contradiction that $Nz/\gamma+1$ is not in the domain of $\L_1$.
Then $\frac{N}{\gamma}z+1=s$ for $s\le0$ or $s\ge1$.
Thus $\Re{z}=\Re\bigl((s-1)\gamma/N\bigr)$ and so we have $\Re{z}\le-\Re\gamma/N$, which is impossible.
\end{itemize}
Thus, we conclude that the arguments in the right-hand side are in the domain of $\L_1$.
\end{rem}
\begin{proof}
We first assume that $\Re{z}>0$.
Then from Lemmas~\ref{lem:TN_gamma} and \ref{lem:TN_1} we have
\begin{align*}
  T_{N}\left(z-\frac{\gamma}{2N}\right)-T_{N}\left(z+\frac{\gamma}{2N}\right)
  &=
  \L_1(z),
  \\
  T_{N}\left(z+1-\frac{\gamma}{2N}\right)-T_{N}\left(z-\frac{\gamma}{2N}\right)
  &=
  -\L_1\left(\frac{N}{\gamma}z\right)
\end{align*}
and the equality follows.
Note that $-\frac{\Re\gamma}{2N}<\Re\bigl(z-\frac{\gamma}{2N}\bigr)<\frac{\Re\gamma}{2N}$ and so we can apply Lemma~\ref{lem:TN_1} to the second equality.
Similarly, if $\Re{z}<0$, we have
\begin{align*}
  T_{N}\left(z+1-\frac{\gamma}{2N}\right)-T_{N}\left(z+1+\frac{\gamma}{2N}\right)
  &=
  \L_1(z+1),
  \\
  T_{N}\left(z+1+\frac{\gamma}{2N}\right)-T_{N}\left(z+\frac{\gamma}{2N}\right)
  &=
  -\L_1\left(\frac{N}{\gamma}z+1\right)
\end{align*}
and the equality also holds.
\par
When $\Re{z}=0$, put $z:=y\i$ for $y\in\R\setminus\{0\}$ and consider the limit
\begin{equation*}
  \lim_{\varepsilon\to0}
  \left(
    T_{N}\left(y\i+1-\frac{\gamma}{2N}+\varepsilon\right)
    -
    T_{N}\left(y\i+\frac{\gamma}{2N}+\varepsilon\right)
  \right).
\end{equation*}
Since $T_{N}$ is a holomorphic function in $-\frac{\Re\gamma}{2N}<\Re{z}<1+\frac{\Re\gamma}{2N}$, the limit above coincides with the left-hand side of \eqref{eq:TN_1}.
From the result above, considering the limit from the right, we have
\begin{align*}
  &T_{N}\left(y\i+1-\frac{\gamma}{2N}\right)
  -
  T_{N}\left(y\i+\frac{\gamma}{2N}\right)
  \\
  =&
  \lim_{\varepsilon\searrow0}
  \left(
    T_{N}\left(y\i+1-\frac{\gamma}{2N}+\varepsilon\right)
    -
    T_{N}\left(y\i+\frac{\gamma}{2N}+\varepsilon\right)
  \right)
  \\
  =&
  \lim_{\varepsilon\searrow0}
  \left(
    \L_1(y\i+\varepsilon)-\L_1\left(\frac{N}{\gamma}(y\i+\varepsilon)\right)
  \right)
  \\
  =&
  \L_1(y\i)-\L_1\left(\frac{N}{\gamma}(y\i)\right)
\end{align*}
if $y\ne0$, because we extend $\L_1(z)$ to $\C\setminus\bigl((-\infty,0]\cup[1,\infty)\bigr)$.
Let us confirm that the limit from the left gives the same answer.
We have
\begin{align*}
  &\lim_{\varepsilon\nearrow0}
  \left(
    \L_1(y\i+\varepsilon+1)-\L_1\left(\frac{N}{\gamma}(y\i+\varepsilon)+1\right)
  \right)
  \\
  =&
  \L_1(y\i+1)-\L_1\left(\frac{N}{\gamma}y\i+1\right),
\end{align*}
which coincides with $\L_1(y\i)-\L_1\left(\frac{N}{\gamma}(y\i)\right)$ if $y\ne0$ from Lemma~\ref{lem:L2_Li2}, noting that $\Im\bigl(\frac{N}{\gamma}y\i\bigr)=\frac{Ny\Re\gamma}{|\gamma|^2}$ has the same sign as $y$.
\par
Now, we consider the case where $z=0$.
Since $\Im\gamma<0$, we have $\Im\frac{N}{\gamma}\varepsilon>0$ for $\varepsilon>0$.
Thus, we have
\begin{equation}\label{eq:TN_gamma_1_proof}
\begin{split}
  &\lim_{\varepsilon\searrow0}
  \left(
    \L_1(\varepsilon)-\L_1\left(\frac{N}{\gamma}\varepsilon\right)
  \right)
  \\
  =&
  \lim_{\varepsilon\searrow0}
  \left(
    \log\left(1-e^{2\pi\i\varepsilon}\right)
    -
    \log\left(1-e^{2N\varepsilon\pi\i/\gamma}\right)
  \right).
\end{split}
\end{equation}
Since we have $\lim_{\varepsilon\searrow0}\arg\left(1-e^{2\pi\i\varepsilon}\right)=-\pi/2$, and $-\pi<\arg\left(1-e^{2N\varepsilon\pi\i/\gamma}\right)<0$ because $\Im\left(1-e^{2N\varepsilon\pi\i/\gamma}\right)<0$, \eqref{eq:TN_gamma_1_proof} turns out to be
\begin{equation*}
  \lim_{\varepsilon\searrow0}
  \log\frac{1-e^{2\pi\i\varepsilon}}{1-e^{2N\varepsilon\pi\i/\gamma}}
  =
  \log\left(\frac{\gamma}{N}\right)
\end{equation*}
by l'H{\^o}pital's rule,
\par
Just for safety, we will check the other limit $\lim_{\varepsilon\nearrow0}\left(\L_1(\varepsilon+1)-\L_1\left(\frac{N}{\gamma}\varepsilon+1\right)\right)$.
Since $\Im(N\varepsilon/\gamma+1)<0$ when $\varepsilon<0$, from Lemma~\ref{lem:L2_Li2}
\begin{align*}
  &\lim_{\varepsilon\nearrow0}
  \left(
    \L_1(\varepsilon+1)-\L_1\left(\frac{N}{\gamma}\varepsilon+1\right)
  \right)
  \\
  =&
  \lim_{\varepsilon\nearrow0}
  \left(
    \log\left(1-e^{2\pi\i\varepsilon}\right)
    -
    \log\left(1-e^{2N\varepsilon\pi\i/\gamma}\right)
    -
    2\pi\i\left\lfloor\Re\left(\frac{N}{\gamma}\varepsilon\right)+1\right\rfloor
  \right)
  \\
  =&
  \lim_{\varepsilon\nearrow0}
  \log\frac{1-e^{2\pi\i\varepsilon}}{1-e^{2N\varepsilon\pi\i/\gamma}}
  =
  \log\left(\frac{\gamma}{N}\right),
\end{align*}
where the second equality follows since $\lim_{\varepsilon\nearrow0}\arg\left(1-e^{2\pi\i\varepsilon}\right)=\pi/2$, $0<\arg\left(1-e^{2N\varepsilon\pi\i/\gamma}\right)<\pi$ because $\Im\left(1-e^{2N\varepsilon\pi\i/\gamma}\right)>0$, and $\lim_{\varepsilon\nearrow0}\left\lfloor\Re(N\varepsilon/\gamma)+1\right\rfloor=0$.
\par
This completes the proof.
\end{proof}
We use Lemma~\ref{lem:TN_1} to extend the function $T_{N}$ to the region
\begin{equation}\label{eq:ell01}
  \{z\in\C\mid-1<\Re{z}<2\}\setminus\left(\ell_0^{+}\cup\ell_1^{-}\right),
\end{equation}
where
\begin{align*}
  \ell_0^{+}
  &:=
  \left\{
    z\in\C\Biggm|
    \text{$z=s\gamma$ with $-\frac{1}{\Re\gamma}<s\le-\frac{1}{2N}$}
  \right\},
  \\
  \ell_1^{-}
  &:=
  \left\{
    z\in\C\Biggm|
    \text{$z=1+s\gamma$ with $\frac{1}{2N}\le s<\frac{1}{\Re\gamma}$}
  \right\}.
\end{align*}
See Figure~\ref{fig:ell01}.
\begin{figure}[h]
\pic{0.3}{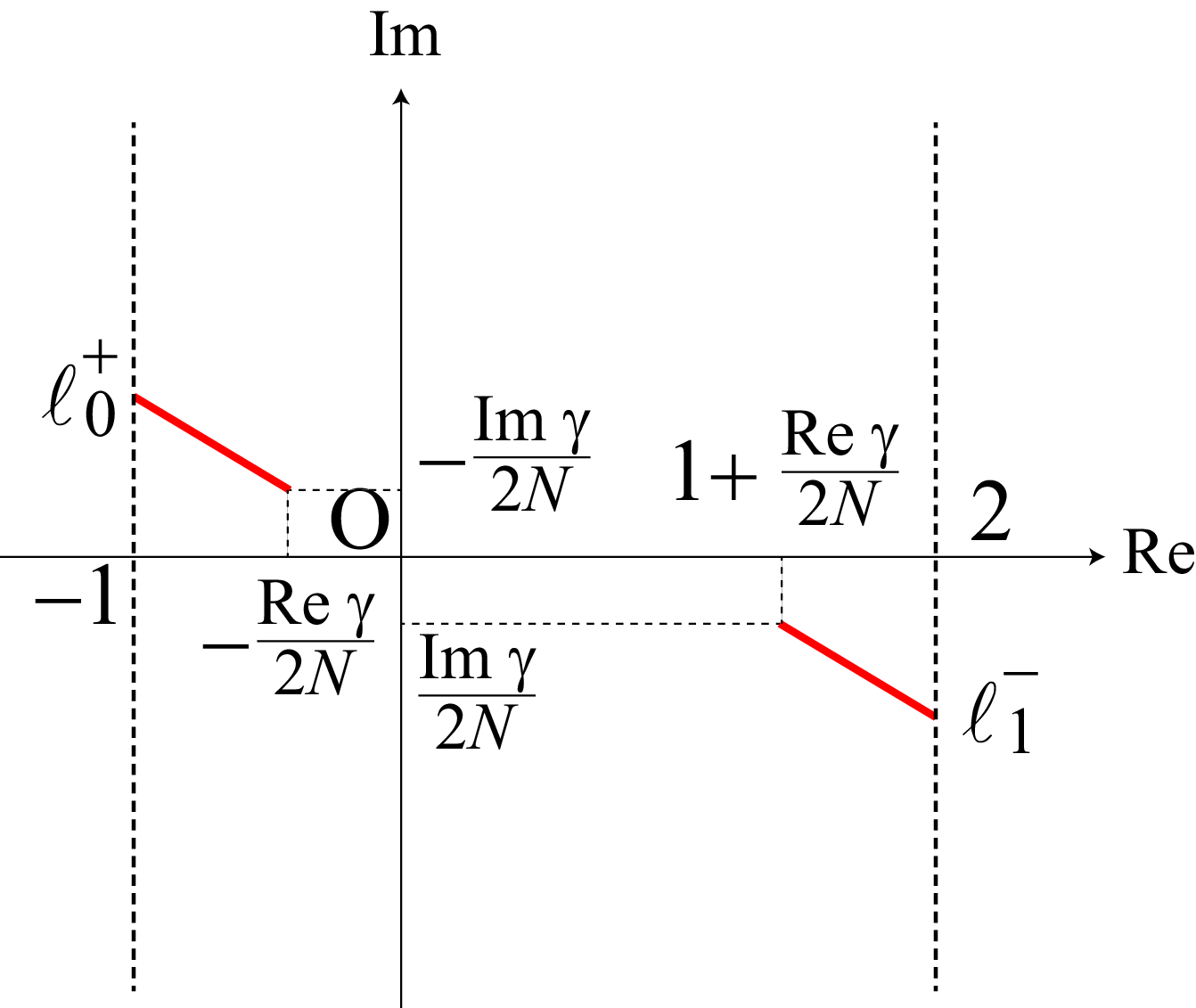}
\caption{The region \eqref{eq:ell01} is between the two thick dotted lines minus the two red lines $\ell_{0}^{+}$ and $\ell_{1}^{-}$.}
\label{fig:ell01}
\end{figure}
Note that $T_N$ is already defined for $z$ with $-\frac{\gamma}{2N}<\Re{z}<1+\frac{\gamma}{2N}$.
\par
If $-1<\Re{z}\le-\frac{\Re\gamma}{2N}$, then we use \eqref{eq:TN_1} to define
\begin{equation}\label{eq:TN_negative}
  T_{N}(z)
  :=
  T_{N}(z+1)
  +
  \L_1\left(\frac{N}{\gamma}z+\frac{1}{2}\right),
\end{equation}
noting that $z+1$ is in the domain of $T_{N}$.
For the argument of $\L_1$, see Remark~\ref{rem:domain_TN} below.
Similarly, if $1+\frac{\Re\gamma}{2N}\le\Re{z}<2$, we define
\begin{equation}\label{eq:TN_positive}
  T_{N}(z)
  :=
  T_{N}(z-1)
  -
  \L_1\left(\frac{N}{\gamma}(z-1)+\frac{1}{2}\right),
\end{equation}
noting that $z-1$ is in the domain of $T_{N}$.
For the argument of $\L_1$, see Remark~\ref{rem:domain_TN} below.
\begin{rem}\label{rem:domain_TN}
Recall that $\L_1(z)$ is defined except for $z\in(-\infty,0]\cup[1,\infty)$.
Therefore $Nz/\gamma+1/2$ and $N(z-1)/\gamma+1/2$ are in the domain of $\L_1$ unless
\begin{itemize}
\item $-1<\Re{z}\le-\frac{\Re\gamma}{2N}$ and $\frac{N}{\gamma}z+\frac{1}{2}=s$ for $s\in(-\infty,0]\cup[1,\infty)$, or
\item $1+\frac{\Re\gamma}{2N}\le\Re{z}<2$ and $\frac{N}{\gamma}(z-1)+\frac{1}{2}=t$ for $t\in(-\infty,0)\cup(1,\infty)$.
\end{itemize}
This is equivalent to
\begin{itemize}
\item
$-1<\Re{z}\le-\frac{\Re\gamma}{2N}$ and $z=s'\gamma$ with $|s'|\ge\frac{1}{2N}$, or
\item
$1+\frac{\Re\gamma}{2N}\le\Re{z}<2$ and $z=1+t'\gamma$ with $|t'|\ge\frac{1}{2N}$.
\end{itemize}
Since $\Re\gamma>0$, the condition above turns out to be $z\in\ell_{0}^{+}$ or $z\in\ell_{1}^{-}$.
\end{rem}
We will also use $T_{N}(z)$ to denote the function extended by using \eqref{eq:TN_negative} and \eqref{eq:TN_positive}.
Then we have
\begin{lem}\label{lem:TN_ell01}
The function $T_{N}(z)$ extended as above also satisfies \eqref{eq:TN_1} for any $z$ in the region
\begin{equation}\label{eq:ell00}
  \{z\in\C\mid-1<\Re{z}<1\}\setminus\left(\ell_0^{+}\cup\ell_0^{-}\right)
\end{equation}
with $\ell_0^{-}:=\bigl\{z\in\C\bigm|\text{$z=s\gamma$ with $\frac{1}{2N}\le s<\frac{1}{\Re\gamma}$}\bigr\}$.
See Figure~\ref{fig:ell00}.
\begin{figure}[h]
\pic{0.3}{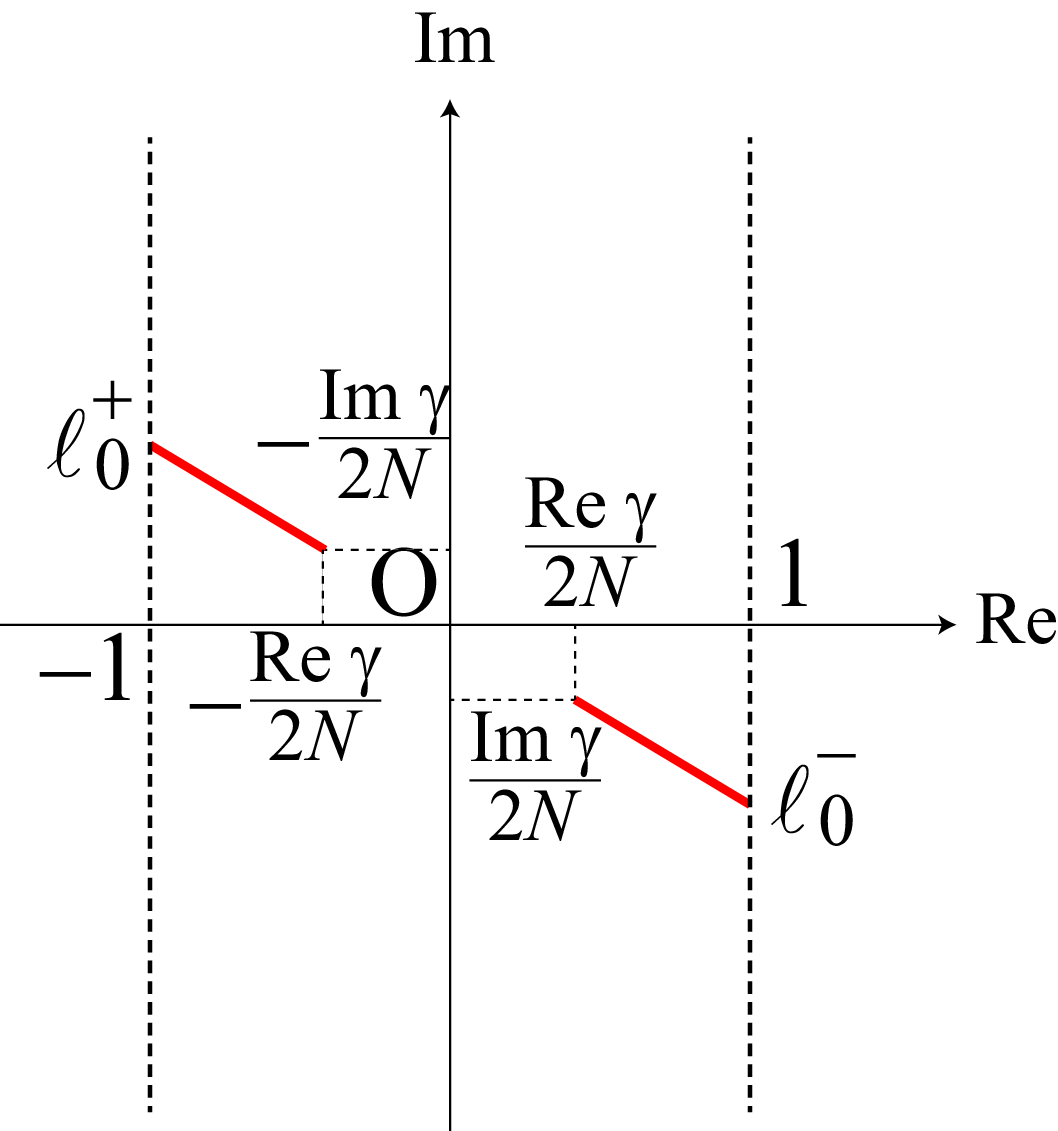}
\caption{The region \eqref{eq:ell00} is between the two thick dotted lines minus the two red lines $\ell_0^{+}$ and $\ell_0^{-}$.}
\label{fig:ell00}
\end{figure}
\end{lem}
\begin{rem}
As in Remark~\ref{rem:domain_TN}, $Nz/\gamma+1/2$ is in the domain of $\L_1$ unless $z\in\ell_0^{+}\cup\ell_0^{-}$.
\end{rem}
\begin{proof}
If $-\frac{\Re\gamma}{2N}<\Re{z}<\frac{\Re\gamma}{2N}$, then \eqref{eq:TN_1} is proved in Lemma~\ref{lem:TN_1}.
If $-1<\Re{z}\le-\frac{\Re\gamma}{2N}$ ($\frac{\Re\gamma}{2N}\le\Re{z}<1$, respectively), then we define $T_{N}$ by \eqref{eq:TN_negative} (\eqref{eq:TN_positive}, respectively) so that \eqref{eq:TN_1} holds.
\end{proof}
\begin{lem}
The function $T_{N}(z)$ defined as above is holomorphic in the region \eqref{eq:ell01}.
\end{lem}
\begin{proof}
From the original definition \eqref{eq:def_TN}, we see that $T_{N}(z)$ is holomorphic in the region $\left\{z\in\C\mid-\frac{\Re\gamma}{2N}<\Re{z}<1+\frac{\Re\gamma}{2N}\right\}$.
Therefore from the definition using \eqref{eq:TN_negative} and \eqref{eq:TN_positive}, $T_{N}(z)$ is holomorphic in the disjoint strips
\begin{equation*}
  \left\{z\in\C\mid-1<\Re{z}<-\frac{\Re\gamma}{2N}\right\}
  \sqcup
  \left\{z\in\C\mid1+\frac{\Re\gamma}{2N}<\Re{z}<2\right\}.
\end{equation*}
\par
So, we need to confirm that $T_{N}(z)$ is holomorphic for $z$ with $\Re{z}=-\frac{\Re\gamma}{2N}$ or $1+\frac{\Re\gamma}{2N}$.
\par
Let $B$ be an open disk centered at $z$ ($\Re{z}=-\frac{\Re\gamma}{2N}$) with radius less than $\frac{\Re\gamma}{2N}$.
Then for $w\in B$ with $\Re{w}\le-\frac{\Re\gamma}{2N}$, we have
\begin{equation*}
  T_{N}(w)
  =
  T_{N}(w+1)
  +
  \L_1\left(\frac{N}{\gamma}w+\frac{1}{2}\right)
\end{equation*}
from \eqref{eq:TN_negative}.
On the other hand, for $w\in B$ with $\Re{w}>-\frac{\Re\gamma}{2N}$, $T_{N}(w)$ is defined by using \eqref{eq:def_TN}.
However, from Lemma~\ref{lem:TN_1} this coincides with $T_{N}(w+1)+\L_1\left(\frac{N}{\gamma}w+\frac{1}{2}\right)$.
Therefore $T_{N}$ is holomorphic in this case.
\par
Similarly, we can prove the holomorphicity of $T_{N}$ for the other case.
\end{proof}
\par
Let $\Omega$ be the region defined as
\begin{equation}\label{eq:Omega}
  \Omega
  :=
  \left\{z\in\C\Bigm|-1+\frac{\Re\gamma}{2N}<\Re{z}<2-\frac{\Re\gamma}{2N}\right\}
  \setminus
  \bigl(\Delta_{0}^{+}\cup\Delta_{1}^{+}\bigr)
\end{equation}
where we put
\begin{align*}
  \Delta_{0}^{+}
  &:=
  \left\{
    z\in\C\Bigm|
    \text{$-1+\frac{\Re\gamma}{2N}<\Re{z}\le0$,
    $\Im{z}\ge0$, and $\Im\left(\frac{z}{\gamma}\right)\le0$}
  \right\},
  \\
  \Delta_{1}^{-}
  &:=
  \left\{
    z\in\C\Bigm|
    \text{$1\le\Re{z}<2-\frac{\Re\gamma}{2N}$,
    $\Im{z}\le0$, and $\Im\left(\frac{z-1}{\gamma}\right)\ge0$}
  \right\}.
\end{align*}
See Figure~\ref{fig:Omega}.
\begin{figure}[h]
\pic{0.3}{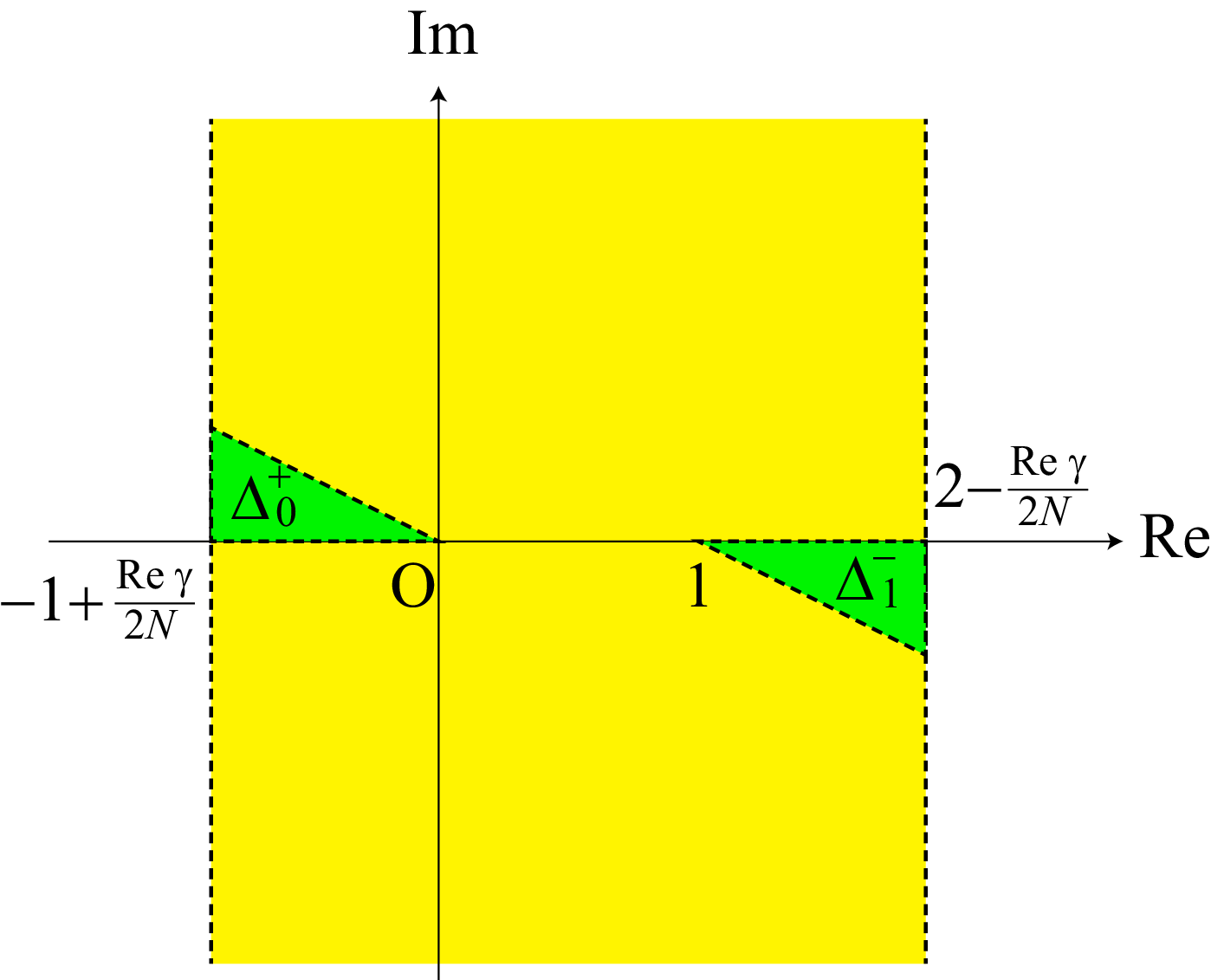}
\caption{The yellow region is $\Omega$.
The green triangles are $\Delta_{0}^{+}$ and $\Delta_{1}^{-}$.}
\label{fig:Omega}
\end{figure}
Note that $\Omega$ is contained in the region \eqref{eq:ell01} because $\ell_{0}^{+}\cap\{z\in\C\mid-1+\Re\gamma/(2N)<\Re{z}<2-\Re\gamma/(2N)\}$ ($\ell_{1}^{-}\cap\{z\in\C\mid-1+\Re\gamma/(2N)<\Re{z}<2-\Re\gamma/(2N)\}$, respectively) is on the upper side of $\Delta_{0}^{+}$ (the lower side of $\Delta_{1}^{-}$, respectively).
\begin{lem}\label{lem:TN_gamma_Omega}
The function $T_{N}(z)$ extended by using \eqref{eq:TN_1} satisfies \eqref{eq:TN_gamma} for $z\in\Omega$.
\end{lem}
\begin{rem}\label{rem:Omega}
The left-hand side of \eqref{eq:TN_gamma} is defined for $z$ such that $z\pm\frac{\gamma}{2N}$ is in the region \eqref{eq:ell01}, that is, $z\pm\frac{\gamma}{2N}\not\in\ell_0^{+}\cup\ell_1^{-}$.
This is equivalent to $z$ is not on the two rays $\{s\gamma\in\C\mid s\le0\}\cup\{1+s\gamma\in\C\mid s\ge0\}$.
Note that the ray $\{s\gamma\in\C\mid s\le0\}$ includes the upper edge of $\Delta_{0}^{+}$, and that the ray $\{1+s\gamma\in\C\mid s\ge0\}$ includes the lower edge of $\Delta_{1}^{-}$.
The right-hand side of \eqref{eq:TN_gamma} is defined unless $z\in(-\infty,0]\cup[1,\infty)$.
\end{rem}
\begin{proof}
We need to prove \eqref{eq:TN_gamma} for $z$ with $-1+\frac{\Re\gamma}{2N}<\Re{z}\le0$ or $1\le\Re{z}<2-\frac{\Re\gamma}{2N}$.
\par
If $-\frac{\Re\gamma}{2N}<\Re{z}<0$, from \eqref{eq:TN_gamma_1}, we have
\begin{align*}
  &T_{N}\left(z-\frac{\gamma}{2N}\right)
  -
  T_{N}\left(z+\frac{\gamma}{2N}\right)
  \\
  =&
  T_{N}\left(z-\frac{\gamma}{2N}\right)
  -
  T_{N}\left(z+1-\frac{\gamma}{2N}\right)
  +
  \L_1(z+1)
  -
  \L_1\left(\frac{N}{\gamma}z+1\right)
  \\
  =&
  \L_1\left(\frac{N}{\gamma}\left(z-\frac{\gamma}{2N}\right)+\frac{1}{2}\right)
  +
  \L_1(z+1)
  -
  \L_1\left(\frac{N}{\gamma}z+1\right)
  \\
  =&
  \L_1\left(\frac{N}{\gamma}z\right)
  +
  \L_1(z+1)
  -
  \L_1\left(\frac{N}{\gamma}z+1\right),
\end{align*}
where we use Lemma~\ref{lem:TN_ell01} for $z-\frac{\gamma}{2N}$ at the second equality.
If $\Im{z}\ge0$, then $\Im(z/\gamma)>0$ from \eqref{eq:Omega}.
So we have $\L_1(z+1)=\L_1(z)$ and $\L_1(Nz/\gamma+1)=\L_1(Nz/\gamma)$ from \eqref{eq:L1_log}, which implies \eqref{eq:TN_gamma}.
If $\Im{z}<0$, then we have $\Im(Nz/\gamma+1)=\frac{N}{|\gamma|^2}(\Re\gamma\Im{z}-\Im\gamma\Re{z})<0$.
So we have $\L_1(z+1)=\L_1(z)+2\pi\i$ and $\L_1(Nz/\gamma+1)=\L_1(Nz/\gamma)+2\pi\i$ from Corollary~\ref{cor:L1_L2_Im_negative}, proving \eqref{eq:TN_gamma}.
\par
If $\Re{z}=0$, then noting that $0$ is not included in $\Omega$, similarly we have
\begin{align*}
  &T_{N}\left(z-\frac{\gamma}{2N}\right)
  -
  T_{N}\left(z+\frac{\gamma}{2N}\right)
  \\
  =&
  T_{N}\left(z-\frac{\gamma}{2N}\right)
  -
  T_{N}\left(z+1-\frac{\gamma}{2N}\right)
  +
  \L_1(z)
  -
  \L_1\left(\frac{N}{\gamma}z\right)
  \\
  =&
  \L_1\left(\frac{N}{\gamma}\left(z-\frac{\gamma}{2N}\right)+\frac{1}{2}\right)
  +
  \L_1(z)
  -
  \L_1\left(\frac{N}{\gamma}z\right)
  \\
  =&
  \L_1(z).
\end{align*}
\par
If $\Re{z}=-\frac{\Re\gamma}{2N}$, then we have $\Re\bigl(z+\frac{\gamma}{2N}\bigr)=0$ and $-1<\Re\bigl(z-\frac{\gamma}{2N}\bigr)=-\frac{\Re\gamma}{N}<-\frac{\Re\gamma}{2N}$.
Therefore from \eqref{eq:TN_negative} we have
\begin{equation}\label{eq:Rez_gamma/2N}
\begin{split}
  &T_{N}\left(z-\frac{\gamma}{2N}\right)
  -
  T_{N}\left(z+\frac{\gamma}{2N}\right)
  \\
  =&
  T_{N}\left(z-\frac{\gamma}{2N}+1\right)
  -
  T_{N}\left(z+\frac{\gamma}{2N}\right)
  +
  \L_1\left(\frac{N}{\gamma}\left(z-\frac{\gamma}{2N}\right)+\frac{1}{2}\right)
  \\
  =&
  T_{N}\left(z-\frac{\gamma}{2N}+1\right)
  -
  T_{N}\left(z+\frac{\gamma}{2N}\right)
  +
  \L_1\left(\frac{N}{\gamma}z\right)
  \\
  =&
  \L_1(z+1)-\L_1\left(\frac{N}{\gamma}z+1\right)+\L_1\left(\frac{N}{\gamma}z\right),
\end{split}
\end{equation}
where the last equality follows from Lemma~\ref{lem:TN_gamma_1} since $\Re{z}<0$.
If $\Im{z}\ge0$, then $\Im(z/\gamma)>0$, and so \eqref{eq:Rez_gamma/2N} turns out to be $\L_1(z)$.
If $\Im{z}<0$, then $\Im(z/\gamma)<0$.
Therefore \eqref{eq:Rez_gamma/2N} equals
\begin{equation*}
  \log(1-e^{2\pi\i z})-2\pi\i
  =
  \L_1(z)
\end{equation*}
from Lemma~\ref{lem:L2_Li2} and Corollary~\ref{cor:L1_L2_Im_negative}.
\par
We consider the case where $-1+\frac{\Re\gamma}{2N}<\Re{z}<-\frac{\Re\gamma}{2N}$.
Note that $-1<\Re\bigl(z\pm\frac{\gamma}{2N}\bigr)<0$.
Therefore from \eqref{eq:TN_negative} we have
\begin{equation*}
\begin{split}
  &T_{N}\left(z-\frac{\gamma}{2N}\right)
  -
  T_{N}\left(z+\frac{\gamma}{2N}\right)
  \\
  =&
  T_{N}\left(z-\frac{\gamma}{2N}+1\right)
  -
  T_{N}\left(z+\frac{\gamma}{2N}+1\right)
  \\
  &+
  \L_1\left(\frac{N}{\gamma}\left(z-\frac{\gamma}{2N}\right)+\frac{1}{2}\right)
  -
  \L_1\left(\frac{N}{\gamma}\left(z+\frac{\gamma}{2N}\right)+\frac{1}{2}\right)
  \\
  =&
  \L_1(z+1)
  +
  \L_1\left(\frac{N}{\gamma}z\right)
  -
  \L_1\left(\frac{N}{\gamma}z+1\right),
\end{split}
\end{equation*}
where we use \eqref{eq:TN_gamma} because $0<\Re(z+1)<1$.
By the same reason as above, this equals $\L_1(z)$.
\par
If $1\le\Re{z}<1+\frac{\Re\gamma}{2N}$, then from \eqref{eq:TN_positive}, we have
\begin{align*}
  &T_{N}\left(z-\frac{\gamma}{2N}\right)
  -
  T_{N}\left(z+\frac{\gamma}{2N}\right)
  \\
  =&
  T_{N}\left(z-\frac{\gamma}{2N}\right)
  -
  T_{N}\left(z+\frac{\gamma}{2N}-1\right)
  +
  \L_1\left(\frac{N}{\gamma}\left(z+\frac{\gamma}{2N}-1\right)+\frac{1}{2}\right)
  \\
  =&
  \L_1(z-1)-\L_1\left(\frac{N}{\gamma}(z-1)\right)
  +
  \L_1\left(\frac{N}{\gamma}(z-1)+1\right),
\end{align*}
where we use \eqref{eq:TN_gamma_1} for $z-1$ at the second identity, noting that $1\not\in\Omega$.
If $\Im{z}\ge0$, then $\Im\bigl((z-1)/\gamma\bigr)=\frac{1}{|\gamma|^2}\bigl(\Im\gamma(1-\Re{z})+\Re\gamma\Im{z}\bigr)\ge0$, and so the last line equals $\L_1(z)$.
If $\Im{z}<0$, then $\Im\bigl((z-1)/\gamma\bigr)<0$ from the definition of $\Delta_{1}^{-}$, and so we have $\L_1(z-1)=\L_1(z)-2\pi\i$ and $\L_1\bigl(N(z-1)/\gamma+1\bigr)=\L_1\bigl(N(z-1)/\gamma\bigr)+2\pi\i$ from Corollary~\ref{cor:L1_L2_Im_negative}, which implies \eqref{eq:TN_gamma}.
\par
Lastly, we consider the case where $1+\frac{\Re\gamma}{2N}\le\Re{z}<2-\frac{\Re\gamma}{2N}$.
Since $1\le\Re\bigl(z\pm\frac{\gamma}{2N}\bigr)<2$, from \eqref{eq:TN_positive}, we have
\begin{equation}\label{eq:Rez_positive}
\begin{split}
  &T_{N}\left(z-\frac{\gamma}{2N}\right)
  -
  T_{N}\left(z+\frac{\gamma}{2N}\right)
  \\
  =&
  T_{N}\left(z-\frac{\gamma}{2N}-1\right)
  -
  T_{N}\left(z+\frac{\gamma}{2N}-1\right)
  \\
  &
  -
  \L_1\left(\frac{N}{\gamma}\left(z-\frac{\gamma}{2N}-1\right)+\frac{1}{2}\right)
  +
  \L_1\left(\frac{N}{\gamma}\left(z+\frac{\gamma}{2N}-1\right)+\frac{1}{2}\right)
  \\
  =&
  \L_1(z-1)
  -
  \L_1\left(\frac{N}{\gamma}(z-1)\right)
  +
  \L_1\left(\frac{N}{\gamma}(z-1)+1\right),
\end{split}
\end{equation}
where we use \eqref{eq:TN_gamma} at the last equality.
If $\Im{z}\ge0$, then we have $\Im\bigl((z-1)/\gamma\bigr)=\frac{1}{|\gamma|^2}\bigl(\Im\gamma(1-\Re{z})+\Re\gamma\Im{z}\bigr)>0$ since $\Re{z}>1$.
So \eqref{eq:Rez_positive} equals $\log\big(1-e^{2\pi\i z}\bigr)=\L_1(z)$.
If $\Im{z}<0$, then $\Im\bigl((z-1)/\gamma\bigr)<0$.
Therefore \eqref{eq:Rez_positive} becomes
\begin{equation*}
  \log\bigl(1-e^{2\pi\i z}\bigr)
  +2\pi\i
  =
  \L_1(z)
\end{equation*}
from Lemma~\ref{lem:L2_Li2}.
\par
The proof is complete.
\end{proof}
\begin{rem}
Even if $z\in\Int\Delta_{0}^{+}\cup\Int\Delta_{1}^{-}$, where $\Int$ means the interior, both sides of \eqref{eq:TN_gamma} are defined from Remark~\ref{rem:Omega}.
However, if $z\in\Int\Delta_{0}^{+}$, then from the proof above, we have
\begin{align*}
  &T_{N}\left(z-\frac{\gamma}{2N}\right)
  -
  T_{N}\left(z+\frac{\gamma}{2N}\right)
  \\
  =&
  \L_1(z+1)+\L_1\left(\frac{N}{\gamma}z\right)-\L_1\left(\frac{N}{\gamma}z+1\right)
  \\
  =&
  \L_1(z)-2\pi\i,
\end{align*}
where the second equality follows from Corollary~\ref{cor:L1_L2_Im_negative} since $\Im{z}>0$ and $\Im\bigl(\frac{N}{\gamma}\bigr)<0$.
Similarly, if $z\in\Int\Delta_{1}^{-}$, we have
\begin{align*}
  &T_{N}\left(z-\frac{\gamma}{2N}\right)
  -
  T_{N}\left(z+\frac{\gamma}{2N}\right)
  \\
  =&
  \L_1(z-1)-\L_1\left(\frac{N}{\gamma}(z-1)\right)+\L_1\left(\frac{N}{\gamma}(z-1)+1\right)
  \\
  =&
  \L_1(z)-2\pi\i
\end{align*}
since $\Im{z}<0$ and $\Im\bigl(\frac{N}{\gamma}(z-1)\bigr)>0$.
\end{rem}
\par
For a real number $0<\nu<1/2$ and a positive real number $M$, we put
\begin{equation}\label{eq:Omega_ast}
  \Omega^{\ast}_{\nu}
  :=
  \left\{z\in\C\Bigm|-1+\nu\le\Re{z}\le2-\nu, |\Im{z}|\le M\right\}
  \setminus
  \bigl(\Delta^{+}_{0,\nu}\cup\Delta^{-}_{1,\nu}\bigr)
\end{equation}
where we put
\begin{align*}
  \Delta^{+}_{0,\nu}
  &:=
  \left\{
    z\in\C\Bigm|
    \text{$-1+\nu\le\Re{z}<\nu$,
    $\Im{z}>-\nu$,  and $\Im\left(\frac{z-\nu}{\gamma}\right)<0$}
  \right\},
  \\
  \Delta^{-}_{1,\nu}
  &:=
  \left\{
    z\in\C\Bigm|
    \text{$1-\nu<\Re{z}\le2-\nu$,
    $\Im{z}<\nu$, and $\Im\left(\frac{z-1+\nu}{\gamma}\right)>0$}
  \right\}.
\end{align*}
Note that $\Omega^{\ast}_{\nu}\subset\Omega$ if $N>\frac{\Re\gamma}{2\nu}$.
Note also that $\Delta^{+}_{0,\nu}\cap\Delta^{-}_{1,\nu}=\emptyset$ since $\nu<1/2$.
See Figure~\ref{fig:Omega_ast}.
\begin{figure}[h]
\pic{0.3}{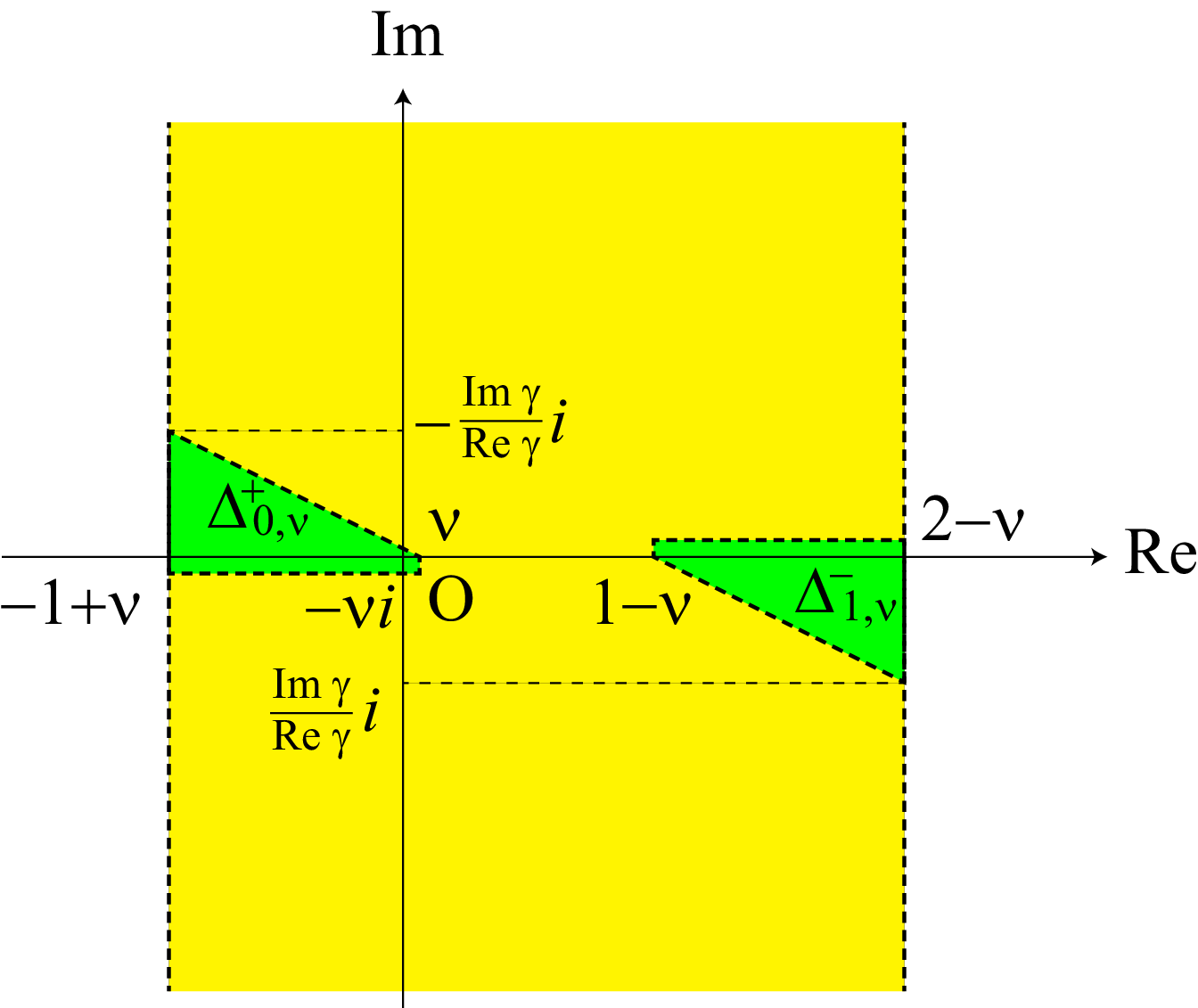}
\caption{The yellow region is $\Omega^{\ast}_{\nu}$.
The green triangles are $\Delta^{+}_{0,\nu}$ and $\Delta^{-}_{1,\nu}$.}
\label{fig:Omega_ast}
\end{figure}
\par
We can prove that $T_{N}(z)$ uniformly converges to $\frac{N}{2\pi\i\gamma}\L_2(z)$ in $\Omega^{\ast}_{\nu}$.
To do that, we prepare several lemmas.
\begin{lem}\label{lem:L2_L1_gamma}
Let $\nu$ and $M$ be positive real numbers with $0<\nu<1/2$.
Then, we have
\begin{equation*}
  \frac{N}{2\pi\i\gamma}\L_2\left(z-\frac{\gamma}{2N}\right)
  -
  \frac{N}{2\pi\i\gamma}\L_2\left(z+\frac{\gamma}{2N}\right)
  =
  \L_1(z)
  +
  O(N^{-2})
\end{equation*}
as $N\to\infty$ for $z$ in the region
\begin{equation}\label{eq:L2_L1_approx}
  \{z\in\C\mid-1+\nu\le\Re{z}\le2-\nu,|\Im{z}|\le M\}
  \setminus\left(\square^{-}_{\nu}\cup\square^{+}_{\nu}\right),
\end{equation}
where
\begin{align*}
  \square^{-}_{\nu}
  &:=
  \{z\in\C\mid-1+\nu\le\Re{z}\le\nu,|\Im{z}|\le\nu\},
  \\
  \square^{+}_{\nu}
  &:=
  \{z\in\C\mid1-\nu\le\Re{z}\le2-\nu,|\Im{z}|\le\nu\}.
\end{align*}
This means that there exists a constant $c>0$ that does not depend on $z$ such that
\begin{equation*}
  \left|
    \frac{N}{2\pi\i\gamma}\L_2\left(z-\frac{\gamma}{2N}\right)
    -
    \frac{N}{2\pi\i\gamma}\L_2\left(z+\frac{\gamma}{2N}\right)
    -
    \L_1(z)
  \right|
  <
  \frac{c}{N^2}
\end{equation*}
for sufficiently large $N$.
\end{lem}
\begin{proof}
Note that if $z$ is in the region \eqref{eq:L2_L1_approx}, then $z\pm\frac{\gamma}{2N}$ is also in the same region, assuming that  $N$ is large enough.
Note also that $\L_1$ and $\L_2$ are holomorphic there.
\par
Since $\L'_2(z)=-2\pi\i\L_1(z)$, $\L''_2(z)=\frac{4\pi^2}{1-e^{-2\pi\i z}}$ and $\L^{(3)}_2(z)=2\pi^3\i\csc^2(\pi z)$ ($\csc{x}=1/\sin{x}$ is the cosecant of $x$, as you may know) from Lemma~\ref{lem:der_L1_L2}, we have
\begin{align*}
  \L_2\left(z\pm\frac{\gamma}{2N}\right)
  =&
  \L_2(z)
  \mp
  2\pi\i\L_1(z)\times\frac{\gamma}{2N}
  +
  \frac{2\pi^2}{1-e^{-2\pi\i z}}\times\frac{\gamma^2}{4N^2}
  \\
  &\pm
  \frac{\pi^3\i}{3\sin^2(\pi z)}\times\frac{\gamma^3}{8N^3}
  +
  \sum_{j=4}^{\infty}\frac{2\pi^3\i}{j!}\frac{d^{j-3}\,\csc^2(\pi z)}{d\,z^{j-3}}
  \times\left(\pm\frac{\gamma}{2N}\right)^{j}
\end{align*}
if $N$ is sufficiently large so that $z\pm\frac{\gamma}{2N}$ is contained in the region \eqref{eq:L2_L1_approx}.
So we have
\begin{equation}\label{eq:L2_L1_gamma_csc}
\begin{split}
  &\frac{N}{2\pi\i\gamma}\L_2\left(z-\frac{\gamma}{2N}\right)
  -
  \frac{N}{2\pi\i\gamma}\L_2\left(z+\frac{\gamma}{2N}\right)
  \\
  =&
  \L_1(z)
  -
  \sum_{k=1}^{\infty}\frac{\pi^2}{(2k+1)!}\frac{d^{2k-2}\,\csc^2(\pi z)}{d\,z^{2k-2}}
  \left(\frac{\gamma}{2N}\right)^{2k}.
\end{split}
\end{equation}
\par
From Lemma~\ref{lem:der_csc} below, we have
\begin{align*}
  \sin^{2k}(\pi z)\frac{d^{2k-2}\,\csc^2(\pi z)}{d\,z^{2k-2}}
  =
  2\pi^{2k-2}\sum_{j=0}^{k-1}a_{2k-2,2j}\cos(2j\pi z)
\end{align*}
with $a_{2k-2,2j}>0$ for $j=0,1,\dots,k-1$ and $\sum_{j=0}^{k-1}a_{2k-2,2j}=(2k-1)!/2$.
Letting $L$ be the maximum of $|\cos(z)|$ in the closure of \eqref{eq:L2_L1_approx}, we have
\begin{align*}
  &\left|
    \sin^{2k}(\pi z)
    \frac{\pi^2}{(2k+1)!}\frac{d^{2k-2}\,\csc^2(\pi z)}{d\,z^{2k-2}}
    \left(\frac{\gamma}{2N}\right)^{2k}
  \right|
  \\
  \le&
  \frac{\pi^2}{(2k+1)!}\times
  2\pi^{2k-2} L
  \times k\times\frac{(2k-1)!}{2}
  \left(\frac{|\gamma|}{2N}\right)^{2k}
  =
  \frac{L}{2(2k+1)}
  \left(\frac{\pi|\gamma|}{2N}\right)^{2k}.
\end{align*}
Let $l$ be the minimum of $|\sin(\pi z)|$ in the closure of the region \eqref{eq:L2_L1_approx}.
Since the closure is compact and does not contain the poles of $\sin(\pi z)$, we conclude that $l>0$.
So we have
\begin{align*}
  &N^2
  \left|
    \sum_{k=1}^{\infty}
    \frac{\pi^2}{(2k+1)!}\frac{d^{2k-2}\,\csc^2(\pi z)}{d\,z^{2k-2}}
    \left(\frac{\gamma}{2N}\right)^{2k}
  \right|
  \\
  <&
  \frac{L\pi^2|\gamma|^2}{8l^2}
  \sum_{k=1}^{\infty}
  \frac{1}{2k+1}\left(\frac{\pi|\gamma|}{2lN}\right)^{2k-2},
\end{align*}
which converges if $N>\frac{\pi|\gamma|}{2l}$.
\par
Therefore the right hand side of \eqref{eq:L2_L1_gamma_csc} turns out to be $\L_1(z)+O(N^{-2})$, completing the proof.
\end{proof}
\begin{lem}\label{lem:der_csc}
Let $m$ be a positive integer.
The $m$-th derivative of $\csc^2(\pi z)$ can be expressed as
\begin{equation*}
  \frac{d^{m}\,\csc^2(\pi z)}{d\,z^{m}}
  =
  2(-\pi)^{m}\csc^{m+2}(\pi z)P_m(z),
\end{equation*}
where $P_m(z)$ is of the form
\begin{equation*}
  P_m(z)=\sum_{\substack{0\le j\le m\\ j\equiv m\pmod{2}}}a_{m,j}\cos(j\pi z)
\end{equation*}
with
\begin{enumerate}
\item
$a_{m,j}>0$ for $0\le j\le m$ and $j\equiv m\pmod{2}$,
\item
$\sum_{0\le j\le m,j\equiv m\pmod{2}}a_{m,j}=\frac{(m+1)!}{2}$, and
\item
$a_{m,m}=1$.
\end{enumerate}
\end{lem}
\begin{proof}
First of all, recall that $\csc'(x)=-\cos(x)\csc^2(x)$.
\par
We proceed by induction on $m$.
\par
For $m=1$, since $\frac{d}{d\,z}\csc^2(\pi z)=2\csc(\pi z)\times\bigl(-\pi\cos(\pi z)\csc^2(\pi z)\bigr)=-2\pi\csc^3(\pi z)\cos(\pi z)$, we have $P_1(z)=\cos(\pi z)$, which agrees with (i), (ii), and (iii).
\par
Suppose that the lemma is true for $m$.
We calculate the $(m+1)$-st derivative by using the inductive hypothesis for $P_m(z)$.
We have
\begin{align*}
  &
  \frac{d^{m+1}\,\csc^2(\pi z)}{d\,z^{m+1}}
  \\
  =&
  2(-\pi)^m\
  \frac{d}{d\,z}
  \left(
    \csc^{m+2}(\pi z)P_m(z)
  \right)
  \\
  =&
  2(-\pi)^{m}
  \left(
    (m+2)\csc^{m+1}(\pi z)\bigl(-\pi\cos(\pi z)\csc^2(\pi z)\bigr)P_m(z)
    +
    \csc^{m+2}(\pi z)P'_{m}(z)
  \right)
  \\
  =&
  2(-\pi)^{m}\csc^{m+3}(\pi z)
  \bigl[
    -(m+2)\pi\cos(\pi z)P_m(z)
    +
    \sin(\pi z)P'_{m}(z)
  \bigr]
  \\
  =&
  2(-\pi)^{m}\csc^{m+3}(\pi z)
  \left[
    -(m+2)\pi\cos(\pi z)
    \sum_{\substack{0\le j\le m\\ j\equiv m\pmod{2}}}a_{m,j}\cos(j\pi z)
  \right.
  \\
  &-
  \left.
    \sin(\pi z)
    \sum_{\substack{0\le j\le m\\ j\equiv m\pmod{2}}}j\pi a_{m,j}\sin(j\pi z)
  \right]
  \\
  =&
  2(-\pi)^{m+1}\csc^{m+3}(\pi z)
  \left[
    (m+2)
    \sum_{\substack{0\le j\le m\\ j\equiv m\pmod{2}}}a_{m,j}\cos(\pi z)\cos(j\pi z)
  \right.
  \\
  &+
  \left.
    \sum_{\substack{0\le j\le m\\ j\equiv m\pmod{2}}}ja_{m,j}\sin(\pi z)\sin(j\pi z)
  \right].
\end{align*}
Now, we will calculate the terms inside the square brackets.
We write $x:=\pi z$.
From the product-sum identities, we have
\begin{align*}
  \sin(x)\sin(jx)
  &=
  \frac{1}{2}\cos\bigl((j-1)x\bigr)-\frac{1}{2}\cos\bigl((j+1)x\bigr),
  \\
  \cos(x)\cos(jx)
  &=
  \frac{1}{2}\cos\bigl((j-1)x\bigr)+\frac{1}{2}\cos\bigl((j+1)x\bigr).
\end{align*}
So we have
\begin{equation}\label{eq:sin_cos}
\begin{split}
  &(m+2)
  \sum_{\substack{0\le j\le m\\ j\equiv m\pmod{2}}}a_{m,j}\cos(x)\cos(jx)
  +
  \sum_{\substack{0\le j\le m\\ j\equiv m\pmod{2}}}ja_{m,j}\sin(x)\sin(jx)
  \\
  =&
  \frac{1}{2}
  \sum_{\substack{0\le j\le m\\ j\equiv m\pmod{2}}}(m+2)a_{m,j}
  \left(\cos\bigl((j-1)x\bigr)+\cos\bigl((j+1)x\bigr)\right)
  \\
  &+
  \frac{1}{2}
  \sum_{\substack{0\le j\le m\\ j\equiv m\pmod{2}}}ja_{m,j}
  \left(\cos\bigl((j-1)x\bigr)-\cos\bigl((j+1)x\bigr)\right)
  \\
  =&
  \frac{1}{2}
  \sum_{\substack{0\le j\le m\\ j\equiv m\pmod{2}}}
  \bigl(
    (m+j+2)a_{m,j}\cos\bigl((j-1)x\bigr)
    +
    (m-j+2)a_{m,j}\cos\bigl((j+1)x\bigr)
  \bigr)
  \\
  =&
  \frac{1}{2}
  \sum_{\substack{-1\le k\le m-1\\ k\equiv m+1\pmod{2}}}(m+k+3)a_{m,k+1}\cos(kx)
  \\
  &+
  \frac{1}{2}
  \sum_{\substack{1\le k\le m+1\\ k\equiv m+1\pmod{2}}}(m-k+3)a_{m,k-1}\cos(kx)
  \\
  =&
  \begin{cases}
    &\frac{1}{2}
    \sum_{\substack{0\le k\le m-1\\ k\equiv m+1\pmod{2}}}
    \left((m+k+3)a_{m,k+1}+(m-k+3)a_{m,k-1}\right)\cos(kx)
    \\
    &\quad+
    a_{m,m}\cos\bigl((m+1)x\bigr)
    \hfill\text{(if $m$ is odd)},
    \\[5mm]
    &\frac{1}{2}
    \sum_{\substack{0\le k\le m-1\\ k\equiv m+1\pmod{2}}}
    \left((m+k+3)a_{m,k+1}+(m-k+3)a_{m,k-1}\right)\cos(kx)
    \\
    &\quad+
    a_{m,m}\cos\bigl((m+1)x\bigr)
    +
    \frac{1}{2}(m+2)a_{m,0}\cos(x)
    \hfill\text{(if $m$ is even)}.
  \end{cases}
\end{split}
\end{equation}
Therefore we obtain the following recursive formula for $a_{m,k}$:
\begin{equation*}
  2a_{m+1,k}
  =
  \begin{cases}
    (m+k+3)a_{m,k+1}+(m-k+3)a_{m,k-1}
    &
    \text{if $k\ne1$},
    \\
    (m+k+3)a_{m,2}+2(m-k+3)a_{m,0}
    &
    \text{if $k=1$}.
  \end{cases}
\end{equation*}
Note that this also holds for $k=0$ and $k=m+1$ by putting $a_{m,-1}=a_{m,m+2}=0$.
Then, (i) follows since $m-k+3\ge3$, (iii) follows since $a_{m+1,m+1}=1$, and (ii) follows since the sum of the coefficients in the third expression of \eqref{eq:sin_cos} equals
\begin{align*}
  &\frac{1}{2}
  \sum_{\substack{0\le j\le m\\j\equiv m\pmod{2}}}
  \left((m+j+2)a_{m,j}+(m-j+2)a_{m,j}\right)
  \\
  =&
  (m+2)\sum_{\substack{0\le j\le m\\j\equiv m\pmod{2}}}a_{m,j}.
\end{align*}
\par
This completes the proof.
\end{proof}
For a real number $\nu>0$, we define the region
\begin{multline*}
  \bowtie_{\nu}
  :=
  \left\{
    z\in\C\Bigm|
    \Im{z}\ge0,\Im\left(\frac{z-\nu}{\gamma}\right)\le0
  \right\}
  \\
  \bigcup
  \left\{
    z\in\C\Bigm|
    \Im{z}\le0,\Im\left(\frac{z+\nu}{\gamma}\right)\ge0
  \right\}.
\end{multline*}
See Figure~\ref{fig:bowtie}.
\begin{figure}[h]
\pic{0.3}{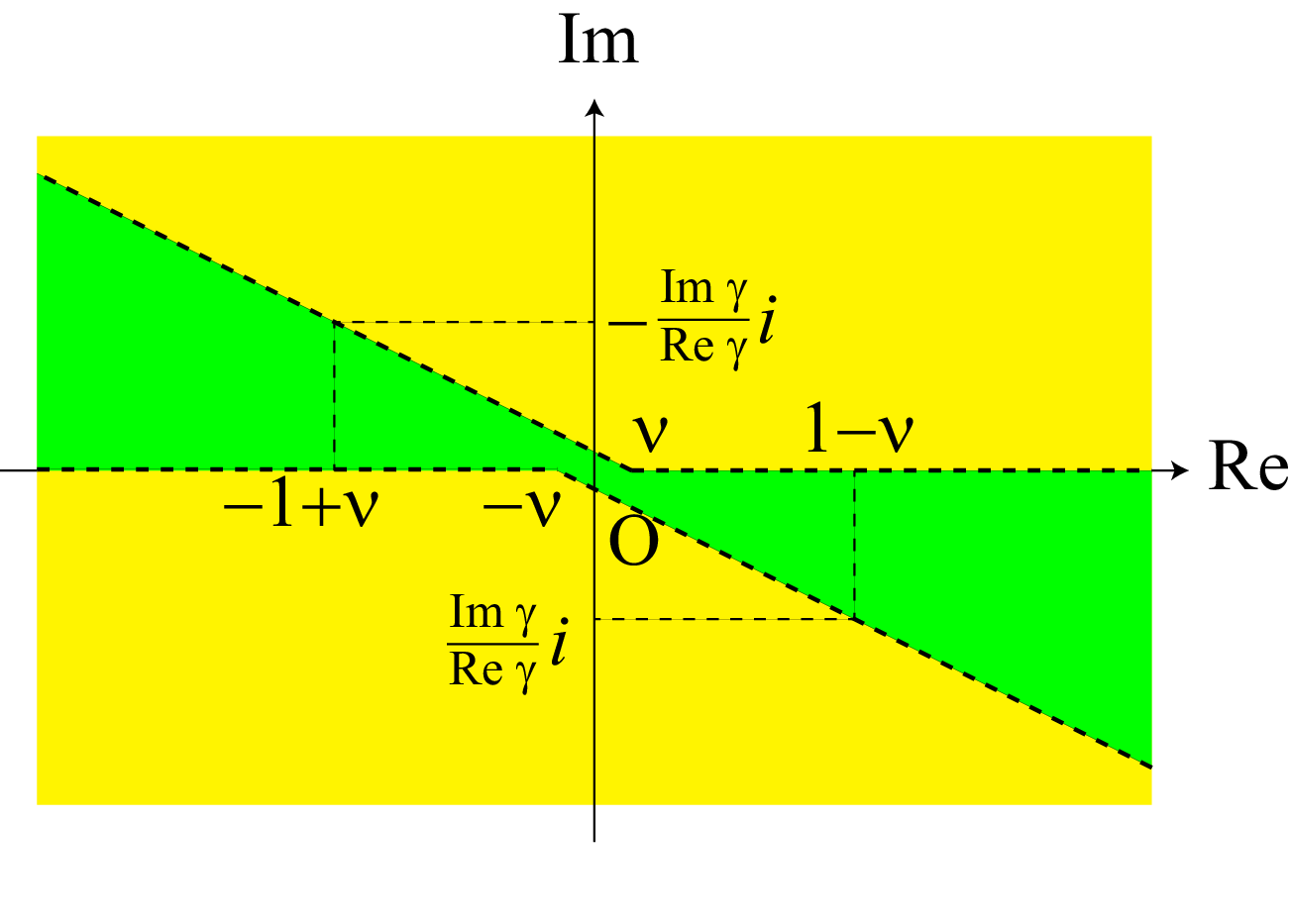}
\caption{The green region is $\bowtie_{\nu}$.}
\label{fig:bowtie}
\end{figure}
\begin{lem}\label{lem:L2_L1}
There exist positive real numbers $c$ and $\varepsilon$ such that
\begin{equation*}
  \left|
    \frac{N}{2\pi\i\gamma}\L_2(z)
    -
    \frac{N}{2\pi\i\gamma}\L_2(z+1)
    -
    \L_1\left(\frac{N}{\gamma}z+\frac{1}{2}\right)
  \right|
  <
  ce^{-\varepsilon N}
\end{equation*}
for any $z$ in the region $\C\setminus\bowtie_{\nu}$ if $N$ is sufficiently large.
\end{lem}
\begin{rem}
The left-hand side is defined unless $\Im{z}=0$ or $z=s\gamma$ ($|s|\ge\frac{1}{2N}$).
Therefore if $z\notin\bowtie_{\nu}$, then the left-hand side is defined.
\end{rem}
\begin{proof}
Note that $\Im{z}\ne0$ if $z\notin\bowtie_{\nu}$.
\par
First, suppose that $\Im{z}>0$.
\par
Since $\L_2(z)=\L_2(z+1)$ from \eqref{eq:L_2}, we will prove that $\left|\L_1\bigl(\frac{N}{\gamma}z+\frac{1}{2}\bigr)\right|<ce^{-\varepsilon N}$ for some $c>0$ and $\varepsilon>0$.
Note that $\Im(z/\gamma)>\Im(\nu/\gamma)=-\nu\Im\gamma/|\gamma|^2>0$ because $z\not\in\bowtie_{\nu}$.
So we have
\begin{equation*}
  \L_1\left(\frac{N}{\gamma}z+\frac{1}{2}\right)
  =
  \log\left(1+e^{2N\pi\i z/\gamma}\right)
\end{equation*}
from \eqref{eq:L_1}.
Now since $\log(1+x)=\sum_{k=1}^{\infty}(-1)^{k-1}x^{k}/k$ for $|x|<1$, one has
\begin{equation*}
  \bigl|\log\left(1+e^{a}\right)\bigr|
  \le
  \sum_{k=1}^{\infty}\frac{e^{k\Re{a}}}{k}
  <
  \sum_{k=1}^{\infty}e^{k\Re{a}}
  =
  \frac{e^{\Re{a}}}{1-e^{\Re{a}}}
\end{equation*}
if $\Re{a}<0$.
Since $\Re(2N\pi\i z/\gamma)=-2N\pi\Im(z/\gamma)<2N\pi\nu\Im\gamma/|\gamma|^2<0$, we have
\begin{align*}
  \left|
    \log\left(1+e^{2N\pi\i z/\gamma}\right)
  \right|
  <
  \frac{e^{2N\pi\nu\Im\gamma/|\gamma|^2}}{1-e^{2N\pi\nu\Im\gamma/|\gamma|^2}}
  <
  ce^{-\varepsilon N}
\end{align*}
where we put $\varepsilon:=-2\pi\nu\Im\gamma/|\gamma|^2>0$ and $c:=\frac{1}{1-e^{-\varepsilon}}>0$.
\par
Next, suppose that $\Im{z}<0$.
\par
From Corollary~\ref{cor:L1_L2_Im_negative}, we have
\begin{equation*}
  \frac{N}{2\pi\i\gamma}\L_2(z)-\frac{N}{2\pi\i\gamma}\L_2(z+1)
  =
  \frac{2N\pi\i z}{\gamma}.
\end{equation*}
Since $z\not\in\bowtie_{\nu}$, we have $\Im(z/\gamma)<-\Im(\nu/\gamma)=\nu\Im\gamma/|\gamma|^2<0$.
Thus from \eqref{eq:L_1} we obtain
\begin{align*}
  \L_1\left(\frac{N}{\gamma}z+\frac{1}{2}\right)
  =
  \log\left(1+e^{-2N\pi\i z/\gamma}\right)
  +
  \frac{2N\pi\i z}{\gamma}.
\end{align*}
Since $\Re(-2N\pi\i z/\gamma)=2N\pi\Im(z/\gamma)<2N\pi\nu\Im\gamma/|\gamma|^2$, we finally have
\begin{align*}
  &\left|
    \frac{N}{2\pi\i\gamma}\L_2(z)
    -
    \frac{N}{2\pi\i\gamma}\L_2(z+1)
    -
    \L_1\left(\frac{N}{\gamma}z+\frac{1}{2}\right)
  \right|
  \\
  =&
  \left|\log\left(1+e^{-2N\pi\i z/\gamma}\right)\right|
  <
  ce^{-\varepsilon N}
\end{align*}
as above, completing the proof.
\end{proof}
The following lemma is similar to \cite[Lemma~2.4]{Murakami:CANJM2023} and the proof is omitted.
\begin{lem}\label{lem:TN_L2}
Let $\nu$ and $M$ be positive real numbers.
Then, there exists a constant $c>0$ such that
\begin{equation*}
  \left|
    T_{N}(z)-\frac{N}{2\pi\i\gamma}\L_2(z)
  \right|
  =
  \frac{c}{N}
\end{equation*}
for $z$ in the region $\{z\in\C\mid\nu\le\Re{z}\le1-\nu,|\Im{z}|\le M\}$ if $N$ is sufficiently large, where $c$ does not depend on $z$.
\end{lem}
\par
Now, we can prove the following proposition.
\begin{prop}\label{prop:TN_L2}
Suppose that $\nu<1/4$.
We have
\begin{equation*}
  T_{N}(z)
  =
  \frac{N}{2\pi\i\gamma}\L_2(z)
  +
  O(N^{-1})
\end{equation*}
as $N\to\infty$ in the region $\Omega^{\ast}_{\nu}$.
\end{prop}
\begin{proof}
We need to prove the proposition for $z$ with $-1+\nu\le\Re{z}<\nu$ or $1-\nu<\Re{z}\le2-\nu$.
\par
If $z\in\Omega^{\ast}_{\nu}$ and $-1+\nu\le\Re{z}<-\nu$, we use \eqref{eq:TN_negative}.
We have
\begin{align*}
  &\left|
    T_{N}(z)-\frac{N}{2\pi\i\gamma}\L_2(z)
  \right|
  =
  \left|
    T_{N}(z+1)
    +
    \L_1\left(\frac{N}{\gamma}z+\frac{1}{2}\right)
    -
    \frac{N}{2\pi\i\gamma}\L_2(z)
  \right|
  \\
  \le&
  \left|
    T_{N}(z+1)-\frac{N}{2\pi\i\gamma}\L_2(z+1)
  \right|
  \\
  &+
  \left|
    \L_1\left(\frac{N}{\gamma}z+\frac{1}{2}\right)
    -
    \frac{N}{2\pi\i\gamma}\L_2(z)
    +
    \frac{N}{2\pi\i\gamma}\L_2(z+1)
  \right|
  =
  O(1/N),
\end{align*}
where we apply Lemmas~\ref{lem:TN_L2} and \ref{lem:L2_L1}, noting that we can apply Lemma~\ref{lem:L2_L1} because $z\notin\bowtie_{\nu}$.
\par
Similarly, if $z\in\Omega^{\ast}_{\nu}$ and $1+\nu<\Re{z}\le2-\nu$, using \eqref{eq:TN_positive}, we have
\begin{align*}
  &\left|
    T_{N}(z)-\frac{N}{2\pi\i\gamma}\L_2(z)
  \right|
  =
  \left|
    T_{N}(z-1)
    -
    \L_1\left(\frac{N}{\gamma}(z-1)+\frac{1}{2}\right)
    -
    \frac{N}{2\pi\i\gamma}\L_2(z)
  \right|
  \\
  \le&
  \left|
    T_{N}(z-1)-\frac{N}{2\pi\i\gamma}\L_2(z-1)
  \right|
  \\
  &+
  \left|
    -\L_1\left(\frac{N}{\gamma}(z-1)+\frac{1}{2}\right)
    -
    \frac{N}{2\pi\i\gamma}\L_2(z)
    +
    \frac{N}{2\pi\i\gamma}\L_2(z-1)
  \right|
  =
  O(1/N),
\end{align*}
noting that we can apply Lemma~\ref{lem:L2_L1} because $z-1\notin\bowtie_{\nu}$.
\par
If $z\in\Omega^{\ast}_{\nu}$ and $-\nu\le\Re{z}<\nu$, we put $m:=\left\lfloor\frac{2N\nu}{\Re\gamma}\right\rfloor+1$.
From Lemma~\ref{lem:TN_gamma_Omega} we have
\begin{equation*}
\begin{split}
  T_{N}(z)
  =&
  T_{N}\left(z+\frac{\gamma}{N}\right)
  +
  \L_1\left(z+\frac{\gamma}{2N}\right)
  \\
  =&
  T_{N}\left(z+\frac{2\gamma}{N}\right)
  +
  \L_1\left(z+\frac{3\gamma}{2N}\right)
  +
  \L_1\left(z+\frac{\gamma}{2N}\right)
  \\
  =&\cdots
  =
  T_{N}\left(z+\frac{m\gamma}{N}\right)
  +
  \sum_{j=1}^{m}\L_1\left(z+\frac{(2j-1)\gamma}{2N}\right).
\end{split}
\end{equation*}
Now since $m\le\frac{2N\nu}{\Re\gamma}+1<m+1$, we have $\nu<\Re\bigl(z+\frac{m\gamma}{N}\bigr)<3\nu+\frac{\Re\gamma}{N}<1-\nu$ if $N>\frac{\Re\gamma}{1-4\nu}$, and so we can apply Lemma~\ref{lem:TN_L2} to $z+\frac{m\gamma}{N}$.
We have
\begin{align*}
  &\left|
    T_{N}(z)-\frac{N}{2\pi\i\gamma}\L_2(z)
  \right|
  \\
  =&
  \left|
    T_{N}\left(z+\frac{m\gamma}{N}\right)
    -
    \frac{N}{2\pi\i\gamma}\L_2(z)
    +
    \sum_{j=1}^{m}\L_1\left(z+\frac{(2j-1)\gamma}{2N}\right)
  \right|
  \\
  \le&
  \left|
    T_{N}\left(z+\frac{m\gamma}{N}\right)
    -
    \frac{N}{2\pi\i\gamma}\L_2\left(z+\frac{m\gamma}{N}\right)
  \right|
  \\
  &+
  \left|
    \frac{N}{2\pi\i\gamma}\L_2\left(z+\frac{m\gamma}{N}\right)
    -
    \frac{N}{2\pi\i\gamma}\L_2(z)
    +
    \sum_{j=1}^{m}\L_1\left(z+\frac{(2j-1)\gamma}{2N}\right)
  \right|
  \\
  =&
  \left|
    \frac{N}{2\pi\i\gamma}\L_2\left(z+\frac{m\gamma}{N}\right)
    -
    \frac{N}{2\pi\i\gamma}\L_2(z)
    +
    \sum_{j=1}^{m}\L_1\left(z+\frac{(2j-1)\gamma}{2N}\right)
  \right|
  +O(1/N).
\end{align*}
Since we have
\begin{equation*}
  \L_2\left(z+\frac{m\gamma}{N}\right)-\L_2(z)
  =
  \sum_{j=1}^{m}
  \left(
    \L_2\left(z+\frac{j\gamma}{N}\right)
    -
    \L_2\left(z+\frac{(j-1)\gamma}{N}\right)
  \right),
\end{equation*}
we have
\begin{equation}\label{eq:sum_L2_L1}
\begin{split}
  &\left|
    \frac{N}{2\pi\i\gamma}\L_2\left(z+\frac{m\gamma}{N}\right)
    -
    \frac{N}{2\pi\i\gamma}\L_2(z)
    +
    \sum_{j=1}^{m}\L_1\left(z+\frac{(2j-1)\gamma}{2N}\right)
  \right|
  \\
  \le&
  \sum_{j=1}^{m}
  \left|
    \frac{N}{2\pi\i\gamma}\L_2\left(z+\frac{j\gamma}{N}\right)
    -
    \frac{N}{2\pi\i\gamma}\L_2\left(z+\frac{(j-1)\gamma}{N}\right)
  \right.
  \\
  &
  \phantom{\sum_{j=1}^{m}}\hfill
  \left.
    +
    \L_1\left(z+\frac{(2j-1)\gamma}{2N}\right)
  \right|.
\end{split}
\end{equation}
We use Lemma~\ref{lem:L2_L1_gamma} to conclude that each summand of the right-hand side of \eqref{eq:sum_L2_L1} is less than $\frac{c}{N^2}$ for $c>0$.
Note that $c$ is independent of $j$.
Since $m=\left\lfloor\frac{2N\nu}{\Re\gamma}\right\rfloor+1$, the right-hand side of \eqref{eq:sum_L2_L1} is less than
\begin{equation*}
  \frac{mc}{N^2}
  \le
  \left(\frac{2N\nu}{\Re\gamma}+1\right)
  \frac{c}{N^2}
  \le
  \frac{c'}{N}
\end{equation*}
if we put $c':=\left(\frac{2c\nu}{\Re\gamma}+1\right)$.
\par
If $1-\nu<\Re{z}\le1+\nu$, from Lemma~\ref{lem:TN_gamma_Omega} we have
\begin{equation*}
\begin{split}
  T_{N}(z)
  =&
  T_{N}\left(z-\frac{\gamma}{N}\right)
  -
  \L_1\left(z-\frac{\gamma}{2N}\right)
  \\
  =&
  T_{N}\left(z-\frac{2\gamma}{N}\right)
  -
  \L_1\left(z-\frac{3\gamma}{2N}\right)
  -
  \L_1\left(z-\frac{\gamma}{2N}\right)
  \\
  =&\cdots
  =
  T_{N}\left(z-\frac{m\gamma}{N}\right)
  -
  \sum_{j=1}^{m}\L_1\left(z-\frac{(2j-1)\gamma}{2N}\right),
\end{split}
\end{equation*}
where we put $m:=\left\lfloor\frac{2N\nu}{\Re\gamma}\right\rfloor+1$ as before.
Since $\nu<\Re\bigl(z-\frac{m\gamma}{N}\bigr)<1-\nu$ as before, we can prove the proposition similarly.
\end{proof}

%% file: sum.tex
\section{The colored Jones polynomial}\label{sec:sum}
In this section, we show several results following \cite{Murakami:CANJM2023}.
\par
First of all we recall the following formula due to Habiro \cite[P.~36 (1)]{Habiro:SURIK2000} and Le  \cite[1.2.2 Example, P.~129]{Le:TOPOA2003} (see also \cite[Theorem~5.1]{Masbaum:ALGGT12003}).
\begin{equation}\label{eq:Habiro_Le}
\begin{split}
  J_N(\FE;q)
  &=
  \sum_{k=0}^{N-1}
  \prod_{l=1}^{k}
  \left(q^{(N-l)/2}-q^{-(N-l)/2}\right)
  \left(q^{(N+l)/2}-q^{-(N+l)/2}\right)
  \\
  &=
  \sum_{k=0}^{N-1}
  q^{-kN}
  \prod_{l=1}^{k}
  \left(1-q^{N-l}\right)
  \left(1-q^{N+l}\right).
\end{split}
\end{equation}
\par
For a positive integer $p$ we put $\xi:=\kappa+2p\pi\i$, where $\kappa:=\arccosh(3/2)$.
We will study the asymptotic behavior of
\begin{equation*}
  J_N\left(\FE;e^{\xi/N}\right)
  =
  \sum_{k=0}^{N-1}
  \prod_{l=1}^{k}
  e^{-k\xi}
  \left(1-e^{(N-l)\xi/N}\right)\left(1-e^{(N+l)\xi/N}\right)
\end{equation*}
as $N\to\infty$.
\par
Now, we can express $J_N(\FE;e^{\xi/N})$ in terms of $T_N$, putting $\gamma:=\frac{\xi}{2\pi\i}$, as in the same way as \cite[Section~3 (3.2)]{Murakami:CANJM2023}.
We have
\begin{multline}\label{eq:J_FN}
  J_N(\FE;e^{\xi/N})
  \\
  =
  \left(1-e^{-4pN\pi^2/\xi}\right)
  \sum_{m=0}^{p-1}
  \left(
    \beta_{p,m}
    \sum_{mN/p<k\le(m+1)N/p}
    \exp\left(Nf_N\left(\frac{2k+1}{2N}-\frac{m}{\gamma}\right)\right)
  \right)
\end{multline}
since $2\sinh(\kappa/2)=1$, where we put
\begin{align}
  \beta_{p,m}
  &:=
  e^{-4mpN\pi^2/\xi}
  \prod_{j=1}^{m}
  \left(1-e^{4(p-j)N\pi^2/\xi}\right)
  \left(1-e^{4(p+j)N\pi^2/\xi}\right),
  \label{eq:beta_pm}
  \\
  f_N(z)
  &:=
  \frac{1}{N}T_N\bigl(\gamma(1-z)-p+1\bigr)
  -
  \frac{1}{N}T_N\bigl(\gamma(1+z)-p\bigr)
  -\kappa z-\frac{2p\pi\i}{\gamma}.
  \label{eq:fN}
\end{align}
\begin{lem}\label{lem:domain_fN}
The function $f_N$ is defined in the region
\begin{equation*}
  \Theta_0
  :=
  \left\{
    z\in\C
    \Bigm|
    -\frac{1}{p}+\frac{1}{2N}<\frac{\Im(\xi z)}{2p\pi}<\frac{2}{p}-\frac{1}{2N}
  \right\}
  \setminus
  \bigl(
    \underline{\nabla}_{0}^{+}
    \cup
    \underline{\nabla}_{0}^{-}
    \cup
    \overline{\nabla}_{0}^{+}
    \cup
    \overline{\nabla}_{0}^{-}
  \bigr),
\end{equation*}
where we put
\begin{align*}
  \underline{\nabla}_{0}^{+}
  :=&
  \left\{
    z\in\C\Bigm|
    -\frac{1}{p}+\frac{1}{2N}<\frac{\Im(\xi z)}{2p\pi}\le0,
    \Re(\xi z)\le\kappa,
    \Im{z}\le-\frac{2p\pi\kappa}{|\xi|^2}
  \right\},
  \\
  \underline{\nabla}_{0}^{-}
  :=&
  \left\{
    z\in\C\Bigm|
    \frac{1}{p}\le\frac{\Im(\xi z)}{2p\pi}<\frac{2}{p}-\frac{1}{2N},
    \Re(\xi z)\ge\kappa,
    \Im{z}\ge\frac{2(1-p)\pi\kappa}{|\xi|^2}
  \right\},
  \\
  \overline{\nabla}_{0}^{+}
  :=&
  \left\{
    z\in\C\Bigm|
    -\frac{1}{p}+\frac{1}{2N}<\frac{\Im(\xi z)}{2p\pi}\le0,
    \Re(\xi z)\le-\kappa,
    \Im{z}\le\frac{2p\pi\kappa}{|\xi|^2}
  \right\},
  \\
  \overline{\nabla}_{0}^{-}
  :=&
  \left\{
    z\in\C\Bigm|
    \frac{1}{p}\le\frac{\Im(\xi z)}{2p\pi}<\frac{2}{p}-\frac{1}{2N},
    \Re(\xi z)\ge-\kappa,
    \Im{z}\ge\frac{2(p+1)\pi\kappa}{|\xi|^2}
  \right\}.
\end{align*}
See Figure~\ref{fig:domain_fN}, where we put
\begin{align*}
  \overline{K}
  &:=
  \{z\in\C\mid z=\frac{2\pi\i}{\xi}t-1,t\in\R\}=\{z\in\C\mid\Re(\xi z)=-\kappa\},
  \\
  \underline{K}
  &:=
  \{z\in\C\mid z=\frac{2\pi\i}{\xi}t+1,t\in\R\}=\{z\in\C\mid\Re(\xi z)=\kappa\},
  \\
  L_{s}
  &:=
  \{z\in\C\mid\Im(\xi z)=2s\pi\}.
\end{align*}
\begin{figure}[h]
\pic{0.3}{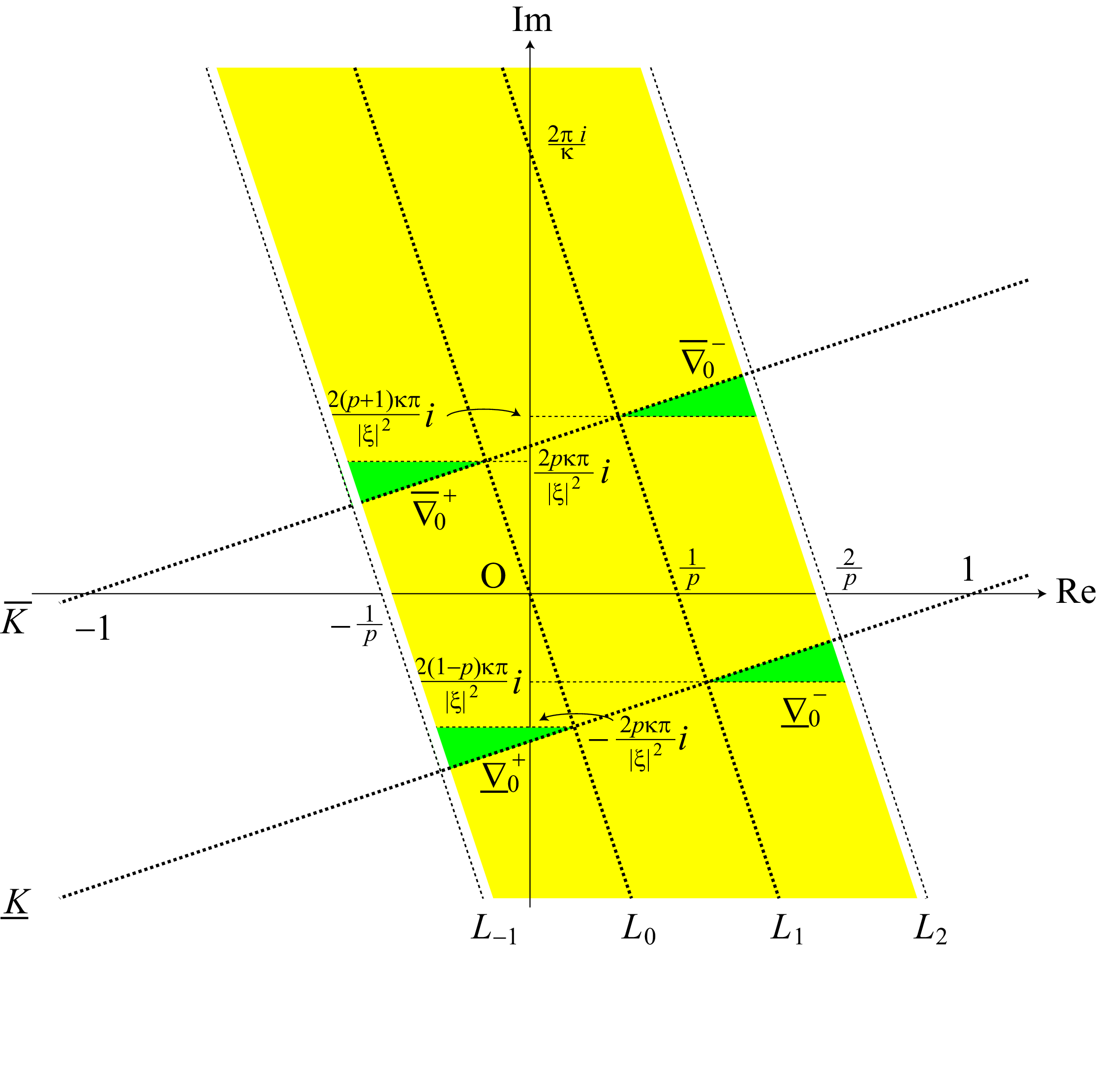}
\caption{The function $f_N$ is defined in the yellow region $\Theta_0$.
The green triangles are $\underline{\nabla}_{0}^{+}$, $\underline{\nabla}_{0}^{-}$, $\overline{\nabla}_{0}^{+}$, and $\overline{\nabla}_{0}^{-}$.}
\label{fig:domain_fN}
\end{figure}
\end{lem}
\begin{proof}
Recall that the the function $T_N$ is defined in $\Omega$ (see \eqref{eq:Omega}).
\par
Since $\gamma=\frac{\xi}{2\pi\i}=\frac{\kappa+2p\pi\i}{2\pi\i}$, we have $\Re\gamma=p$, $\Im\gamma=-\frac{\kappa}{2\pi}$, and
\begin{align*}
  \Re\bigl(\gamma(1\pm z)\bigr)
  &=
  p\pm\frac{\Im(\xi z)}{2\pi},
  \\
  \Im\bigl(\gamma(1\pm z)\bigr)
  &=
  -\frac{\kappa}{2\pi}\mp\frac{\Re(\xi z)}{2\pi}.
\end{align*}
Therefore we have
\begin{align*}
  &-1+\frac{p}{2N}<\Re\bigl(\gamma(1-z)-p+1\bigr)<2-\frac{p}{2N}
  \\
  \Leftrightarrow&
  -\frac{1}{p}+\frac{1}{2N}<\frac{\Im(\xi z)}{2p\pi}<\frac{2}{p}-\frac{1}{2N}
  \\
  \Leftrightarrow&
  -1+\frac{p}{2N}<\Re\bigl(\gamma(1+z)-p\bigr)<2-\frac{p}{2N}.
\end{align*}
\par
We can also see that the condition $\gamma(1-z)-p+1\in\Delta_{0}^{+}$ is equivalent to $z\in\underline{\nabla}_{0}^{-}$, that the condition $\gamma(1-z)-p+1\in\Delta_{1}^{-}$ is equivalent to $z\in\underline{\nabla}_{0}^{+}$, that the condition $\gamma(1+z)-p\in\Delta_{0}^{+}$ is equivalent to $z\in\overline{\nabla}_{0}^{+}$, and that the condition $\gamma(1+z)-p\in\Delta_{1}^{-}$ is equivalent to $z\in\overline{\nabla}_{0}^{-}$.
\end{proof}
\par
We would like to approximate $f_N(z)$ by using $\L_2$.
From Proposition~\ref{prop:TN_L2} and \eqref{eq:fN}, the series of functions $\{f_N(z)\}$ converges uniformly to
\begin{equation*}
  F(z)
  :=
  \frac{1}{\xi}
  \left(
    \L_2\left(\frac{\xi(1-z)}{2\pi\i}-p+1\right)
    -
    \L_2\left(\frac{\xi(1+z)}{2\pi\i}-p\right)
  \right)
  -\kappa z+\frac{4p\pi^2}{\xi}
\end{equation*}
in the following region:
\begin{multline}\label{eq:convergence_F}
  \left\{
    z\in\C\Bigm|
    -1+\frac{\nu}{p}\le\frac{\Im(\xi z)}{2p\pi}\le\frac{2}{p}-\frac{\nu}{p},
    \left|\Re(\xi z)\right|\le2M\pi-\kappa
  \right\}
  \\
  \setminus
  \left(
    \underline{\nabla}_{0,\nu}^{+}
    \cup
    \underline{\nabla}_{0,\nu}^{-}
    \cup
    \overline{\nabla}_{0,\nu}^{+}
    \cup
    \overline{\nabla}_{0,\nu}^{-}
  \right),
\end{multline}
where we put
\begin{align*}
  \underline{\nabla}_{0,\nu}^{+}
  &:=
  \left\{
    z\in\C\Bigm|
    -\frac{1}{p}+\frac{\nu}{p}\le\frac{\Im(\xi z)}{2p\pi}<\frac{\nu}{p},
    \Re(\xi z)<\kappa+2\pi\nu,
  \right.
  \\
  &\qquad\qquad\phantom{z\in\C\Bigm|}
  \left.
    \Im{z}<-\frac{2(p-\nu)\pi\kappa}{|\xi|^2}
  \right\},
  \\
  \underline{\nabla}_{0,\nu}^{-}
  &:=
  \left\{
    z\in\C\Bigm|
    \frac{1}{p}-\frac{\nu}{p}<\frac{\Im(\xi z)}{2p\pi}\le\frac{2}{p}-\frac{\nu}{p},
    \Re(\xi z)\ge\kappa-2\pi\nu,
  \right.
  \\
  &\qquad\qquad\phantom{z\in\C\Bigm|}
  \left.
    \Im{z}>\frac{2(1-p-\nu)\pi\kappa}{|\xi|^2}
  \right\},
  \\
  \overline{\nabla}_{0,\nu}^{+}
  &:=
  \left\{
    z\in\C\Bigm|
    -\frac{1}{p}+\frac{\nu}{p}\le\frac{\Im(\xi z)}{2p\pi}<\frac{\nu}{p},
    \Re(\xi z)<-\kappa+2\pi\nu,
  \right.
  \\
  &\qquad\qquad\phantom{z\in\C\Bigm|}
  \left.
    \Im{z}<\frac{2(p+\nu)\pi\kappa}{|\xi|^2}
  \right\},
  \\
  \overline{\nabla}_{0,\nu}^{-}
  &:=
  \left\{
    z\in\C\Bigm|
    \frac{1}{p}-\frac{\nu}{p}<\frac{\Im(\xi z)}{2p\pi}\le\frac{2}{p}-\frac{\nu}{p},
    \Re(\xi z)>-\kappa-2\pi\nu,
  \right.
  \\
  &\qquad\qquad\phantom{z\in\C\Bigm|}
  \left.
    \Im{z}>\frac{2(p+1-\nu)\pi\kappa}{|\xi|^2}
  \right\}.
\end{align*}
\begin{lem}\label{lem:convergence_fN}
The series of functions $\{f_N(z)\}$ uniformly converges to $F(z)$ in the region \eqref{eq:convergence_F}.
\end{lem}
\begin{proof}
In a way similar to the proof of Lemma~\ref{lem:domain_fN}, we have
\begin{align*}
  &-1+\nu<\Re\bigl(\gamma(1-z)-p+1\bigr)<2-\nu
  \\
  \Leftrightarrow\quad&
  -\frac{1}{p}+\frac{\nu}{p}<\frac{\Im(\xi z)}{2p\pi}<\frac{2}{p}-\frac{\nu}{p}
  \\
  \Leftrightarrow\quad&
  -1+\nu<\Re\bigl(\gamma(1+z)-p\bigr)<2-\nu,
\end{align*}
\begin{align*}
  \left|\Im\bigl(\gamma(1\pm z)\bigr)\right|\le M
  \quad\Leftrightarrow\quad
  \kappa-2M\pi\le\mp\Re(\xi z)\le\kappa+2M\pi,
\end{align*}
and
\begin{align*}
  \gamma(1-z)-p+1\in\Delta^{+}_{0,\nu}
  \quad\Leftrightarrow\quad&
  z\in\underline{\nabla}_{0,\nu}^{-},
  \\
  \gamma(1-z)-p+1\in\Delta^{-}_{1,\nu}
  \quad\Leftrightarrow\quad&
  z\in\underline{\nabla}_{0,\nu}^{+},
  \\
  \gamma(1+z)-p\in\Delta^{+}_{0,\nu}
  \quad\Leftrightarrow\quad&
  z\in\overline{\nabla}_{0,\nu}^{+},
  \\
  \gamma(1+z)-p\in\Delta^{-}_{1,\nu}
  \quad\Leftrightarrow\quad&
  z\in\overline{\nabla}_{0,\nu}^{-}.
\end{align*}
Then, the lemma follows from Proposition~\ref{prop:TN_L2}.
\end{proof}
\par
We can express $F(z)$ in terms of $\Li_2$ for certain cases.
\begin{lem}\label{lem:F_Li2}
If $z$ is in between $\overline{K}$ and $\underline{K}$, or between $L_0$ and $L_1$, then we have
\begin{equation}\label{eq:F_K_L0_L1}
  F(z)
  =
  \frac{1}{\xi}\Li_2\bigl(e^{-\xi(1+z)}\bigr)
  -
  \frac{1}{\xi}\Li_2\bigl(e^{-\xi(1-z)}\bigr)
  +\kappa z-2\pi\i.
\end{equation}
Moreover, if $z$ is between $L_0$ and $L_1$, we also have
\begin{equation}\label{eq:F_L0_L1}
  F(z)
  =
  \frac{1}{\xi}\Li_2\bigl(e^{\xi(1-z)}\bigr)
  -
  \frac{1}{\xi}\Li_2\bigl(e^{\xi(1+z)}\bigr)
  -\kappa z+\frac{4p\pi^2}{\xi}.
\end{equation}
\end{lem}
\begin{proof}
Since $\Im\frac{\xi(1\pm z)}{2\pi\i}=\frac{-1}{2\pi}\bigl(\kappa\pm\Re(\xi z)\bigr)$, we see that $\Im\frac{\xi(1+z)}{2\pi\i}<0$ and $\Im\frac{\xi(1-z)}{2\pi\i}<0$ if $z$ is between $\overline{K}$ and $\underline{K}$.
Thus, in this case we have \eqref{eq:F_K_L0_L1} from \eqref{eq:L2_dilog}.
\par
Next, we consider the case where $z$ is between $L_0$ and $L_1$, that is, where $0<\Im(\xi z)<2\pi$.
\par
We have $\Re\frac{\xi(1-z)}{2\pi\i}-p+1=1-\frac{\Im(\xi z)}{2\pi}$ and $\Re\frac{\xi(1+z)}{2\pi\i}-p=\frac{\Im(\xi z)}{2\pi}$, both of which are between $0$ and $1$.
So, from Lemmas~\ref{lem:L0_L1_L2} and \ref{lem:L2_Li2} we have \eqref{eq:F_L0_L1}.
\par
Now, we will show that \eqref{eq:F_K_L0_L1} also holds in this case.
\par
From \eqref{eq:dilog_inversion}, we have
\begin{align*}
  \Li_2\left(e^{\xi(1-z)}\right)
  =&
  -\Li_2\left(e^{-\xi(1-z)}\right)
  -\frac{\pi^2}{6}
  -\frac{1}{2}\left(\log\bigl(-e^{-\xi(1-z)}\bigr)\right)^2
  \\
  =&
  -\Li_2\left(e^{-\xi(1-z)}\right)
  -\frac{\pi^2}{6}
  -\frac{1}{2}\left(-\xi(1-z)+(2p-1)\pi\i\right)^2
\end{align*}
since $\Im\xi(1-z)=2p\pi-\Im(\xi z)$, which is between $2(p-1)\pi$ and $2p\pi$ when $0<\Im(\xi z)<2\pi$, that is, when $z$ is between $L_0$ and $L_1$.
Similarly we have
\begin{align*}
  \Li_2\left(e^{\xi(1+z)}\right)
  =&
  -\Li_2\left(e^{-\xi(1+z)}\right)
  -\frac{\pi^2}{6}
  -\frac{1}{2}\left(\log\bigl(-e^{-\xi(1+z)}\bigr)\right)^2
  \\
  =&
  -\Li_2\left(e^{-\xi(1+z)}\right)
  -\frac{\pi^2}{6}
  -\frac{1}{2}\left(-\xi(1+z)+(2p+1)\pi\i\right)^2
\end{align*}
since $\Im\xi(1+z)=2p\pi+\Im(\xi z)$, which is between $2p\pi$ and $2(p+1)\pi$.
Thus, from \eqref{eq:F_L0_L1}, we obtain \eqref{eq:F_K_L0_L1}, completing the proof.
\end{proof}
\par
The derivatives of $F(z)$ are given as follows from Lemma~\ref{lem:der_L1_L2}:
\begin{align}
  F'(z)
  =&
  \L_1\left(\frac{\xi(1-z)}{2\pi\i}-p+1\right)
  +
  \L_1\left(\frac{\xi(1+z)}{2\pi\i}-p\right)
  -\kappa,
  \label{eq:F'}
  \\
  F''(z)
  =&
  \frac{\xi\left(e^{-\xi z}-e^{\xi z}\right)}{3-e^{\xi z}-e^{-\xi z}},
  \label{eq:F''}
  \\
  F^{(3)}(z)
  =&
  \frac{\xi^2(4-3(e^{\xi z}+e^{-\xi z}))}{(3-e^{\xi z}-e^{-\xi z})^2}.
  \label{eq:F'''}
\end{align}
If $z$ is between $\overline{K}$ and $\underline{K}$, or between $L_0$ and $L_1$, we have
\begin{equation}\label{eq:F'_log}
\begin{split}
  F'(z)
  =&
  \log(1-e^{-\kappa-\xi z})
  +
  \log(1-e^{-\kappa+\xi z})
  +\kappa
  \\
  =&
  \log\left(3-e^{\xi z}-e^{-\xi z}\right)
\end{split}
\end{equation}
from Lemma~\ref{lem:F_Li2}, where the second equality follows from the same reason as \cite[Equation~(4.2)]{Murakami:CANJM2023}.
\par
Put $\saddle_{0}:=\frac{2\pi\i}{\xi}=\frac{2\pi}{|\xi|^2}(2p\pi+\kappa\i)$.
Since $\Re(\xi\saddle_0)=0$ and $\Im(\xi\saddle_0)=2\pi$, we conclude that $\saddle_0$ is on $L_1$ and between $\overline{K}$ and $\underline{K}$.
From \eqref{eq:F_K_L0_L1}, \eqref{eq:F'_log}, \eqref{eq:F''}, and \eqref{eq:F'''} we have
\begin{equation}\label{eq:F_derivatives}
\begin{split}
  F(\saddle_0)
  =&
  \frac{4p\pi^2}{\xi},
  \\
  F'(\saddle_{0})
  =&0,
  \\
  F''(\saddle_{0})
  =&
  0,
  \\
  F^{(3)}(\saddle_{0})
  =&
  -2\xi^2.
\end{split}
\end{equation}

%% file: Poisson.tex
\section{Poisson summation formula}\label{sec:Poisson}
In \eqref{eq:J_FN}, we put $\varphi_{m,N}(z):=f_N(z-2m\pi\i/\xi)$ ($m=0,1,2,\dots,p-1$) so that
\begin{multline}\label{eq:J_phiN}
  J_N(\FE;e^{\xi/N})
  \\
  =
  \left(1-e^{-4pN\pi^2/\xi}\right)
  \sum_{m=0}^{p-1}
  \left(
    \beta_{p,m}
    \sum_{mN/p<k\le(m+1)N/p}
    \exp\left(N\varphi_{m,N}\left(\frac{2k+1}{2N}\right)\right)
  \right).
\end{multline}
Note that the function $\varphi_{m,N}(z)$ is defined in the region
\begin{equation*}
  \Theta_m
  :=
  \left\{
    z\in\C
    \Bigm|
    -\frac{1}{p}+\frac{1}{2N}<\frac{\Im(\xi z)}{2p\pi}<\frac{2}{p}-\frac{1}{2N}
  \right\}
  \setminus
  \bigl(
    \underline{\nabla}_{m}^{+}
    \cup
    \underline{\nabla}_{m}^{-}
    \cup
    \overline{\nabla}_{m}^{+}
    \cup
    \overline{\nabla}_{m}^{-}
  \bigr)
\end{equation*}
from Lemma~\ref{lem:domain_fN}, where we put
\begin{align*}
  \underline{\nabla}_{m}^{+}
  :=&
  \left\{
    z\in\C\Bigm|
    \frac{m-1}{p}+\frac{1}{2N}<\frac{\Im(\xi z)}{2p\pi}\le\frac{m}{p},
  \right.
  \\
  &\qquad\qquad\phantom{z\in\C\Bigm|}
  \left.
    \Re(\xi z)\le\kappa,
    \Im{z}\le\frac{2(m-p)\pi\kappa}{|\xi|^2}
  \right\},
  \\
  \underline{\nabla}_{m}^{-}
  :=&
  \left\{
    z\in\C\Bigm|
    \frac{m+1}{p}\le\frac{\Im(\xi z)}{2p\pi}<\frac{m+2}{p}-\frac{1}{2N},
  \right.
  \\
  &\qquad\qquad\phantom{z\in\C\Bigm|}
  \left.
    \Re(\xi z)\ge\kappa,
    \Im{z}\ge\frac{2(m-p+1)\pi\kappa}{|\xi|^2}
  \right\},
  \\
  \overline{\nabla}_{m}^{+}
  :=&
  \left\{
    z\in\C\Bigm|
    \frac{m-1}{p}+\frac{1}{2N}<\frac{\Im(\xi z)}{2p\pi}\le\frac{m}{p},
  \right.
  \\
  &\qquad\qquad\phantom{z\in\C\Bigm|}
  \left.
    \Re(\xi z)\le-\kappa,
    \Im{z}\le\frac{2(m+p)\pi\kappa}{|\xi|^2}
  \right\},
  \\
  \overline{\nabla}_{m}^{-}
  :=&
  \left\{
    z\in\C\Bigm|
    \frac{m+1}{p}\le\frac{\Im(\xi z)}{2p\pi}<\frac{m+2}{p}-\frac{1}{2N},
  \right.
  \\
  &\qquad\qquad\phantom{z\in\C\Bigm|}
  \left.
    \Re(\xi z)\ge-\kappa,
    \Im{z}\ge\frac{2(m+p+1)\pi\kappa}{|\xi|^2}
  \right\}.
\end{align*}
\par
We would like to show that the sum 
\begin{equation*}
  \sum_{mN/p<k\le(m+1)N/p}
  \exp\left(N\varphi_{m,N}\left(\frac{2k+1}{2N}\right)\right)
\end{equation*}
is approximated by the integral
\begin{equation*}
  N\int_{m/p}^{(m+1)/p}e^{N\varphi_{m,N}(z)}\,dz.
\end{equation*}
To do that, we use the following proposition, known as the Poisson summation formula.
\begin{prop}\label{prop:Poisson}
Let $a$ and $b$ be real numbers with $a<b$, and $\{\psi_N(z)\}_{N=1,2,3,\dots}$ be a series of holomorphic functions in a domain $D\subset\C$ containing the closed interval $[a,b]$.
We assume that $\psi_N(z)$ uniformly converges to a holomorphic function $\psi(z)$ in $D$.
We also assume that $\Re{\psi(a)}<0$ and $\Re{\psi(b)}<0$.
\par
Putting $R_{+}:=\{z\in D\mid\Im{z}\ge0,\Re{\psi(z)}<2\pi\Im{z}\}$ and $R_{-}:=\{z\in D\mid\Im{z}\le0,\Re{\psi(z)}<-2\pi\Im{z}\}$, we also assume that there are paths $C_{\pm}$ connecting $a$ and $b$ such that $C_{\pm}\subset R_{\pm}$ and that $C_{\pm}$ is homotopic to $[a,b]$ in $D$ with $a$ and $b$ fixed.
\par
Then we have
\begin{equation*}
  \frac{1}{N}\sum_{a\le k/N\le b}e^{N\psi_N(k/N)}
  =
  \int_{a}^{b}e^{N\psi_N(z)}\,dz+O(e^{-\varepsilon N})
\end{equation*}
for some $\varepsilon>0$ independent of $N$.
\end{prop}
A proof, which is essentially the same as that of \cite[Proposition~4.2]{Ohtsuki:QT2016}, is given in Appendix~\ref{sec:Poisson_proof}.
\par
To apply Proposition~\ref{prop:Poisson} to the function $\psi_N(z):=\varphi_{m,N}(z)-\varphi_{m,N}(\saddle_m)$, we will study properties of $\varphi_{m,N}(z)$.
\par
From Lemma~\ref{lem:convergence_fN}, the series of functions $\{\varphi_{m,N}(z)\}$ uniformly converges to the function $\Phi_{m}(z):=F(z-2m\pi\i/\xi)$ in the region $\Theta^{\ast}_{m,\nu}$ defined as
\begin{equation}\label{eq:convergence_Phi}
\begin{split}
  &\Theta^{\ast}_{m,\nu}
  \\
  :=&
  \{z\in\C\mid2(m-1+\nu)\pi\le\Im(\xi z)\le2(m+2-\nu)\pi,|\Re(\xi z)|\le2M\pi-\kappa\}
  \\
  &\quad\setminus
  \left(
    \underline{\nabla}^{+}_{m,\nu}
    \cup
    \underline{\nabla}^{-}_{m,\nu}
    \cup
    \overline{\nabla}^{+}_{m,\nu}
    \cup
    \overline{\nabla}^{-}_{m,\nu}
  \right),
\end{split}
\end{equation}
where we put
\begin{align*}
  \underline{\nabla}^{+}_{m,\nu}
  &:=
  \{z\in\C\mid2(m-1+\nu)\pi\le\Im(\xi z)<2(m+\nu)\pi,
  \\
  &\qquad\qquad\phantom{z\in\C\mid}
  \Re(\xi z)<\kappa+2\pi\nu,
  \Im{z}<2(\nu-p+m)\pi\kappa/|\xi|^2\},
  \\
  \underline{\nabla}^{-}_{m,\nu}
  &:=
  \{z\in\C\mid2(m+1-\nu)\pi<\Im(\xi z)\le2(m+2-\nu)\pi,
  \\
  &\qquad\qquad\phantom{z\in\C\mid}
  \Re(\xi z)\ge\kappa-2\pi\nu,
  \Im{z}>2(1-p+m-\nu)\pi\kappa/|\xi|^2\},
  \\
  \overline{\nabla}^{+}_{m,\nu}
  &:=
  \{z\in\C\mid2(m-1+\nu)\pi\le\Im(\xi z)<2(m+\nu)\pi,
  \\
  &\qquad\qquad\phantom{z\in\C\mid}
  \Re(\xi z)<-\kappa+2\pi\nu,
  \Im{z}<2(\nu+p+m)\pi\kappa/|\xi|^2\},
  \\
  \overline{\nabla}^{-}_{m,\nu}
  &:=
  \{z\in\C\mid2(m+1-\nu)\pi<\Im(\xi z)\le2(m+2-\nu)\pi,
  \\
  &\qquad\qquad\phantom{z\in\C\mid}
  \Re(\xi z)>-\kappa-2\pi\nu,
  \Im{z}>2(p+m+1-\nu)\pi\kappa/|\xi|^2\},
\end{align*}
and we always assume that $M$ is sufficiently large.
From \eqref{eq:F'}--\eqref{eq:F'''}, we have
\begin{align}
  \Phi'_{m}(z)
  =&
  \L_1\left(\frac{\xi(1-z)}{2\pi\i}+m-p+1\right)
  +
  \L_1\left(\frac{\xi(1+z)}{2\pi\i}-m-p\right)
  -\kappa,
  \label{eq:Phi'}
  \\
  \Phi''_{m}(z)
  =&
  \frac{\xi\left(e^{-\xi z}-e^{\xi z}\right)}{3-e^{\xi z}-e^{-\xi z}},
  \label{eq:Phi''}
  \\
  \Phi^{(3)}_{m}(z)
  =&
  \frac{\xi^2(4-3(e^{\xi z}+e^{-\xi z}))}{(3-e^{\xi z}-e^{-\xi z})^2}.
  \label{eq:Phi'''}
\end{align}
Since $z-\frac{2m\pi\i}{\xi}$ is between $L_0$ and $L_1$ ($\overline{K}$ and $\underline{K}$, respectively) if and only if $z$ is between $L_m$ and $L_{m+1}$ ($\overline{K}$ and $\underline{K}$, respectively), from \eqref{eq:F'_log} we have
\begin{equation}\label{eq:Phi'_log}
  \Phi'_{m}(z)
  =
  \log\left(3-e^{\xi z}-e^{-\xi z}\right)
\end{equation}
when $z$ is between $\overline{K}$ and $\underline{K}$, or $L_{m}$ and $L_{m+1}$.
\par
We also put $\saddle_m:=\saddle_0+\frac{2m\pi\i}{\xi}=\frac{2(m+1)\pi\i}{\xi}=\frac{2(m+1)\pi}{|\xi|^2}(2p\pi+\kappa\i)$ so that
\begin{equation}\label{eq:Phi_derivatives}
\begin{split}
  \Phi_m(\saddle_{m})
  =&
  \frac{4p\pi^2}{\xi},
  \\
  \Phi'_m(\saddle_{m})
  =&0,
  \\
  \Phi''_m(\saddle_{m})
  =&
  0,
  \\
  \Phi^{(3)}_m(\saddle_{m})
  =&
  -2\xi^2
\end{split}
\end{equation}
from \eqref{eq:F_derivatives}.
Since $\Re(\xi\saddle_m)=0$ and $\Im(\xi\saddle_m)=2(m+1)\pi$, we see that $\saddle_m$ is between $\overline{K}$ and $\underline{K}$ and on the line $L_{m+1}$.
See Figure~\ref{fig:domain_Phi}.
\begin{figure}[h]
  \pic{0.3}{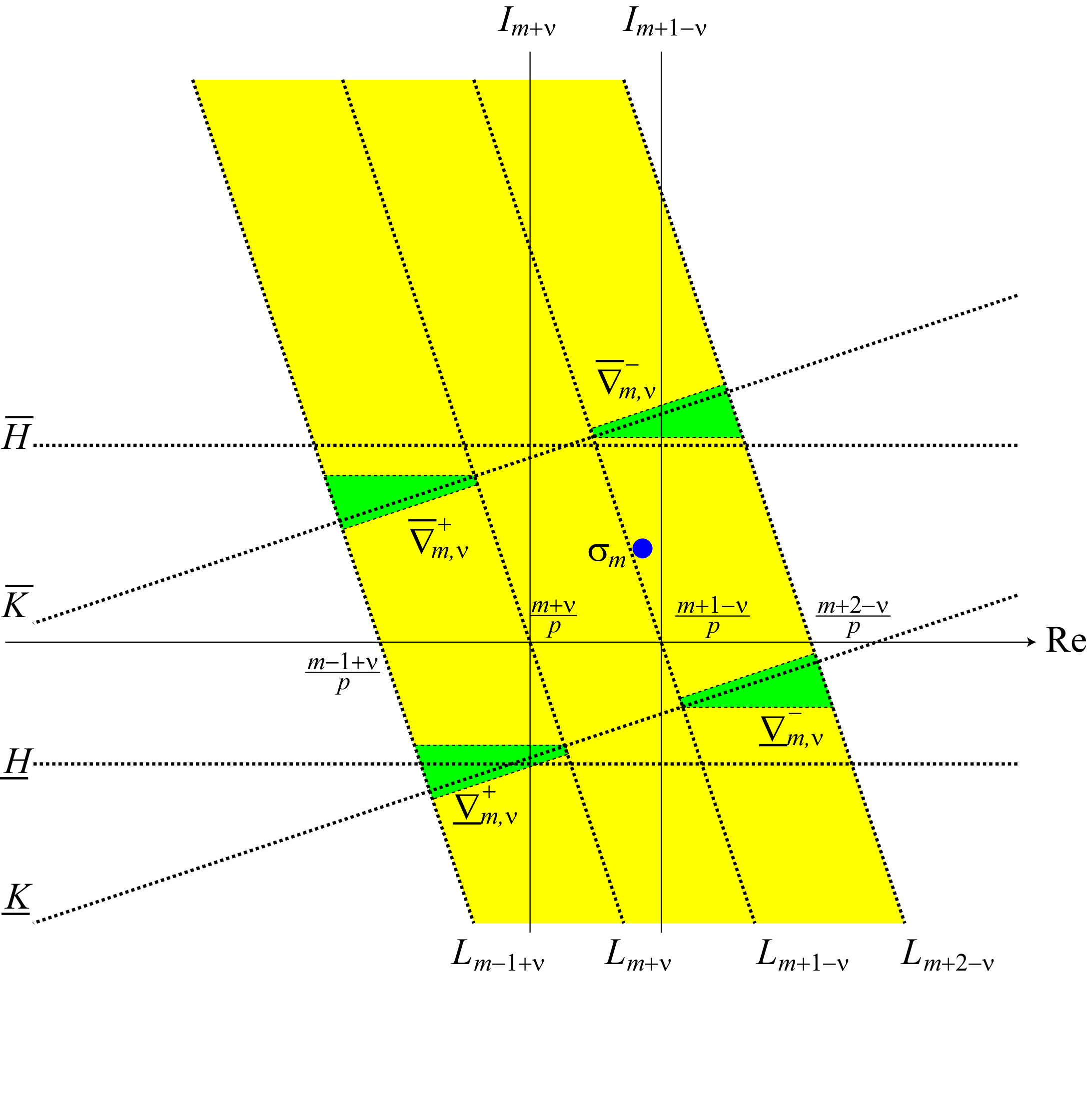}
\caption{The yellow region is $\Theta^{\ast}_{m,\nu}$.
         The blue point is $\saddle_m$.
         The green trapezoids are $\underline{\nabla}_{m,\nu}^{+}$,
         $\underline{\nabla}_{m,\nu}^{-}$,$\overline{\nabla}_{m,\nu}^{+}$,
         and $\overline{\nabla}_{m,\nu}^{-}$.}
\label{fig:domain_Phi}
\end{figure}
\par
Let $I_{s}$ be the vertical line $\Re{z}=s/p$ for $s\in\R$.
\par
For a small number $\chi>0$, let $\Xi_{m,\chi}$ be the pentagonal region defined as follows.
\begin{align*}
  \Xi_{m,\chi}
  :=
  \Bigl\{z\in\C\Bigm|
  &\frac{m-\chi}{p}<\Re{z}<\frac{m+1+\chi}{p},
  -\frac{2(m+1)\kappa\pi}{|\xi|^2}<\Im{z}<\frac{(p+m)\kappa}{2p^2\pi},
  \\
  &\Im(\xi z)+\frac{(2\chi+1)|\xi|^2}{2(m+1)\kappa}\Im{z}>2(m-\chi)\pi
  \Bigr\}
\end{align*}
when $m<p-1$, and
\begin{align*}
  \Xi_{p-1,\chi}
  :=
  \Bigl\{z\in\C\Bigm|
  &\frac{p-1-\chi}{p}<\Re{z}<\frac{p+\chi}{p},
  -\frac{2p\kappa\pi}{|\xi|^2}<\Im{z}<\frac{(2p-1)\kappa}{2p^2\pi},
  \\
  &\Im(\xi z)+\frac{(2\chi+1)|\xi|^2}{2p\kappa}\Im{z}>2(p-1-\chi)\pi
  \Bigr\}
  \setminus\diamond_{\nu},
\end{align*}
where we put
\begin{equation*}
  \diamond_{\nu}
  :=
  \{z\in\C\mid\Re{z}<1+\chi/p\}
  \cap
  \underline{\nabla}_{p-1,\nu}^{-}.
\end{equation*}
Note that $\Xi_{m,\chi}$ ($m<p-1$) is surrounded by $I_{m-\chi}$, $I_{m+1+\chi}$, $\overline{H}$, $\underline{H}$, and $J$, where $\overline{H}$ and $\underline{H}$ are the horizontal lines $\Im{z}=\frac{(p+m)\kappa}{2p^2\pi}$ and $\Im{z}=-\Im\saddle_m$, respectively, and $J$ is  the line connecting $(m-\chi)/p$ and $L_{m+1/2}\cap\underline{H}$, which is given as
\begin{equation}\label{eq:J}
  J:=\{z\in\C\mid\Im(\xi z)+\frac{(2\chi+1)|\xi|^2}{2(m+1)\kappa}\Im{z}>2(m-\chi)\pi\}.
\end{equation}
See the left picture of Figure~\ref{fig:pentagon}.
\begin{figure}[h]
  \pic{0.3}{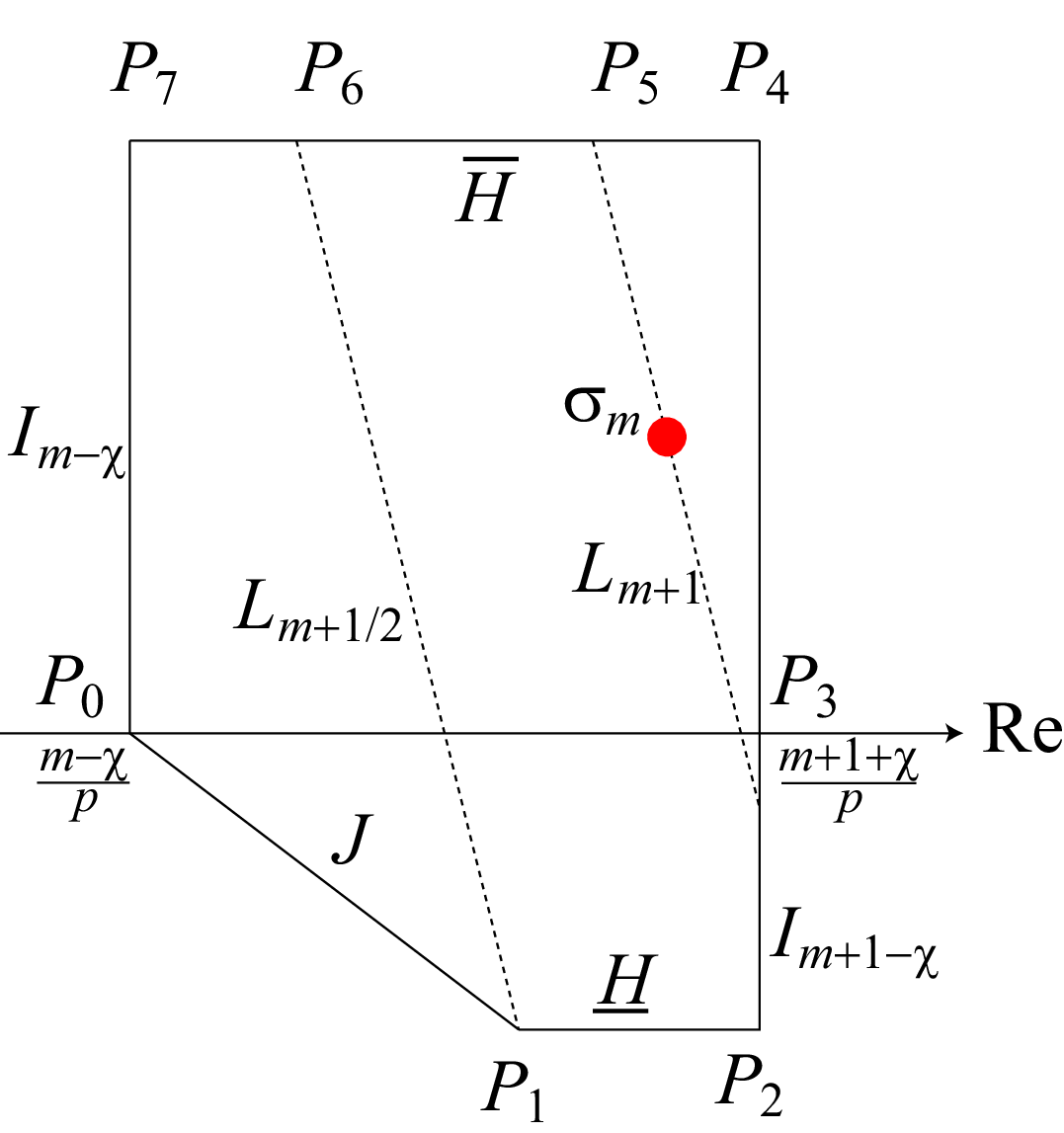}
  \qquad
  \pic{0.3}{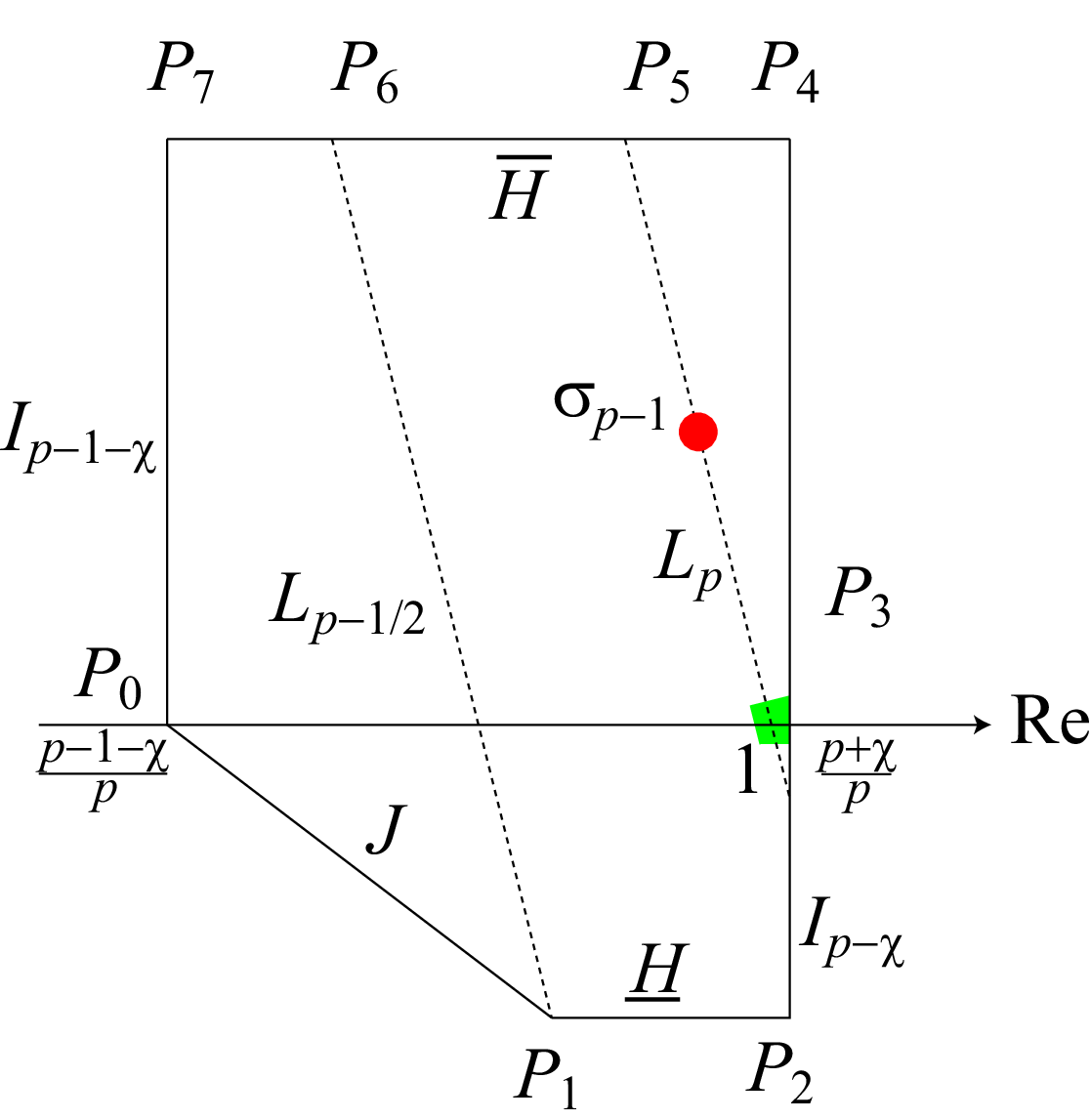}
\caption{The region $\Xi_{m,\chi}$ when $m<p-1$ (left) and the region $\Xi_{p-1,\chi}$ (right), where the green quadrilateral indicates $\diamond_{\nu}$.
Precisely speaking, the points $P_0$ and $P_7$ should be a little more to the right than indicated, and the points $P_2$, $P_3$, and $P_4$ should be a little more to the left than indicated.}
\label{fig:pentagon}
\end{figure}
The right picture of Figure~\ref{fig:pentagon} indicates $\Xi_{p-1,\chi}$, where $\diamond_{\nu}$ is indicated by the green quadrilateral.
Note that it is a neighborhood of the point $1$.
\begin{lem}
If $\nu>0$ is sufficiently small, then we can choose $\chi>0$ so that $\Xi_{m,\chi}$ is included in $\Theta^{\ast}_{m,\nu}$ for $m=0,1,2,\dots,p-1$.
\end{lem}
\begin{proof}
First of all, $\Xi_{m,\chi}$ is in the rectangle surrounded by $I_{m-\chi}$, $I_{m+1+\chi}$, $\overline{H}$, and $\underline{H}$, with bottom-left vertex $v_{1}:=\frac{m-\chi}{p}-\frac{2(m+1)\kappa\pi}{|\xi|^2}\i$ and top-right vertex $v_{2}:=\frac{m+1+\chi}{p}+\frac{(p+m)\kappa}{2p^2\pi}\i$.
The vertices $v_{1}$ and $v_{2}$ are on the lines $L_{\Im(\xi v_{1})/(2\pi)}$ and $v_{2}$ $L_{\Im(\xi v_{\rm2})/(2\pi)}$, respectively.
Since we have
\begin{align*}
  \frac{\Im(\xi v_{1})}{2\pi}-(m-1+\nu)
  =&
  (1-\chi-\nu)
  -
  \frac{(m+1)\kappa^2}{|\xi|^2}
  \\
  \intertext{and}
  (m+2-\nu)-\frac{\Im(\xi v_{2})}{2\pi}
  =&
  (1-\chi-\nu)-\frac{(p+m)\kappa^2}{4p^2\pi^2},
\end{align*}
if
\begin{equation}\label{eq:Upsilon_L}
  \nu+\chi
  <
  \min
  \left\{
    1-(m+1)\left(\frac{\kappa}{|\xi|}\right)^2,
    1-(p+m)\left(\frac{\kappa}{2p\pi}\right)^2
  \right\}
  =
  1-(p+m)\left(\frac{\kappa}{2p\pi}\right)^2,
\end{equation}
then $\Xi_{m,\chi}$ is between the lines $L_{m-1+\nu}$ and $L_{m+2-\nu}$ for $m=0,1,2,\dots,p-1$.
\par
So it remains to show that $\Xi_{m,\chi}$ excludes $\overline{\nabla}^{+}_{m,\nu}$, $\overline{\nabla}^{-}_{m,\nu}$, $\underline{\nabla}^{+}_{m,\nu}$, and $\underline{\nabla}^{-}_{m,\nu}$.
\begin{itemize}
\item
The real part of the bottom right corner of $\overline{\nabla}^{+}_{m,\nu}$ is $\frac{\kappa(2\pi\nu-\kappa)+4(m+\nu)p\pi^2}{|\xi|^2}$, which is smaller than $(m-\chi)/p$ if
\begin{equation}\label{eq:Upsilon_o+}
  2p\pi(\kappa+2p\pi)\nu+|\xi|^2\chi<(p+m)\kappa^2.
\end{equation}
So the trapezoid $\overline{\nabla}^{+}_{m,\nu}$ is to the right of $I_{m-\chi}$ if \eqref{eq:Upsilon_o+} holds.
\item
The difference between the imaginary parts of the bottom line of $\overline{\nabla}^{-}_{m,\nu}$ and $\overline{H}$ is
\begin{equation*}
  \frac{2(p+m+1-\nu)\kappa\pi}{|\xi|^2}
  -
  \frac{(p+m)\kappa}{2p^2\pi}
  =
  \frac{\kappa\bigl(4p^2(1-\nu)\pi^2-(p+m)\kappa^2\bigr)}{2p^2\pi|\xi|^2},
\end{equation*}
which is positive if
\begin{equation}\label{eq:Upsilon_o-}
  \nu<1-(p+m)\left(\frac{\kappa}{2p\pi}\right)^2.
\end{equation}
So, we conclude that $\overline{\nabla}^{-}_{m,\nu}$ is outside of $\Xi_{m,\chi}$ if \eqref{eq:Upsilon_o-} holds.
\item
To obtain a condition that $\underline{\nabla}^{+}_{m,\nu}$ is below $J$, it is enough to find a condition that the top right corner $z_0$ of the trapezoid is below $J$, since $L_{m+\nu}$ is steeper than $J$.
Since $\Im{z_0}=2(\nu-p+m)\kappa\pi/|\xi|^2$ and $z_0$ is on $L_{m+\nu}$, the condition is
\begin{equation*}
  2(m+\nu)\pi+\frac{(2\chi+1)|\xi|^2}{2(m+1)\kappa}\times\frac{2(\nu-p+m)\kappa\pi}{|\xi|^2}
  <
  2(m-\chi)\pi
\end{equation*}
from \eqref{eq:J}.
Therefore, if
\begin{equation}\label{eq:Upsilon_u+}
  2\nu\chi+(2m+3)\nu+2(2m-p+1)\chi<p-m,
\end{equation}
the trapezoid $\underline{\nabla}^{+}_{m,\nu}$ is out of $\Xi_{m,\chi}$.
\item
The real part of the top left corner of $\underline{\nabla}^{-}_{m,\nu}$ is $\frac{\kappa(\kappa-2\pi\nu)+4(m+1-\nu)p\pi^2}{|\xi|^2}$, which is bigger than $(m+1+\chi)/p$
if
\begin{equation}\label{eq:Upsilon_u-}
  2p\pi(\kappa+2p\pi)\nu+|\xi|^2\chi<(p-m-1)\kappa^2.
\end{equation}
So the trapezoid $\underline{\nabla}^{-}_{m,\nu}$ is outside of $\Xi_{m,\chi}$ if \eqref{eq:Upsilon_u-} holds.
\end{itemize}
From \eqref{eq:Upsilon_o+}--\eqref{eq:Upsilon_u-}, we conclude that if $m<p-1$ and
\begin{equation*}
  \nu<
  \min
  \left\{
    \frac{(p+m)\kappa}{2p\pi(\kappa+2p\pi)},
    1-(p+m)\left(\frac{\kappa}{2p\pi}\right)^2,
    \frac{p-m}{2m+3},
    \frac{(p-m-1)\kappa^2}{2p\pi(\kappa+2p\pi)}
  \right\},
\end{equation*}
then we can choose $\chi>0$ so that $\Xi_{m,\chi}$ is included in $\Theta^{\ast}_{m,\nu}$.
\par
If $m=p-1$, then the real part of the top left corner of $\underline{\nabla}^{-}_{p-1,\nu}$ is $1-\frac{2\pi\nu(\kappa+2p\pi)}{|\xi|^2}$, which is slightly to the left of $1+\chi/p$.
Its imaginary part is $\frac{2\pi\nu(2p\pi-\kappa)}{|\xi|^2}$, which is slightly above the real axis.
The bottom left corner of $\underline{\nabla}^{-}_{p-1,\nu}$ is $1-\frac{4p\nu\pi^2}{|\xi|^2}$, which is slightly smaller than $1$.
Its imaginary part is $-\frac{-2\nu\kappa\pi}{|\xi|^2}$, which is below the real axis.
So if we exclude $\underline{\nabla}^{-}_{p-1,\nu}$, the rest is included in $\Theta^{\ast}_{p-1,\nu}$.
See the right picture of Figure~\ref{fig:pentagon}.
\par
The proof is complete.
\end{proof}
\par
We will show that the assumption of Proposition~\ref{prop:Poisson} holds for the function $\psi_N(z):=\varphi_{m,N}(z)-\varphi_{m,N}(\saddle_m)$, the domain $D:=\Xi_{m,\chi}$, and the numbers $a:=m/p$ and $b:=(m+1)/p$, with small $\chi>0$.
Note that the series of functions $\{\psi_{N}(z)\}:=\{\varphi_{m,N}(z)-\varphi_{m,N}(\saddle_m)\}$ uniformly converges to $\psi(z):=\Phi_m(z)-\Phi_m(\saddle_m)$ in $\Xi_{m,\chi}$ for sufficiently small $\chi>0$.
\par
From now we will study properties of $\Phi_{m}(z)$ in the region $\Xi_{m,\chi}$ as if $\chi=0$, taking care of the case where $\chi>0$ if necessary.
\par
Let $P_0,P_1,\dots,P_7$ be points defined as follows, which are already indicated in Figure~\ref{fig:pentagon}.
\begin{align*}
  P_0&:=I_m\cap\text{ real axis},
  \\
  P_1&:=L_{m+1/2}\cap\underline{H},
  \\
  P_2&:=I_{m+1}\cap\underline{H},
  \\
  P_3&:=I_{m+1}\cap\text{ real axis},
  \\
  P_4&:=I_{m+1}\cap\overline{H},
  \\
  P_5&:=L_{m+1}\cap\overline{H},
  \\
  P_6&:=L_{m+1/2}\cap\overline{H},
  \\
  P_7&:=I_{m}\cap\overline{H}.
\end{align*}
Their coordinates are given as follows.
\begin{align*}
  P_0&:=\frac{m}{p},
  \\
  P_1&:=\frac{m+1/2}{p}+\frac{\Im\saddle_m}{2p\pi}\overline{\xi},
  \\
  P_2&:=\frac{m+1}{p}-\Im\saddle_m\i,
  \\
  P_3&:=\frac{m+1}{p},
  \\
  P_4&:=\frac{m+1}{p}+\frac{(p+m)\kappa}{2p^2\pi}\i,
  \\
  P_5&:=\frac{m+1}{p}-\frac{(p+m)\kappa}{4p^3\pi^2}\overline{\xi},
  \\
  P_6&:=\frac{m+1/2}{p}-\frac{(p+m)\kappa}{4p^3\pi^2}\overline{\xi},
  \\
  P_7&:=\frac{m}{p}+\frac{(p+m)\kappa}{2p^2\pi}\i.
\end{align*}
\begin{lem}\label{lem:P}
We have the following inequalities:
\begin{equation*}
  \Re{P_6}<\Re{P_{1}}<\Re{P_5}.
\end{equation*}
\end{lem}
\begin{proof}
It is clear that $\Re{P_{6}}<\Re{P_{1}}$, and so we will show the other inequality.
\par
Since $\Im\saddle_m=\frac{2(m+1)\kappa\pi}{|\xi|^2}$ and $\kappa>1$, we have
\begin{align*}
  \Re{P_{5}}-\Re{P_{1}}
  =&
  \frac{m+1}{p}-\frac{(p+m)\kappa^2}{4p^3\pi^2}
  -
  \left(
    \frac{2m+1}{2p}
    +
    \frac{(m+1)\kappa^2}{p|\xi|^2}
  \right)
  \\
  >&
  \frac{1}{2p}
  -
  \frac{p+m}{4p^3\pi^2}
  -
  \frac{m+1}{4p^3\pi^2}
  \ge
  \frac{1}{2p}-\frac{3p-1}{4p^3\pi^2}
  >
  \frac{1}{2p}-\frac{3}{4p^2\pi^2}
  >0,
\end{align*}
proving the inequality $\Re{P_5}>\Re{P_{1}}$.
\end{proof}
\par
We put
\begin{align*}
  W^{+}_m
  &:=
  \{z\in\Xi_{m,\chi}\mid\Re\Phi_m(z)>\Re\Phi_m(\saddle_m)\},
  \\
  W^{-}_m
  &:=
  \{z\in\Xi_{m,\chi}\mid\Re\Phi_m(z)<\Re\Phi_m(\saddle_m)\}
\end{align*}
for $m=0,1,2,\dots,p-1$.
Recall that in this case, $R_{\pm}$ in Proposition~\ref{prop:Poisson} becomes
\begin{align*}
  R_{+}
  &:=
  \{z\in\Xi_{m,\chi}\mid\Im{z}\ge0,\Re\Phi_m(z)-\Re\Phi_m(\saddle)<2\pi\Im{z}\},
  \\
  R_{-}
  &:=
  \{z\in\Xi_{m,\chi}\mid\Im{z}\le0,\Re\Phi_m(z)-\Re\Phi_m(\saddle)<-2\pi\Im{z}\}.
\end{align*}
\par
In fact, we will show the following lemma, whose proof will be given later.
\begin{lem}\label{lem:Poisson_Phi}
The following hold for $m=0,1,\dots,p-2$.
\begin{enumerate}
\item
The points $m/p$ and $(m+1)/p$ are in $W^{-}_m$,
\item
There is a path $C_{+}$ in $R_{+}$ connecting $m/p$ and $(m+1)/p$,
\item
There is a path $C_{-}$ in $R_{-}$ connecting $m/p$ and $(m+1)/p$.
\end{enumerate}
When $m=p-1$, there exists $\delta>0$ such that the following hold.
\begin{enumerate}
\item
The points $1-1/p$ and $1-\delta$ are in $W^{-}_{p-1}$,
\item
There is a path $C_{+}$ in $R_{+}$ connecting $1-1/p$ and $1-\delta$,
\item
There is a path $C_{-}$ in $R_{-}$ connecting $1-1/p$ and $1-\delta$.
\end{enumerate}
\end{lem}
Note that since $\Xi_{m,\chi}$ is simply connected, both $C_{+}$ and $C_{-}$ are homotopic to the segment $[m/p,(m+1)/p]$ ($[1-1/p,1-\delta]$ if $m=p-1$) in $\Xi_{m,\chi}$ keeping the boundary points fixed.
\par
See Figure~\ref{fig:Phi}.
\begin{figure}[h]
\pic{0.6}{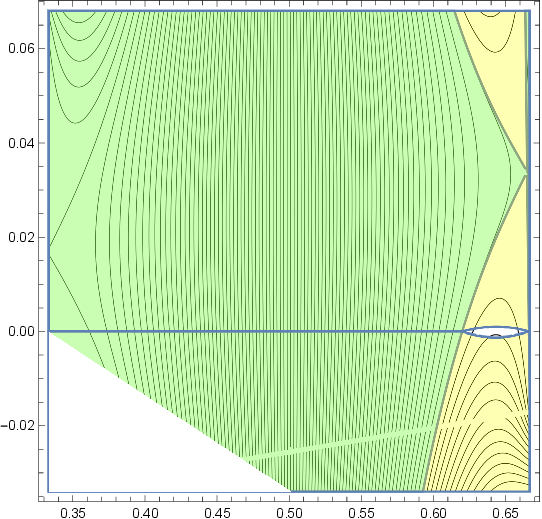}\qquad\pic{0.6}{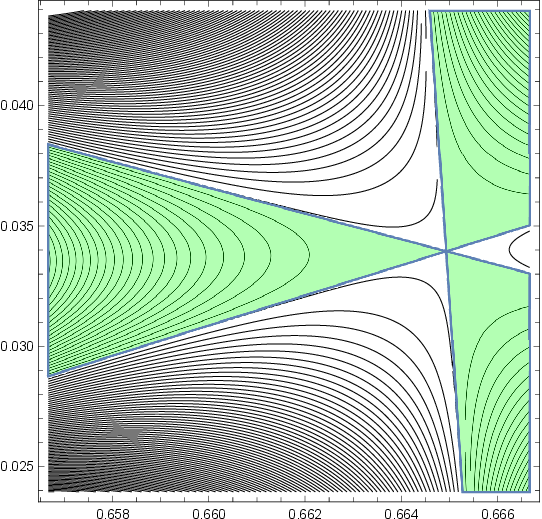}
\caption{The left picture shows a contour plot of $\Re\Phi_1(z)$ in $\Xi_{1,0}$ ($p=3$), where $R_{\pm}$ are indicated by yellow and green and $W_1^{-}$ is indicated by dark green.
The right picture shows a contour plot of $\Re\Phi_1(z)$ in a neighborhood of $\saddle_1$, where $W_1^{-}$ is indicated by green.}
\label{fig:Phi}
\end{figure}
\par
To prove the lemma above, we study the behavior of $\Re\Phi_m$ in $\Xi_{m,0}$ more precisely.
\par
We divide $\Xi_{m,0}$ into six parts by the three lines $L_{m+1}$, $L_{m+1/2}$, and $K_{\saddle}$, where we put
\begin{equation*}
  K_{\saddle}
  :
  \Re(\xi z)=0.
\end{equation*}
We can see that $\saddle_m$ is just the intersection of $L_{m+1}$ and $K_{\saddle}$.
\par
We also introduce the following four points
\begin{align*}
  P_{34}&:=I_{m+1}\cap K_{\saddle},
  \\
  P_{70}&:=I_m\cap K_{\saddle}
\end{align*}
with coordinates
\begin{align*}
  P_{34}&:=\frac{(m+1)\overline{\xi}\i}{2p^2\pi}=\frac{m+1}{p}+\frac{(m+1)\kappa}{2p^2\pi}\i,
  \\
  P_{70}&:=\frac{m\overline{\xi}\i}{2p^2\pi}=\frac{m}{p}+\frac{m\kappa}{2p^2\pi}\i.
\end{align*}
Note that $P_{34}$ is between $P_3$ and $P_4$ (when $p=1$, $P_{34}$ coincides with $P_4$), and that $P_{70}$ is between $P_7$ and $P_0$ (when $m=0$, $P_{70}$ coincides with $P_0$).
\par
As in the proof of Lemma~5.2 in \cite{Murakami:CANJM2023}, we can prove the following lemma.
\begin{lem}\label{lem:hexagon_derivative}
Write $z=x+y\i$ for $z\in\Xi_{m,\chi}$ with $x,y\in\R$.
Then we have
\begin{itemize}
\item
$\frac{\partial\,\Re\Phi_m}{\partial\,y}(z)>0$ if and only if
\begin{align*}
  &\text{$\Re(\xi z)>0$ and $2k\pi<\Im(\xi z)<(2k+1)\pi$
  for some integer $k$},
  \\
  &\quad
  \text{or $\Re(\xi z)<0$ and $(2l-1)\pi<\Im(\xi z)<2l\pi$
  for some integer $l$}
\end{align*}
\item
$\frac{\partial\,\Re\Phi_m}{\partial\,y}(z)<0$ if and only if
\begin{align*}
  &\text{$\Re(\xi z)<0$ and $2k\pi<\Im(\xi z)<(2k+1)\pi$
  for some integer $k$},
  \\
  &\quad
  \text{or $\Re(\xi z)>0$ and $(2l-1)\pi<\Im(\xi z)<2l\pi$
  for some integer $l$.}
\end{align*}
\end{itemize}
See Figure~\ref{fig:derivative}.
\begin{figure}[h]
  \pic{0.3}{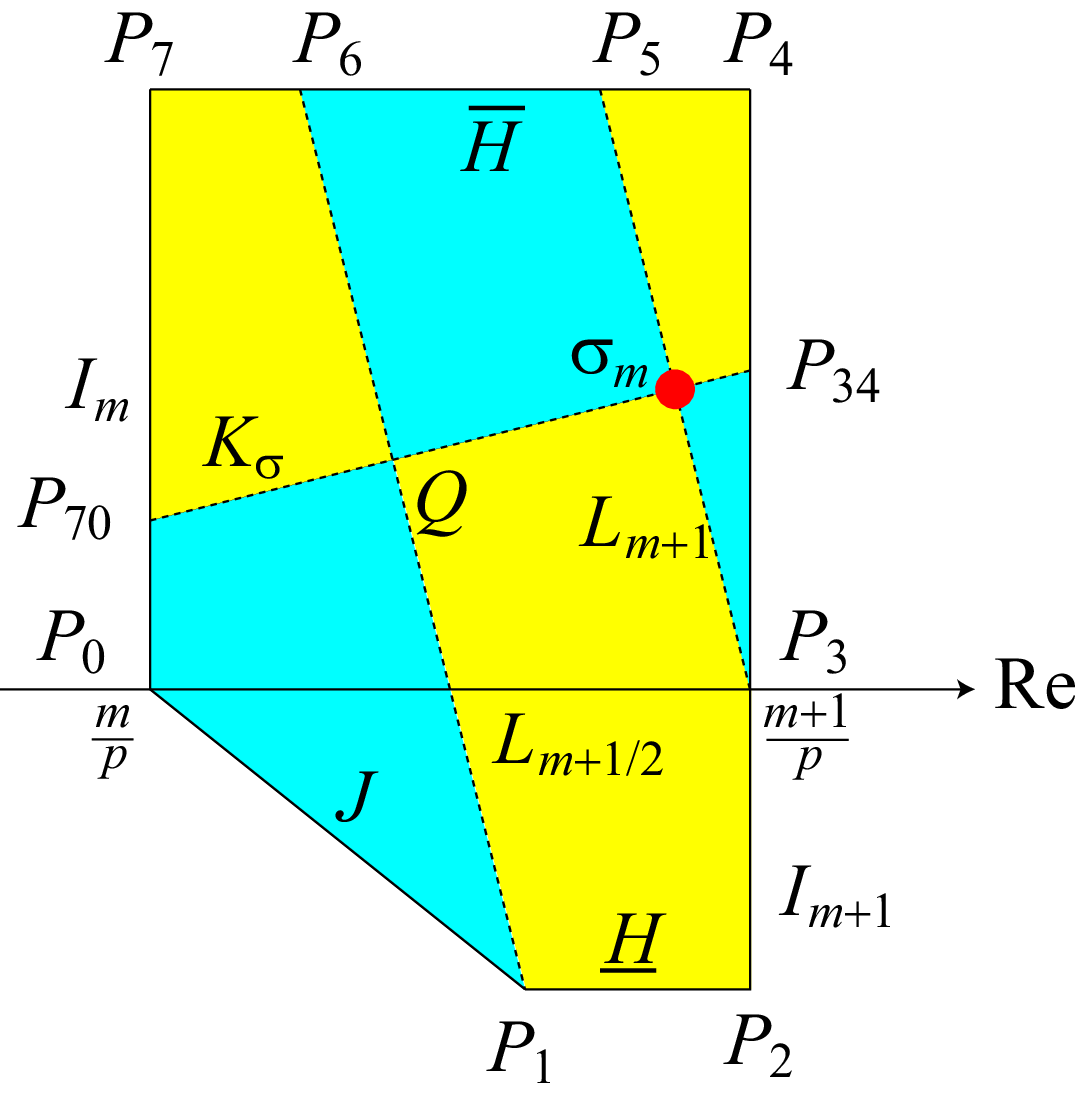}
\caption{In the cyan (yellow, respectively) region, $\Re\Phi_m(z)$ is increasing (decreasing, respectively) with respect to $\Im{z}$.}
\label{fig:derivative}
\end{figure}
\end{lem}
\begin{proof}
From \eqref{eq:Phi'_log}, we have
\begin{equation*}
  \frac{\partial\,\Re\Phi_m(z)}{\partial\,y}
  =
  -\arg\bigl(3-2\cosh(\xi z)\bigr).
\end{equation*}
The right-hand side is positive (negative, respectively) if and only if $\Im\bigl(3-2\cosh(\xi z)\bigr)$ is negative (positive, respectively).
Since $\Im\bigl(3-2\cosh(\xi z)\bigr)=-2\sinh\bigl(\Re(\xi z)\bigr)\sin\bigl(\Im(\xi z)\bigr)$, $\frac{\partial\,\Re\Phi_m(z)}{\partial\,y}$ is positive (negative, respectively) if and only if $\Re(\xi z)>0$ and $2k\pi<\Im(\xi z)<(2k+1)\pi$ for some integer $k$, or $\Re(\xi z)<0$ and $(2l-1)\pi<\Im(\xi z)<2l\pi$ for some integer $l$ ($\Re(\xi z)<0$ and $2k\pi<\Im(\xi z)<(2k+1)\pi$ for some integer $k$, or $\Re(\xi z)>0$ and $(2l-1)\pi<\Im(\xi z)<2l\pi$ for some integer $l$, respectively).
The proof is complete.
\end{proof}
\begin{lem}\label{lem:P60P34}
Let $z$ be a point on the segment $\overline{P_{70}P_{34}}$.
If $z\ne\saddle_m$ is between $\saddle_m$ and $P_{70}$, then $z\in W^{-}_m$.
Moreover, if $z\ne\saddle_m$ is between $\saddle_m$ and $P_{34}$, then $z\in W^{+}_m$.
\end{lem}
\begin{proof}
The segment $\overline{P_{70}P_{34}}\subset K_{\saddle}$ is parametrized as $\frac{2\pi\i}{\xi}t$ ($m+\frac{m\kappa^2}{4p^2\pi^2}\le t\le m+1+\frac{(m+1)\kappa^2}{4p^2\pi^2}$).
From \eqref{eq:Phi'_log} we have
\begin{align*}
  \frac{d}{d\,t}\Re\Phi_m\left(\frac{2\pi\i}{\xi}t\right)
  =&
  \Re\left(
    \frac{2\pi\i}{\xi}\log\left(3-2\cosh(2\pi\i t)\right)
  \right)
  \\
  =&
  \frac{4p\pi^2}{|\xi|^2}\log(3-2\cos{2\pi t})\ge0,
\end{align*}
and the equality holds only when $t=m+1$, that is, $z=\saddle_m$.
So we conclude that if $z\in\overline{P_{70}\saddle_m}\setminus\{\saddle_m\}$, then $\Re\Phi_m(z)<\Re\Phi_m(\saddle_m)$, and that if $z\in\overline{\saddle_m P_{34}}\setminus\{\saddle_m\}$, then $\Re\Phi_m(z)>\Re\Phi_m(\saddle_m)$ as required.
\end{proof}
We can prove a similar result for $\overline{P_3P_5}$.
\begin{lem}\label{lem:P3P5}
Let $z$ be a point on the segment $\overline{P_3P_5}$.
If $z\ne\saddle_m$ is between $\saddle_m$ and $P_{3}$, then $z\in W^{-}_m$.
Moreover, if $z\ne\saddle_m$ is between $\saddle_m$ and $P_{5}$, then $z\in{W}^{+}_m$.
\end{lem}
\begin{proof}
A point on $\overline{P_3P_5}$ is parametrized as $\frac{m+1}{p}-\frac{(p+m)\kappa}{4p^3\pi^2}\overline{\xi}t$ ($0\le t\le1$).
From \eqref{eq:Phi'_log} we have
\begin{align*}
  &\frac{d}{dt}\Re\Phi_m\left(\frac{m+1}{p}-\frac{(p+m)\kappa}{4p^3\pi^2}\overline{\xi}t\right)
  \\
  =&
  -\Re
  \left[
    \frac{(p+m)\kappa}{4p^3\pi^2}\overline{\xi}
    \log
    \left(
      3-2\cosh\left(\frac{(m+1)\xi}{p}-\frac{(p+m)\kappa}{4p^3\pi^2}|\xi|^2t\right)
    \right)
  \right]
  \\
  =&
  \frac{-(p+m)\kappa^2}{4p^3\pi^2}
  \log
  \left(
    3-2\cosh\left(\frac{(m+1)\kappa}{p}-\frac{(p+m)\kappa}{4p^3\pi^2}|\xi|^2t\right)
  \right)
  \ge0,
\end{align*}
where the equality holds when $t=\frac{4(m+1)p^2\pi^2}{(p+m)|\xi|^2}$, which shows that $\Re\Phi_m(z)<\Re\Phi_m(\saddle_m)$ if $z\in\overline{P_3\saddle_m}\setminus\{\saddle_m\}$ and that $\Re\Phi_m(z)>\Re\Phi_m(\saddle_m)$ if $z\in\overline{\saddle_mP_5}\setminus\{\saddle_m\}$, completing the proof.
\end{proof}
So far we have found two directions $\overrightarrow{\saddle_{m}P_{70}}$ and $\overrightarrow{\saddle_{m}P_{3}}$ that go down valleys, and two directions $\overrightarrow{\saddle_{m}P_{34}}$ and $\overrightarrow{\saddle_{m}P_{5}}$ that go up hills.
Since the function $\Phi_m(z)$ is of the form $\Phi_m(\saddle_{m})-\frac{1}{3}\xi^2z^3+\cdots$ from \eqref{eq:Phi_derivatives}, that is, $\saddle_m$ is a saddle point of order two, there should be another pair of valley and hill.
\begin{lem}\label{lem:G}
Let $G$ be the line segment in $\Xi_{m,0}$ that bisects the angle $\angle P_{34}\saddle_{m}P_{5}$.
If $z\in G\setminus\{\saddle_m\}$ is on the same side of $P_{34}$ and $P_5$, then $z\in W^{-}_m$.
If $z\in G\setminus\{\saddle_m\}$ is on the opposite side of $P_{34}$ and $P_5$, and close enough to $\saddle_m$, then $z\in W^{+}_m$.
\end{lem}
\begin{proof}
Since the vector $\overrightarrow{\saddle_{m}P_{34}}$ has the same direction as $\i/\xi$ and the vector $\overrightarrow{\saddle_{m}P_{5}}$ has the same direction as $-1/\xi$, the bisector is parametrized as $\saddle_m+\frac{(\i-1)t}{\xi}$ with $t\in\R$.
Note that if $t>0$, it goes to top right, and that if $t<0$, it goes to bottom left.
\par
From \eqref{eq:Phi'_log}, we have
\begin{equation*}
\begin{split}
  &\frac{d}{dt}\Re\Phi_m\left(\saddle_m+\frac{\i-1}{\xi}t\right)
  \\
  =&
  \Re\left(
    \frac{\i-1}{\xi}\times
    \log\left(3-2\cosh\bigl((\i-1)t\bigr)\right)
  \right)
\end{split}
\end{equation*}
and so $\frac{d}{dt}\Re\Phi_m\left(\saddle_m+\frac{\i-1}{\xi}t\right)=0$ when $t=0$.
Thus, it is sufficient to show that the second derivative of $\Re\Phi_m(\saddle_m+(\i-1)t/\xi)$ is positive when $t<0$ and $|t|$ is small, and that it is negative when $t>0$ and $\saddle_m+\frac{(\i-1)t}{\xi}\in\Xi_{m,0}$.
\par
From \eqref{eq:Phi''}, we have
\begin{align*}
  &\frac{d^2}{dt^2}\Re\Phi_m\left(\saddle_m+\frac{\i-1}{\xi}t\right)
  \\
  =&
  \Re\left(
    \frac{(\i-1)^2}{\xi^2}\times\frac{\xi(e^{-(\i-1)t}-e^{(\i-1)t})}{3-e^{(\i-1)t}-e^{-(\i-1)t}}
  \right)
  \\
  =&
  \frac{1}{|\xi|^2}
  \Re\bigl((-4p\pi-2\kappa\i)\lambda(t)\bigr)
  \\
  =&
  \frac{1}{|\xi|^2}
  \bigl(-4p\pi\Re\lambda(t)+2\kappa\Im\lambda(t)\bigr),
\end{align*}
where we put $\lambda(t):=\frac{e^{-(\i-1)t})-e^{(\i-1)t}}{3-e^{(\i-1)t}-e^{-(\i-1)t}}$.
We have
\begin{align*}
  \lambda(t)
  =&
  \frac{2\sinh{t}\cos{t}-2\i\cosh{t}\sin{t}}{3-2\cosh{t}\cos{t}+2\i\sinh{t}\sin{t}}
  \\
  =&
  \frac{
  2(\sinh{t}\cos{t}-\i\cosh{t}\sin{t})}{(3-2\cosh{t}\cos{t})^2+4\sinh^2{t}\sin^2{t}}
  \\
  &\times
  (3-2\cosh{t}\cos{t}-2\i\sinh{t}\sin{t})
  \\
  =&
  \frac{2\sinh{t}(3\cos{t}-2\cosh{t})+2\i\sin{t}(2\cos{t}-3\cosh{t})}
       {(3-2\cosh{t}\cos{t})^2+4\sinh^2{t}\sin^2{t}}.
\end{align*}
\par
Therefore if $t$ is negative and $|t|$ is small enough, then $\Re\lambda(t)<0$ and $\Im\lambda(t)>0$, and so in this case $\frac{d^2}{dt^2}\Re\Phi_m\left(\saddle_m+\frac{\i-1}{\xi}t\right)>0$.
\par
Next, we consider the case where $t>0$.
\par
Since $\Re\left(\saddle_m+\frac{(\i-1)t}{\xi}\right)=\frac{1}{|\xi|^2}(4(m+1)p\pi^2+(2p\pi-\kappa)t)$, a point in $G$ that is between $\saddle_m$ and $I_{m+1}$ is parametrized as $\saddle_m+\frac{(\i-1)t}{\xi}$ with $0<t<\frac{(m+1)\kappa^2}{p(2p\pi-\kappa)}$.
Since $\frac{(m+1)\kappa^2}{p(2p\pi-\kappa)}\le\frac{\kappa^2}{2p\pi-\kappa}\le\frac{\kappa^2}{2\pi-\kappa}$, it is sufficient to prove $\frac{d^2}{dt^2}\Re\Phi_m\left(\saddle_m+\frac{\i-1}{\xi}t\right)<0$ for $0<t<\frac{\kappa^2}{2\pi-\kappa}$.
\par
Since $3\cos\left(\frac{\kappa^2}{2\pi-\kappa}\right)-2\cosh\left(\frac{\kappa^2}{2\pi-\kappa}\right)=0.924...$, and the function $3\cos{t}-2\cosh{t}$ is monotonically decreasing when $t>0$, we see that $\Re\lambda(t)>0$ for $0<t<\frac{\kappa^2}{2\pi-\kappa}$.
We can easily see that $\Im\lambda(t)<0$ for $t>0$, and so we conclude that $\frac{d^2}{dt^2}\Re\Phi\left(\saddle_m+\frac{\i-1}{\xi}t\right)<0$.
\par
This completes the proof.
\end{proof}
\begin{rem}\label{rem:G}
The imaginary part of the intersection of $G$ with $I_{m+1}$ is $\frac{(m+1)\kappa}{p(2p\pi-\kappa)}$, which is smaller than the imaginary part of $\overline{H}$ when $p>1$.
This is because
\begin{align*}
  \frac{(p+m)\kappa}{2p^2\pi}
  -
  \frac{(m+1)\kappa}{p(2p\pi-\kappa)}
  =&
  \frac{\bigl(2p\pi(p-1)-(p+m)\kappa\bigr)\kappa}{2p^2\pi(2p\pi-\kappa)}
  \\
  >&
  \frac{\bigl(2p\pi(p-1)-(2p-1)\bigr)\kappa}{2p^2\pi(2p\pi-\kappa)}
  =
  \frac{\bigl((2p\pi-1)(p-1)-p)\bigr)\kappa}{2p^2\pi(2p\pi-\kappa)},
\end{align*}
which is positive when $p>1$, where we use the inequalities $\kappa<1$ and $m\le p-1$.
So $G$ intersects with the segment $\overline{P_4P_{34}}$.
\par
If $p=1$, $G$ intersects with the segment $\overline{P_4P_5}$.
\par
Note that $G$ does not intersect with $L_{m+1/2}$ in $\Xi_{m,0}$.
This is because the intersection between $G$ and $L_{m+1/2}$ is $\frac{\pi+(2m+1)\pi\i}{\xi}$ whose imaginary part is less than $-\Im\saddle_m$.
\end{rem}
There are more line segments that are included in $W^{-}_m$.
\begin{lem}\label{lem:segments_W-m}
The line segments $\overline{P_{6}P_{1}}$, $\overline{P_{0}P_{70}}$, and $\overline{P_{0}P_{1}}$ are in $W^{-}_m$.
\end{lem}
\begin{proof}
\par
A point on the segment $\overline{P_{6}P_{1}}$ is parametrized as $\frac{m+1/2}{p}+\frac{\overline{\xi}}{2p\pi}t$ ($-\frac{(p+m)\kappa}{2p^2\pi}\le t\le\Im\saddle_m$).
We have
\begin{align*}
  \frac{d}{d\,t}\Re\Phi_m\left(\frac{m+1/2}{p}+\frac{\overline{\xi}}{2p\pi}t\right)
  =&
  \Re
  \left(
    \frac{\overline{\xi}}{2p\pi}
    \log
    \left(
      3-2\cosh\left(\frac{m+1/2}{p}\xi+\frac{|\xi|^2t}{2p\pi}\right)
    \right)
  \right)
  \\
  =&
  \frac{\kappa}{2p\pi}
  \log
  \left(
    3+2\cosh\left(\frac{(m+1/2)\kappa}{p}+\frac{|\xi|^2t}{2p\pi}\right)
  \right)>0.
\end{align*}
From Lemma~\ref{lem:P1} below we know that $\Re\Phi_m(P_{1})<\Re\Phi_m(\saddle_m)$.
It follows that $\overline{P_{6}P_{1}}\subset W^{-}_m$.
\par
From Lemma~\ref{lem:hexagon_derivative}, $\Re\Phi_m(z)$ is increasing with respect to $\Im{z}$ in the quadrilateral $P_{70}P_0P_1Q$, where $Q$ is the crossing between $K_{\saddle}$ and $L_{m+1/2}$.
Since the upper segments $\overline{P_{70}Q}$ and $\overline{QP_1}$ are in $W^{-}_m$, so are the lower segments $\overline{P_{70}P_0}$ and $\overline{P_0P_1}$.
\end{proof}
\begin{lem}\label{lem:P1}
The point $P_1$ is in $W^{-}_m$.
\end{lem}
\begin{proof}
The following proof is similar to that of \cite[Lemma~5.3]{Murakami:CANJM2023}.
\par
Since $P_1$ is on $L_{m+1/2}$, we have $\Im\bigl(\xi(P_1-2m\pi\i/\xi)\bigr)=\pi$ and so $P_1-\frac{2m\pi\i}{\xi}$ is on $L_{1/2}$.
So, from \eqref{eq:F_K_L0_L1} we have
\begin{align*}
  &\xi\Phi_m(P_1)-\xi\Phi_m(\saddle_m)
  \\
  =&
  \Li_2\left(-e^{-\kappa-\frac{(4m+3)\kappa}{2p}}\right)
  -
  \Li_2\left(-e^{-\kappa+\frac{(4m+3)\kappa}{2p}}\right)
  \\
  &+
  \frac{(4m+3)\kappa^2}{2p}-\kappa\pi\i
\end{align*}
Its real part is
\begin{equation*}
  \Li_2\left(-e^{-\kappa-\frac{(4m+3)\kappa}{2p}}\right)
  -
  \Li_2\left(-e^{-\kappa+\frac{(4m+3)\kappa}{2p}}\right)
  +
  \frac{(4m+3)\kappa^2}{2p}
\end{equation*}
and its imaginary part is $-\kappa\pi$.
\par
Therefore we have
\begin{equation}\label{eq:P1_saddle}
\begin{split}
  &\frac{|\xi|^2}{\kappa}
  \bigl(\Re\Phi_m(P_1)-\Re\Phi_m(\saddle_m)\bigr)
  \\
  =&
  \Re\bigl(\xi\Phi_m(P_1)-\xi\Phi_m(\saddle_m)\bigr)
  +
  \frac{2p\pi}{\kappa}
  \Im\bigl(\xi\Phi_m(P_1)-\xi\Phi_m(\saddle_m)\bigr)
  \\
  =&
  \Li_2\left(-e^{-\kappa-\frac{(4m+3)\kappa}{2p}}\right)
  -
  \Li_2\left(-e^{-\kappa+\frac{(4m+3)\kappa}{2p}}\right)
  +
  \frac{(4m+3)\kappa^2}{2p}
  -2p\pi^2,
\end{split}
\end{equation}
which is increasing with respect to $m$, fixing $p$.
When $m=p-1$, \eqref{eq:P1_saddle} equals
\begin{equation}\label{eq:P1_saddle2}
  \Li_2\left(-e^{-\kappa-\frac{(4p-1)\kappa}{2p}}\right)
  -
  \Li_2\left(-e^{-\kappa+\frac{(4p-1)\kappa}{2p}}\right)
  +
  \frac{(4p-1)\kappa^2}{2p}
  -2p\pi^2.
\end{equation}
Its derivative with respect to $p$ is
\begin{equation*}
  \frac{\kappa}{2p^2}
    \log\left(3+2\cosh\left(\kappa\bigl(2-\frac{1}{2p}\bigr)\right)\right)
  -2\pi^2,
\end{equation*}
which is less than $\log\bigl(3+2\cosh(2\kappa)\bigr)-2\pi^2=\log(10)-2\pi^2<0$.
Since \eqref{eq:P1_saddle2} equals $-17.2195...$ when $p=1$, we conclude that \eqref{eq:P1_saddle} is negative, proving the lemma.
\end{proof}
\begin{rem}
One can also show that the polygonal line $\overline{P_{70}P_7P_{6}}$ is in $W^{-}_m$.
\end{rem}
The results in Lemmas~\ref{lem:P60P34}, \ref{lem:P3P5}, \ref{lem:G}, and \ref{lem:segments_W-m} are summarized in Figure~\ref{fig:cross}.
\begin{figure}[h]
\pic{0.4}{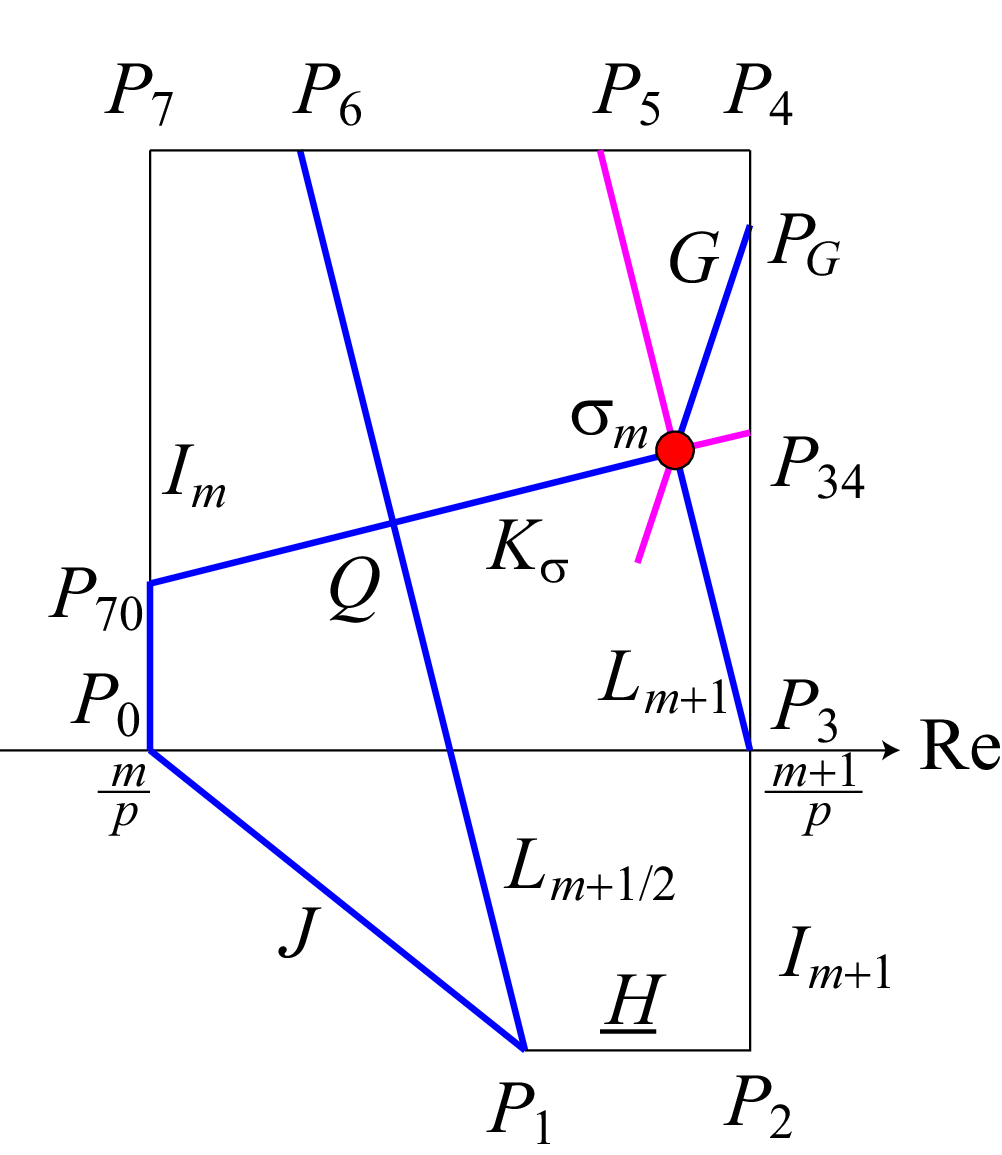}
\caption{The blue (magenta, respectively) lines are included in $W^{-}_m$ ($W^{+}_{m}$, respectively).}
\label{fig:cross}
\end{figure}
\par
Now, we are ready to prove Lemma~\ref{lem:Poisson_Phi}.
\begin{proof}[Proof of Lemma~\ref{lem:Poisson_Phi}]
First, suppose that $m<p-1$.
\begin{enumerate}
\item
Since $m/p=P_0$ and $(m+1)/p=P_3$, it follows from Figure~\ref{fig:cross} that these points are in $W^{-}_m$.
\item
Consider the polygonal line $C_{+}:=\overline{P_0P_{70}\saddle_mP_3}$.
From Figure~\ref{fig:cross}, it is in $W^{-}_m$ and in the upper half plane $\{z\in\C\mid\Im{z}\ge0\}$.
So, it is contained in $R_{+}$.
\item
From Figure~\ref{fig:cross}, we know that the line segment $J$ is in $W^{-}_m$ and in the lower half plane $\{z\in\C\mid\Im{z}\le0\}$.
This implies that $J\subset R_{-}$.
\par
We will show that the segments $\overline{P_{1}P_{2}}$ and $\overline{P_{2}P_{3}}$ are also in $R_{-}$.
\par
We first show that $\overline{P_{1}P_{2}}\subset R_{-}$, that is, $\Re\Phi_m(z)-\Re\Phi_m(\saddle_m)<-2\pi\Im{z}$, if $z\in\overline{P_{1}P_{2}}$.
From the proof of Lemma~\ref{lem:hexagon_derivative}, we see that $-\pi<\frac{\partial\,\Re\Phi_m(z)}{\partial\,y}<0$ if $z=x+y\i$ is in the pentagonal region $QP_{1}P_{2}P_{3}\saddle_m$.
We also know that if $z\in\overline{Q\saddle_mP_{3}}$, then $\Re\Phi_m(z)-\Re\Phi_m(\saddle_m)\le0$.
Since the difference between the imaginary part of the point on $\overline{Q\saddle_mP_{3}}$ and that of the point on $\overline{P_{1}P_{2}}$ is less than or equal to $2\Im\saddle_m$, it follows that for $z\in\overline{P_{1}P_{2}}$, we have $\Re\Phi_m(z)-\Re\Phi_m(\saddle_m)<\pi\times(2\Im\saddle_m)=-2\pi\times\Im{z}$.
\par
Next we will show that $\overline{P_2P_3}\subset R_{-}$.
To prove this, consider the function $r(y):=\Re\Phi_m\bigl((m+1)/p+y\i\bigr)-\Re\Phi_m(\saddle_m)+2\pi y$.
Since $\frac{d}{d\,y}r(y)=\frac{\partial}{\partial\,y}\Re\Phi_m\bigl((m+1)/p+y\i\bigr)+2\pi>0$ and $r(0)<0$ from the argument above, we conclude that $r(y)<0$ if $y\ge-\Im\saddle_m$.
So if $z\in\overline{P_2P_3}$, then $z\in R_{-}$.
\par
Therefore, we can put $C_{-}:=\overline{P_0P_1P_2P_3}\subset R_{-}$.
\end{enumerate}
Next, we consider the case where $m=p-1$.
In this case, we can push $P_3$ slightly to the left to avoid $\diamond_{\nu}$.
Accordingly, we move the segments $\overline{\saddle_mP_3}$ and $\overline{P_2P_3}$ slightly.
\par
The proof is complete.
\end{proof}
Therefore we can apply Proposition~\ref{prop:Poisson} to the series of functions $\psi_{N}(z)=\varphi_{m,N}(z)-\varphi_{m,N}(\saddle_m)$.
We conclude that
\begin{equation}\label{eq:sum_integral_Phi}
\begin{split}
  &\frac{1}{N}
  e^{-N\varphi_{m,N}(\saddle_m)}
  \sum_{m/p\le k/N\le(m+1)/p}
  e^{N\varphi_{m,N}(k/N)}
  \\
  =&
  e^{-N\varphi_{m,N}(\saddle_m)}
  \int_{m/p}^{(m+1)/p}
  e^{N\varphi_{m,N}(z)}
  \,dz
  +
  O(e^{-\varepsilon_mN})
\end{split}
\end{equation}
for $\varepsilon_m>0$ if $m<p-1$, and
\begin{equation}\label{eq:sum_integral_Phi_p-1}
\begin{split}
  &\frac{1}{N}
  e^{-N\varphi_{p-1,N}(\saddle_m)}
  \sum_{(p-1)/p\le k/N\le1-\delta}
  e^{N\varphi_{p-1,N}(k/N)}
  \\
  =&
  e^{-N\varphi_{p-1,N}(\saddle_{p-1})}
  \int_{(p-1)/p}^{1-\delta}
  e^{N\varphi_{p-1,N}(z)}
  \,dz
  +
  O(e^{-\varepsilon_{p-1}N})
\end{split}
\end{equation}
for $\varepsilon_{p-1}$.

%% file: saddle.tex
\section{Saddle point method of order two}\label{sec:saddle}
We would like to know the asymptotic behavior of the integrals appearing in the right hand sides of \eqref{eq:sum_integral_Phi} and \eqref{eq:sum_integral_Phi_p-1} by using the saddle point method of order two.
\par
To describe it, let us consider a holomorphic function $\eta(z)$ in a domain $D\ni O$ with $\eta(0)=\eta'(0)=\eta''(0)=0$ and $\eta^{(3)}(0)\ne0$.
Write $\eta^{(3)}(0)=6re^{\theta\i}$ with $r>0$ and $-\pi<\theta\le\pi$.
Then $\eta(z)$ is of the form $\eta(z)=re^{\theta\i}z^3\times g(z)$, where $g(z)$ is holomorphic with $g(0)=1$.
The origin is called a saddle point of $\Re{\eta(z)}$ of order two.
We put $V:=\{z\in D\mid\Re{\eta(z)}<0\}$.
\par
\label{page:E_G}
There exists a small disk $\hat{D}\subset D$ centered at $O$, where we can define a cubic root $g^{1/3}(z)$ of $g(z)$ such that $g^{1/3}(0)=1$.
Put $G(z):=zg^{1/3}(z)$ in $\hat{D}\subset D$.
We can choose $\hat{D}$ so that $G$ gives a bijection from $\hat{D}$ to $E:=G(\hat{D})$ from the inverse function theorem because $G'(0)=1$.
Since $re^{\theta\i}G(z)^3=\eta(z)$, the function $G$ also gives a bijection from the region $V\cap\hat{D}$ to the region $U:=\{w\in E\mid\Re(re^{\theta\i}w^3)<0\}$.
\par
The region $U$ splits into the three connected components (valleys) $U_1$, $U_2$, and $U_3$.
Therefore the region $V\cap\hat{D}$ also splits into three valleys $V_k:=G^{-1}(U_k)$ ($k=1,2,3$) of $\Re{\eta(z)}$.
\begin{rem}\label{rem:ray}
Since $G'(0)=1$, and $U_{k}$ contains the ray $\{w\in E\mid w=se^{((2k-1)\pi-\theta)\i/3},s>0\}$ as a bisector, $V_{k}$ also contains a segment $\{z\in\hat{D}\mid z=te^{((2k-1)\pi-\theta)\i/3}\,(\text{$t>0$: small})\}$.
\end{rem}
The following is the statement of the saddle point method of order two.
\begin{prop}\label{prop:saddle}
Let $\eta(z)$ be a holomorphic function in a domain $D\ni O$ with $\eta(0)=\eta'(0)=\eta''(0)=0$ and $\eta^{(3)}(0)\ne0$.
Write $\eta^{(3)}(0)=6re^{\theta\i}$ with $r>0$ and $-\pi<\theta\le\pi$.
Put $V:=\{z\in D\mid\Re{\eta(z)}<0\}$ and define $V_k$ \rm{(}$k=1,2,3$\rm{)} as above.
Let $C\subset D$ be a path from $a$ to $b$ with $a,b\in V$.
\par
We assume that there exist paths $P_{k}\subset V\cup\{O\}$ from $a$ to $O$, and $P_{k+1}\subset V\cup\{O\}$ from $O$ to $b$ such that the following hold:
\begin{enumerate}
\item
$(P_{k}\cap\hat{D})\setminus\{O\}\subset V_{k}$,
\item
$(P_{k+1}\cap\hat{D})\setminus\{O\}\subset V_{k+1}$,
\item
the path $P_{k}\cup P_{k+1}$ is homotopic to $C$ in $D$ keeping $a$ and $b$ fixed,
\end{enumerate}
where $\hat{D}\in O$ is a disk as above.
\par
Let $\{h_N(z)\}$ be a series of holomorphic function in $D$ that uniformly converges to a holomorphic function $h(z)$ with $h(0)\ne0$.
We also assume that $\left|h_N(z)\right|$ is bounded irrelevant to $z$ or $N$.
Then we have
\begin{equation}\label{eq:saddle}
  \int_{C}h_N(z)e^{N\eta(z)}\,dz
  =
  \frac{h(0)\Gamma(1/3)\i}{\sqrt{3}r^{1/3}N^{1/3}}\omega^{k}e^{-\theta\i/3}
  \left(1+O(N^{-1/3})\right)
\end{equation}
as $N\to\infty$, where $\omega:=e^{2\pi\i/3}$.
\end{prop}
The proposition may be well known to experts, but we give a proof in Appendix~\ref{sec:saddle_proof} because the author is not an expert and could not find appropriate references.
\par
\par
We will apply Proposition~\ref{prop:saddle} to
\begin{itemize}
\item
$\eta(z):=\Phi_{m}(z+\saddle_m)-\Phi_m(\saddle_m)$,
\item
$D:=\{z\in\C\mid z+\saddle_m\in\Xi_{m,\chi}\}$,
\item
$h_N(z):=\exp\left[N\bigl(\varphi_{m,N}(z+\saddle_m)-\Phi_{m}(z+\saddle_m)\bigr)\right]$, and
\item
$C:=[m/p-\saddle_m,(m+1)/p-\saddle_m]$ for $m<p-1$, and $C:=[(p-1)/p,1-\delta]$ for $m=p-1$, where $\delta$ is a positive small number (see Lemma~\ref{lem:Poisson_Phi}).
\end{itemize}
Note that $\eta(0)=\eta'(0)=\eta''(0)=0$, $\eta^{(3)}(0)=-2\xi^2\ne0$, $h(z):=\lim_{N\to\infty}h_N(z)=1$, and that $V$ is equal to the region $\{z\in\C\mid z+\sigma_m\in W^{-}_m\}$.
\par
Since $\eta(z)=-\frac{\xi^2}{3}z^3+\dots$, we can define a holomorphic function $g(z):=-\frac{3\eta(z)}{\xi^2z^3}$ so that $g(0)=1$.
Put $G(z):=zg^{1/3}(z)$ as above.
Let $\hat{D}\subset D$ be a small disk centered at $0$ such that the function $G(z)$ is a bijection.
Then the region $V$ splits into three valleys $V_{m,1}$, $V_{m,2}$, and $V_{m,3}$.
From Remark~\ref{rem:ray}, the argument of the bisector of $V_{m,k}$ is given by $(2k-1)\pi/3-\theta/3\pmod{2\pi}$ ($k=1,2,3$), where $\theta:=\arg(-2\xi^2)=-\pi+2\arctan(2p\pi/\kappa)$.
So the valley $V_{m,k}$ is approximated by the small sector
\begin{equation*}
  \{z\in\C\mid z=te^{\tau\i},|\tau-\alpha_k|<\frac{\pi}{6},\text{$t>0$: small}\},
\end{equation*}
where we put
\begin{equation}\label{eq:alpha}
\begin{split}
  \alpha_1
  &:=
  -\frac{2}{3}\arctan(2p\pi/\kappa)+\frac{2\pi}{3},
  \\
  \alpha_2
  &:=
  -\frac{2}{3}\arctan(2p\pi/\kappa)-\frac{2\pi}{3},
  \\
  \alpha_3
  &:=
  -\frac{2}{3}\arctan(2p\pi/\kappa).
\end{split}
\end{equation}
Note that since $\pi/4<\arctan(2p\pi/\kappa)<\pi/2$, we have $\pi/3<\alpha_1<\pi/2$, $-\pi<\alpha_2<-5\pi/6$, and $-\pi/3<\alpha_3<-\pi/6$.
See Figure~\ref{fig:valleys}.
\begin{figure}[h]
\begin{center}\pic{0.3}{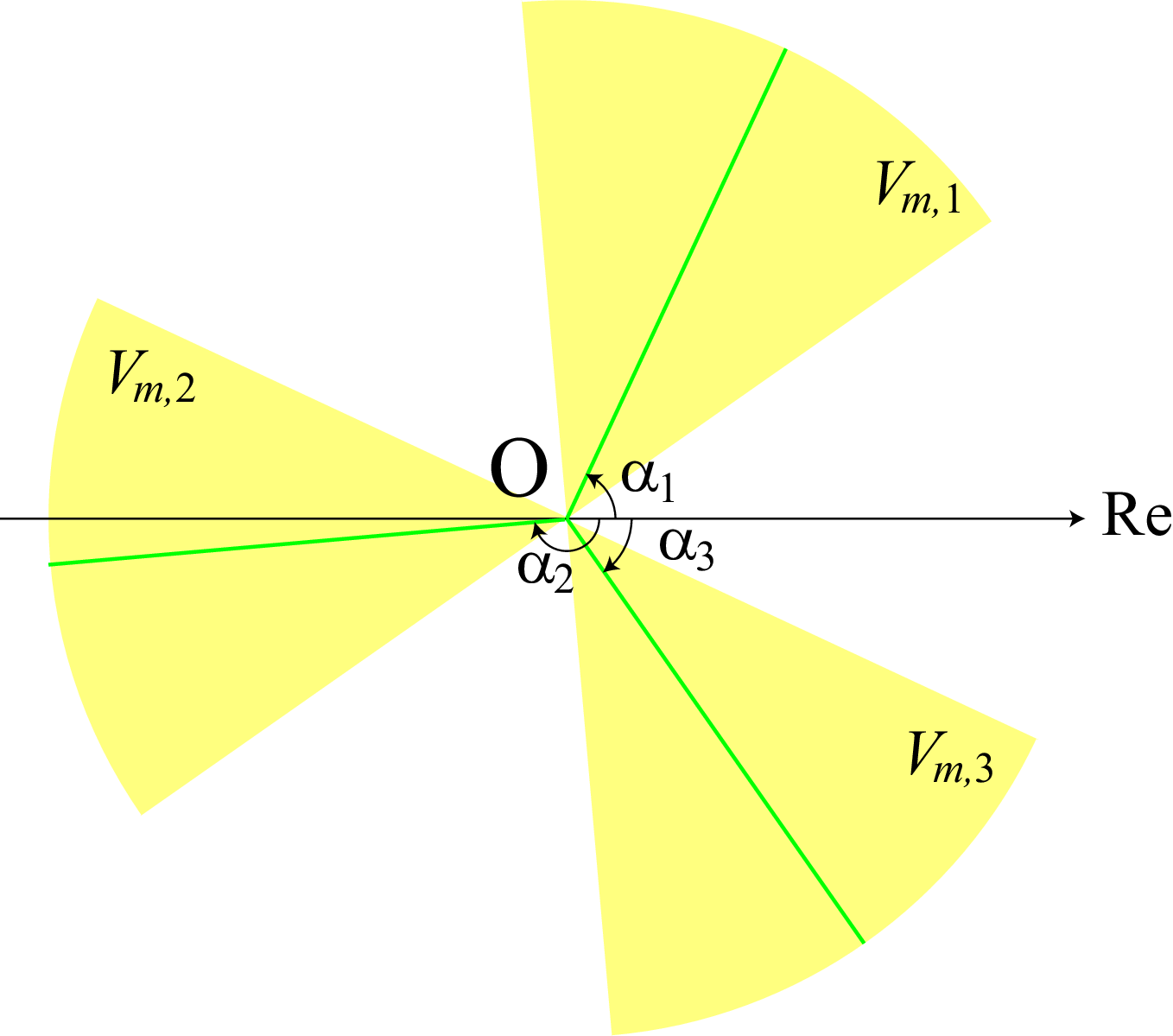}\end{center}
\caption{The yellow regions indicates the valleys $V_{m,1}$, $V_{m,2}$, and $V_{m,3}$.}
\label{fig:valleys}
\end{figure}
\begin{rem}
Denote by $P_{G}$ the intersection between $G$ and the boundary of $\Xi_{m,0}$ as in Figure~\ref{fig:cross}.
Note that $P_{G}\subset I_{m+1}$ if $m<p-1$ and $P_{G}\subset\overline{H}$ if $p=1$ from Remark~\ref{rem:G}.
The arguments of $\overrightarrow{\saddle_{m}P_{G}}$, $\overrightarrow{\saddle_{m}P_{70}}$, and $\overrightarrow{\saddle_{m}P_{3}}$ are
\begin{equation}\label{eq:beta}
\begin{split}
  \beta_1
  &:=
  -\arctan(2p\pi/\kappa)+\frac{3}{4}\pi,
  \\
  \beta_2
  &:=
  -\arctan(2p\pi/\kappa)-\frac{\pi}{2},
  \\
  \beta_3
  &:=
  -\arctan(2p\pi/\kappa),
\end{split}
\end{equation}
respectively, because the vector $\overrightarrow{\saddle_{m}P_{70}}$ has the same direction as $-\i/\xi$, the vector $\overrightarrow{\saddle_{m}P_3}$ has the same direction as $1/\xi$, and $G$ is their bisection.
\par
Since $\pi/4<\arg(2p\pi/\kappa)<\pi/2$, we can see
\begin{align*}
  \alpha_1-\beta_1
  &=
  -\frac{\pi}{12}+\frac{\arctan(2p\pi/\kappa)}{3},
  \\
  \beta_2-\alpha_2
  &=
  \frac{\pi}{6}-\frac{\arctan(2p\pi/\kappa)}{3},
  \\
  \alpha_3-\beta_3
  &=
  \frac{\arctan(2p\pi/\kappa)}{3},
\end{align*}
and
\begin{equation*}
  \alpha_2<\beta_2<\beta_3<\alpha_3<\beta_1<\alpha_1,
\end{equation*}
where $\alpha_k$ ($k=1,2,3$) are given in \eqref{eq:alpha}.
We also conclude that $|\alpha_k-\beta_k|<\pi/6$, that is, $\overline{\saddle_mP_G}$ is in the valley $V_{m,1}$, $\overline{\saddle_mP_{50}}$ is in the valley $V_{m,2}$ and $\overline{\saddle_mP_3}$ is in the valley $V_{m,3}$.
See Figure~\ref{fig:saddle_point}.
\begin{figure}[h]
\pic{0.3}{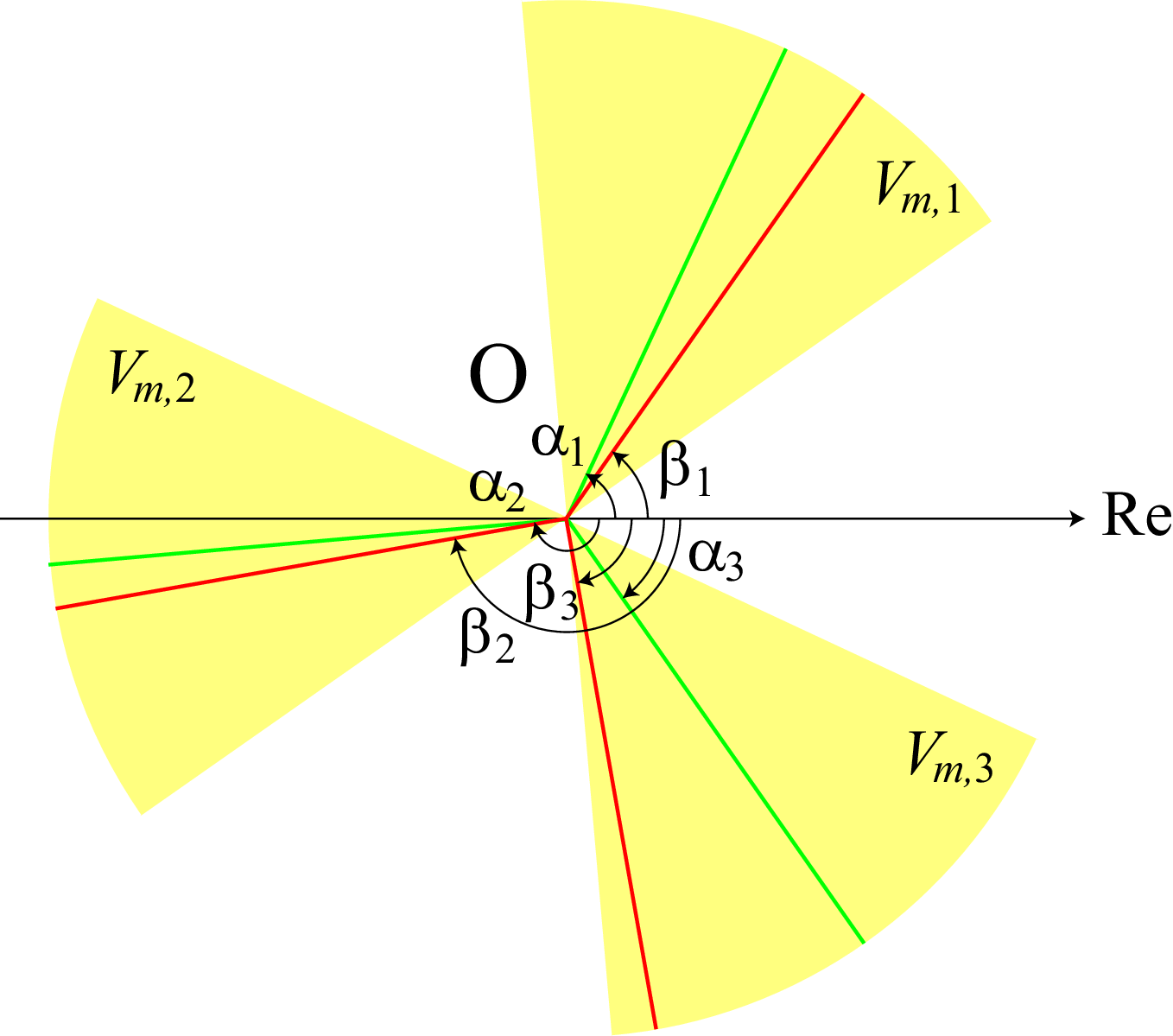}
\caption{The yellow regions indicate the valleys.}
\label{fig:saddle_point}
\end{figure}
\end{rem}
\par
We need to show that the assumption of Proposition~\ref{prop:saddle} holds, that is, we will show the following lemma.
\begin{lem}\label{lem:saddle_Phi}
First suppose that $m=0,1,2,\dots,p-2$.
If a disk $\tilde{D}\subset\Xi_{m,\chi}$ centered at $\saddle_m$ is small enough, then the following hold
\begin{enumerate}
\item
There exists a path $\rho_{2}\subset W^{-}_m\cup\{\saddle_m\}$ connecting $m/p$ and $\saddle_m$ such that $(\rho_{2}\cap\tilde{D})\setminus\{\saddle_m\}\subset V_{m,2}$,
\item
There exists a path $\rho_{3}\subset W^{-}_m\cup\{\saddle_m\}$ connecting $\saddle_m$ and $(m+1)/p$ such that $(\rho_{3}\cap\tilde{D})\setminus\{\saddle_m\}\subset V_{m,3}$.
\end{enumerate}
Next, suppose that $m=p-1$.
If a disk $\tilde{D}\subset\Xi_{p-1,\chi}$ centered at $\saddle_{p-1}$ is small enough, then the following hold:
\begin{enumerate}
\item
There exists a path $\rho_{2}\subset W^{-}_{p-1}\cup\{\saddle_{p-1}\}$ connecting $1-1/p$ and $\saddle_{p-1}$ such that $(\rho_{2}\cap\tilde{D})\setminus\{\saddle_{p-1}\}\subset V_{p-1,2}$,
\item
There exists a path $\rho_{3}\subset W^{\nu}_m\cup\{\saddle_{p-1}\}$ connecting $\saddle_{p-1}$ and $1-\delta$ such that $(\rho_{3}\cap\hat{D})\setminus\{\saddle_{p-1}\}\subset V_{p-1,3}$,
\end{enumerate}
\end{lem}
Note again that since $\Xi_{m,\chi}$ is simply-connected, the path $\rho_{2}\cup\rho_{3}$ is homotopic to the interval $[m/p,(m+1)/p]$ ($[1-1/p,1-\delta]$, respectively) if $m<p-1$ (if $m=p-1$, respectively).
\begin{proof}
The proof is essentially the same for both cases $m<p-1$ and $m=p-1$.
\begin{enumerate}
\item
The path $\rho_2:=\overline{P_0P_{70}\saddle_m}$ is a required one for $m=0,1,2,\dots,p-1$.
\item
When $m<p-1$ consider the path $\rho_3:=\overline{\saddle_m P_3}$, and when $m=p-1$ push it a little more to the left near the point $1$.
\end{enumerate}
This completes the proof.
\end{proof}
\par
If $m<p-1$, we apply Proposition~\ref{prop:saddle} to $\eta(z)=\Phi_m(z+\saddle_m)-\Phi_m(\saddle_m)$, $h_N(z)=\exp\left[N\bigl(\varphi_{m,N}(z+\saddle_m)-\Phi_{m}(z+\saddle_m)\bigr)\right]$, $C:=[m/p-\saddle_m,(m+1)/p-\saddle_m]$, and $k=2$.
Noting that $h_N(z)$ converges to $1$, that $\eta^{(3)}(0)=-2\xi^2=2|\xi|^2e^{\theta\i}$ with $\theta=-\pi+2\arctan(2p\pi/\kappa)$ from the argument above, we have
\begin{equation*}
\begin{split}
  &\int_{m/p}^{(m+1)/p}
  e^{N\bigl(\varphi_{m,N}(z)-\Phi_{m}(\saddle_m)\bigr)}\,dz
  \\
  =&
  \int_{C}
  e^{N\bigl(\varphi_{m,N}(z+\saddle_m)-\Phi_{m}(\saddle_m)\bigr)}\,dz
  =
  \int_{C}
  h_N(z)e^{N\eta(z)}\,dz
  \\
  =&
  \frac{\Gamma(1/3)\i}{\sqrt{3}(|\xi|^2/3)^{1/3}N^{1/3}}
  \omega^2e^{\pi\i/3-2\arctan(2p\pi/\kappa)\i/3}
  \bigl(1+O(N^{-1/3})\bigr)
  \\
  =&
  \frac{\Gamma(1/3)\i}{3^{1/6}|\xi|^{2/3}N^{1/3}}
  e^{-\bigl(\pi+2\arctan(2p\pi/\kappa)\bigr)\i/3}
  \bigl(1+O(N^{-1/3})\bigr)
\end{split}
\end{equation*}
as $N\to\infty$.
Similarly, if $m=p-1$, putting $C:=[1-1/p-\saddle_m,1-\saddle_m-\delta]$, we have
\begin{equation*}
\begin{split}
  &\int_{1-1/p}^{1-\delta}
  e^{N\bigl(\varphi_{p-1,N}(z)-\Phi_{p-1}(\saddle_{p-1})\bigr)}\,dz
  \\
  =&
  \frac{\Gamma(1/3)\i}{3^{1/6}|\xi|^{2/3}N^{1/3}}
  e^{-\bigl(\pi+2\arctan(2p\pi/\kappa)\bigr)\i/3}
  \bigl(1+O(N^{-1/3})\bigr)
\end{split}
\end{equation*}
as $N\to\infty$.
Since $\Phi_m(\saddle_m)=4p\pi^2/\xi$ from \eqref{eq:F_derivatives}, we conclude
\begin{multline*}
  \int_{m/p}^{(m+1)/p}
  e^{N\varphi_{m,N}(z)}\,dz
  \\
  =
  \frac{\Gamma(1/3)\i}{3^{1/6}|\xi|^{2/3}N^{1/3}}
  e^{-\bigl(\pi+2\arctan(2p\pi/\kappa)\bigr)\i/3}
  e^{4p\pi^2N/\xi}
  \bigl(1+O(N^{-1/3})\bigr)
\end{multline*}
if $m<p-1$, and
\begin{multline*}
  \int_{1-1/p}^{1-\delta}
  e^{N\varphi_{p-1,N}(z)}\,dz
  \\
  =
  \frac{\Gamma(1/3)\i}{3^{1/6}|\xi|^{2/3}N^{1/3}}
  e^{-\bigl(\pi+2\arctan(2p\pi/\kappa)\bigr)\i/3}
  e^{4p\pi^2N/\xi}
  \bigl(1+O(N^{-1/3})\bigr).
\end{multline*}
\par
Since $\varphi_{m,N}(\saddle_m)=f_N(\saddle_0)$ converges to $F(\saddle_0)=4p\pi^2/\xi$ as $N\to\infty$ from \eqref{eq:F_derivatives}, together with \eqref{eq:sum_integral_Phi} and \eqref{eq:sum_integral_Phi_p-1}, we finally have
\begin{multline}\label{eq:approx_sum_Phi}
  \sum_{m/p\le k/N\le(m+1)/p}e^{N\varphi_{m,N}(k/N)}
  \\
  =
  \frac{\Gamma(1/3)e^{\pi\i/6}}{3^{1/6}}
  \left(\frac{N}{\xi}\right)^{2/3}
  e^{4p\pi^2N/\xi}
  \bigl(1+O(N^{-1/3})\bigr)
\end{multline}
if $m<p-1$, and
\begin{multline}\label{eq:approx_sum_Phi_p-1}
  \sum_{1-1/p\le k/N\le1-\delta}e^{N\varphi_{p-1,N}(k/N)}
  \\
  =
  \frac{\Gamma(1/3)e^{\pi\i/6}}{3^{1/6}}
  \left(\frac{N}{\xi}\right)^{2/3}
  e^{4p\pi^2N/\xi}
  \bigl(1+O(N^{-1/3})\bigr)
\end{multline}
because $\Re(4p\pi^2/\xi)>0$, where we define $\xi^{2/3}$ to be $|\xi|^{2/3}e^{2\arctan(2p\pi/\kappa)\i/3}$.
\par
It remains to obtain the asymptotic behavior of $\sum_{1-1/p\le k/N<1}e^{N\varphi_{p-1,N}(k/N)}$ instead of the sum for $1-1/p\le k/N\le 1-\delta$.
To do that, we need to estimate the sum $\sum_{1-\delta<k/N<1}e^{N\varphi_{p-1,N}(k/N)}$.
We use the following lemma, which corresponds to \cite[Lemma~6.1]{Murakami:CANJM2023}.
\begin{lem}\label{lem:Phi1}
For any $\varepsilon$, there exists $\delta'>0$ such that
\begin{equation*}
  \Re\varphi_{p-1,N}
  \left(
    \frac{2k+1}{2N}
  \right)
  <
  \Re\Phi_{p-1}(\saddle_{p-1})-\varepsilon
\end{equation*}
for sufficiently large $N$, if $1-\delta'<k/N<1$.
\end{lem}
Since a proof is similar to that of \cite[Lemma~6.1]{Murakami:CANJM2023}, we omit it.
\par
From Lemma~\ref{lem:Phi1}, we conclude that
\begin{equation*}
  \sum_{1-\delta<k/N<1}
  \exp\left(N\varphi_{p-1,N}\left(\frac{2k+1}{2N}\right)\right)
\end{equation*}
is of order $O\left(e^{N(\Re\Phi_{p-1}(\saddle_{p-1})-\varepsilon)}\right)$ if $\delta'<\delta$.
Since $\Phi_{p-1}(\saddle_{p-1})=4p\pi^2\i/\xi$ from \eqref{eq:Phi_derivatives}, we have
\begin{equation*}
  \sum_{1-1/p\le k/N<1}e^{N\Phi_{p-1,N}(k/N)}
  =
  \frac{\Gamma(1/3)e^{\pi\i/6}}{3^{1/6}}
  \left(\frac{N}{\xi}\right)^{2/3}
  e^{4p\pi^2N/\xi}
  \bigl(1+O(N^{-1/3})\bigr)
\end{equation*}
from \eqref{eq:approx_sum_Phi_p-1}.
Together with \eqref{eq:J_phiN} and \eqref{eq:approx_sum_Phi}, we have
\begin{equation}
\begin{split}
  &J_N\bigl(\FE;e^{\xi/N}\bigr)
  \\
  =&
  \bigl(1-e^{-4pN\pi^2/\xi}\bigr)
  \frac{\Gamma(1/3)e^{\pi\i/6}}{3^{1/6}}
  \left(\frac{N}{\xi}\right)^{2/3}
  e^{4p\pi^2N/\xi}
  \left(\sum_{m=0}^{p-1}\beta_{p,m}\right)
  \bigl(1+O(N^{-1/3})\bigr).
\end{split}
\end{equation}
Now from \eqref{eq:beta_pm} and \eqref{eq:Habiro_Le}, the sum in the parentheses is just $J_{p}\bigl(\FE;e^{4\pi^2N/\xi}\bigr)$.
Therefore we finally have
\begin{align*}
  &J_N\left(\FE;e^{\xi/N}\right)
  \\
  =&
  J_p\left(\FE;e^{4\pi^2N/\xi}\right)
  \frac{\Gamma(1/3)e^{\pi\i/6}}{3^{1/6}}
  \left(\frac{N}{\xi}\right)^{2/3}\exp\left(\frac{2\kappa\pi\i}{\xi}N\right)
  \left(1+O(N^{-1/3})\right),
\end{align*}
where we replace $e^{4p\pi^2N/\xi}$ with $e^{(4p\pi^2N)/\xi+2N\pi\i}=e^{2N\kappa\pi\i/\xi}$ on purpose; see \S~\ref{sec:CS}.
Note that we choose the argument of $\xi^{2/3}$ as $\frac{2}{3}\arctan(2p\pi/\kappa)$, which is between $\pi/6$ and $\pi/3$.
\begin{proof}[Proof of Corollary~\ref{cor:conjugate}]
Since the figure-eight knot is amphicheiral, that is, it is equivalent to its mirror image, we have $J_N(\FE;q^{-1})=J_N(\FE;q)$.
It follows that $J_N\left(\FE;e^{\xi'/N}\right)=J_N\left(\FE;e^{-\xi'/N}\right)=J_N\left(\FE;e^{\overline{\xi}/N}\right)=\overline{J_N\left(\FE;e^{\xi/N}\right)}$, where $\overline{\xi}$ is the complex conjugate.
So we obtain
\begin{equation*}
\begin{split}
  &J_N\left(\FE;e^{\xi'/N}\right)
  \\
  \underset{N\to\infty}{\sim}&
  \overline{J_p\left(\FE;e^{4\pi^2N/\xi}\right)}
  \frac{\Gamma(1/3)e^{-\pi\i/6}}{3^{1/6}}
  \left(\frac{N}{\overline{\xi}}\right)^{2/3}
  \exp\left(\frac{-2\kappa\pi\i}{\overline{\xi}}N\right)
  \\
  \underset{\phantom{N\to\infty}}{=}&
  J_p\left(\FE;e^{4\pi^2N/\xi'}\right)
  \frac{\Gamma(1/3)e^{\pi\i/6}}{3^{1/6}}
  \left(\frac{N}{\xi'}\right)^{2/3}
  \exp\left(\frac{S_{-\kappa}(\FE)}{\xi'}N\right),
\end{split}
\end{equation*}
where we put $(\xi')^{1/3}:=|\xi'|^{1/3}e^{-\arctan(2p\pi/\kappa)\i/3}\times e^{-\pi\i/3}$.
The last equality follows since $e^{-2\kappa\pi\i N/\overline{\xi}}=e^{2\kappa\pi\i N/\xi'}=e^{2\kappa\pi\i N/\xi'+4N\pi\i}=e^{(-2\kappa\pi\i-8pN\pi^2)N/\xi'}=e^{(S_{-\kappa}(\FE)-8pN\pi^2)/\xi'}$ and the Chern--Simons invariant is defined modulo an integer multiple of $\pi^2$ (see \S~\ref{sec:CS}).
\end{proof}

%% file: CS.tex
\section{Chern--Simons invariant}\label{sec:CS}
In this section, we show a relation between $S_{\kappa}(E)=2\kappa\pi\i$ appearing in Theorem~\ref{thm:main} to the Chern--Simons invariant.
For the definition of the Chern--Simons invariant of a representation from the fundamental group of a three-manifold with toric boundary to $\SL(2;\C)$, we refer the readers to \cite{Kirk/Klassen:COMMP1993}.
\par
Let $W$ be the three-manifold obtained from $S^3$ by removing the open tubular neighborhood of a knot $K\subset{S^3}$.
We denote by $X(W)$ the $\SL(2;\C)$ character variety, that is, the set of characters of representations from $\pi_1(W)$ to $\SL(2;\C)$.
Let $E(\partial{W})$ be the quotient space $\bigl(\Hom(\pi_1(\partial{W}),\C)\times\C^{\times}\bigr)/G$, where $\C^{\times}:=\C\setminus\{0\}$ and $G:=\langle x,y,b\mid xy=yx,bxbx=byby=b^2=1\rangle$ acts on $\Hom(\pi_1(\partial{W}),\C)\times\C^{\times}$ as follows:
\begin{equation}\label{eq:G_action}
\begin{split}
  x\cdot(\alpha,\beta;z)
  &:=
  (\alpha+1/2,\beta;z\exp(-4\pi\i\beta)),
  \\
  y\cdot(\alpha,\beta;z)
  &:=
  (\alpha,\beta+1/2;z\exp(4\pi\i\alpha)),
  \\
  b\cdot(\alpha,\beta)
  &:=
  (-\alpha,-\beta;z).
\end{split}
\end{equation}
Here we fix a generator $(\mu^{\ast},\lambda^{\ast})\in\Hom(\pi_1(\partial{W});\C)\cong\C^2$ for a meridian $\mu$ (homotopy class of the loop that goes around $K$) and a preferred longitude $\lambda$ (homotopy class of the loop that goes along $K$ so that its linking number with $K$ is zero).
Then the projection $p\colon E(\partial{W})\to X(\partial{W})$ sending $[\alpha,\beta;z]$ to $[\alpha,\beta]$ becomes a $\C^{\times}$-bundle, where the square brackets mean the equivalence class.
\par
The $\SL(2;\C)$ Chern--Simons invariant of $W$ defines a lift $\cs_{W}\colon X(W)\to E(\partial{W})$ of $X(W)\xrightarrow{i^{\ast}}X(\partial{W})$, that is, $p\circ c_{W}=i^{\ast}$ holds, where $i^{\ast}$ is induced by the inclusion map $i\colon\partial{W}\to W$.
\begin{equation*}
\begin{tikzcd}
                                            &E(\partial{W})\arrow[d,"p"] \\
X(W)\arrow[ru,"\cs_{W}"]\arrow[r,"i^{\ast}"]&X(\partial{W})
\end{tikzcd}
\end{equation*}
For a representation $\rho$, we have $\cs_{W}([\rho])=\left[\frac{u}{4\pi\i},\frac{v}{4\pi\i};\exp\left(\frac{2}{\pi\i}\CS_{u,v}(\rho)\right)\right]$ if $\rho(\mu)=\begin{pmatrix}e^{u/2}&\ast\\0&e^{-u/2}\end{pmatrix}$ and $\rho(\lambda)=\begin{pmatrix}e^{v/2}&\ast\\0&e^{-v/2}\end{pmatrix}$ up to conjugation, where $[\rho]\in X(W)$ means the equivalence class, and $\CS_{u,v}(\rho)$ is the $\SL(2;\C)$ Chern--Simons invariant of $\rho$ associated with $(u,v)$.
Note that $\CS_{u,v}(\rho)$ is defined modulo $\pi^2$, and depends on the choice of branches of logarithms of $e^{u/2}$ and $e^{v/2}$.
\par
Now, we calculate the $\SL(2;\C)$ Chern--Simons invariant of the figure-eight knot.
See also \cite[\S~5.2]{Murakami/Yokota:2018} for calculation about the figure-eight knot complement.
\par
\begin{figure}[h]
\pic{0.3}{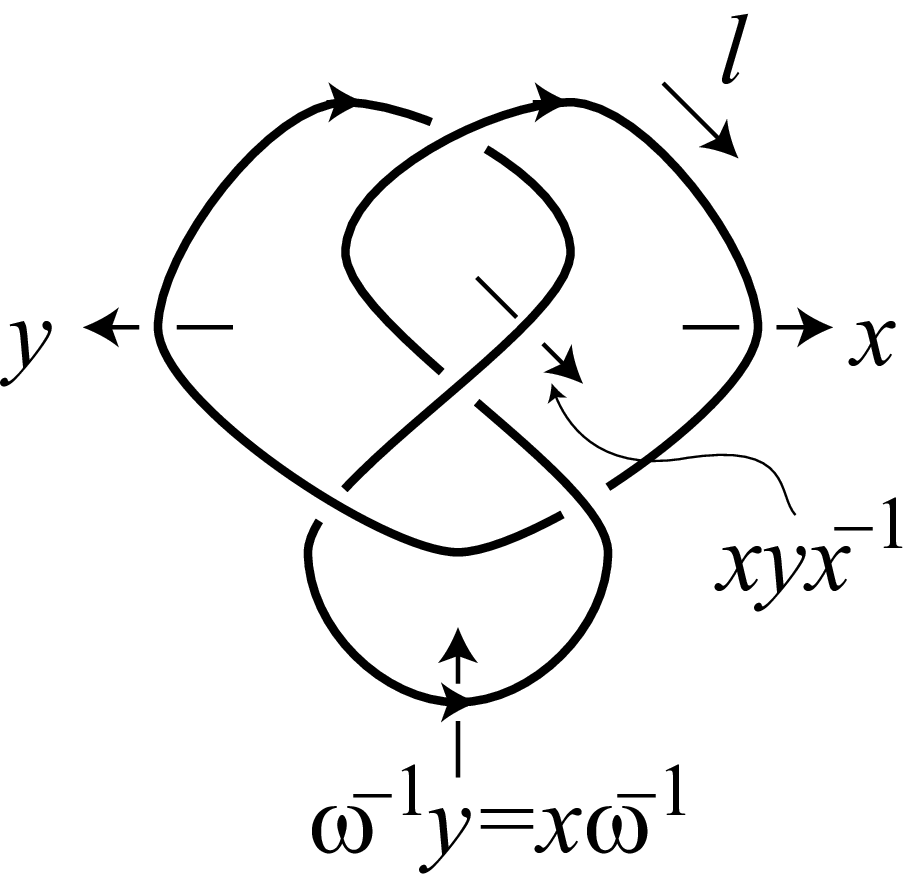}
\caption{The figure-eight knot $\FE$ and generators of $G_{\FE}:=\pi_1(S^3\setminus{\FE})$.}
\label{fig:pi1_fig8}
\end{figure}
By using generators $x,y$ as indicated in Figure~\ref{fig:pi1_fig8}, the fundamental group $G_{\FE}:=\pi_1(S^3\setminus\FE)$ has a presentation $\langle x,y\mid \omega x=y\omega\rangle$, where $\omega:=xy^{-1}x^{-1}y$.
We choose (the homotopy class of) $x$ as the meridian $\mu$, and (the homotopy class of) $l$ depicted in Figure~\ref{fig:pi1_fig8} as the preferred longitude $\lambda$.
The loop $l$ presents the element $x\omega^{-1}\overleftarrow{\omega}^{-1}x^{-1}\in G_{\FE}$, where $\overleftarrow\omega:=yx^{-1}y^{-1}x$ is the word obtained from $\omega$ by reading backward.
Due to \cite{Riley:QUAJM31984} (see also \cite[\S~3]{Murakami:ACTMV2008}), for a real number $u$ with $0\le u\le\kappa$ we consider the non-Abelian representation $\rho_{u}\colon G_{\FE}\to\SL(2;\C)$ sending $x$ to $\begin{pmatrix}e^{u/2}&1\\0&e^{-u/2}\end{pmatrix}$ and $y$ to $\begin{pmatrix}e^{u/2}&0\\d&e^{-u/2}\end{pmatrix}$, where $d$ is given as
\begin{equation*}
  d
  :=
  \frac{3}{2}-\cosh{u}
  +
  \frac{1}{2}\sqrt{(2\cosh(u)+1)(2\cosh(u)-3)}.
\end{equation*}
The preferred longitude is sent to $\begin{pmatrix}e^{v(u)/2}&\ast\\0&e^{-v(u)/2}\end{pmatrix}$,
where
\begin{equation*}
\begin{split}
  v(u)
  :=&
  2\log
  \left(
    \cosh(2u)-\cosh(u)-1
    -
    \sinh(u)\sqrt{(2\cosh(u)+1)(2\cosh(u)-3)}
  \right)
  \\
  &+2\pi\i.
\end{split}
\end{equation*}
Here we add $2\pi\i$ so that $v(0)=0$.
\par
It is well known \cite{Thurston:GT3M} that when $u=0$, the irreducible representation $\rho_{0}$ induces a complete hyperbolic structure in $S^3\setminus{\FE}$, and when $0<u<\kappa$, $\rho_{u}$ is irreducible and induces an incomplete hyperbolic structure.
When $u=\kappa$, the representation $\rho_{\kappa}$ becomes reducible (and non-Abelian), and the hyperbolic structure collapses.
In fact, in this case, both $x$ and $y$ are sent to upper triangular matrices, and so every element of $G_{\FE})$ is sent to an upper triangular matrix, which means that $\rho_{\kappa}$ is reducible.
This kind of reducible and non-Abelian representation is called affine, and corresponds to the zeroes of the Alexander polynomial.
See \cite{Burde:MATHA1967}, \cite{deRham:ENSEM21967}, \cite[Exercise~11.2]{Kauffman:Knots}, \cite[2.4.3. Corollary]{Le:RUSSM1993}.
\par
Now, we calculate the $\SL(2;\C)$ Chern--Simons invariant $\CS_{\kappa,v(\kappa)}(\rho_{\kappa})$ associated with $(\kappa,v(\kappa))=(\kappa,2\pi\i)$.
See \cite{Kirk/Klassen:COMMP1993} for details.
\par
Since the Chern--Simons invariant of a representation is determined by its character, and $\rho_{\kappa}$ shares the same character (trace) with the Abelian representation $\rho^{\rm{Abel}}_{\kappa}$ sending $\mu:=x$ to the diagonal matrix $\begin{pmatrix}e^{\kappa/2}&0\\0&e^{-\kappa/2}\end{pmatrix}$ and $\lambda:=l$ to the identity matrix, it can be easily seen that $\cs_{W}(\rho^{\rm{Abel}}_{\kappa})=\left[\frac{\kappa}{4\pi\i},0;1\right]$, where we put $W:=S^3\setminus{N(\FE)}$ with $N(\FE)$ is the open tubular neighborhood of $\FE$ in $S^3$.
Since we have
\begin{equation*}
  \left[\frac{\kappa}{4\pi\i},0;1\right]
  =
  \left[\frac{\kappa}{4\pi\i},\frac{1}{2};e^{\kappa}\right]
\end{equation*}
from \eqref{eq:G_action}, we conclude that $\CS_{\kappa,2\pi\i}(\rho_{\kappa})=\kappa\pi\i/2$.
Note that here we change the pair $(\kappa,0)$ to $(\kappa,2\pi\i)$.
\par
As in \cite{Murakami:JTOP2013}, if we define
\begin{equation}\label{eq:S_fig8}
  S_{u}(\FE)
  :=
  \CS_{u,v(u)}(\rho_{u})
  +\pi\i u+\frac{1}{4}uv(u)
\end{equation}
for $0\le u\le\kappa$, then we see that $S_{\kappa}(\FE)=2\kappa\pi\i$ when $\bigl(u,v(u)\bigr)=(\kappa,2\pi\i)$.
\par
Similarly, we see that $\CS_{-\kappa,2\pi\i}(\rho_{-\kappa})=-\kappa\pi\i/2$, and that $S_{-u}(\FE)=-2\kappa\pi\i$.

%% file: Poisson_proof.tex
\section{Proof of the Poisson summation formula}\label{sec:Poisson_proof}
In this appendix, we give a proof of the Poisson summation formula following \cite[Proposition~4.2]{Ohtsuki:QT2016}.
\begin{proof}[Proof of Proposition~\ref{prop:Poisson}]
Let $\varepsilon>0$ be small enough so that $\Re{\psi(a)}<-\varepsilon$,
$\Re{\psi(b)}<-\varepsilon$, $\Re{\psi(z)}-2\pi\Im{z}<-\varepsilon$ if $z\in C_{+}$, and $\Re{\psi(z)}+2\pi\Im{z}<-\varepsilon$ if $z\in C_{-}$.
Then for sufficiently large $N$, the following also hold:
\renewcommand{\theenumi}{\roman{enumi}}
\begin{enumerate}
\item\label{item:psi(a)}
$\Re{\psi_N(a)}<-\varepsilon$,
\item\label{item:psi(b)}
$\Re{\psi_N(b)}<-\varepsilon$,
\item\label{item:C+}
$\Re{\psi_N(z)}-2\pi\Im{z}<-\varepsilon$ if $z\in C_{+}$,
\item\label{item:C-}
$\Re{\psi_N(z)}+2\pi\Im{z}<-\varepsilon$ if $z\in C_{-}$.
\end{enumerate}
Moreover, there exists $\delta>0$ such that $\Re{\psi_N(t)}<-\varepsilon$ if $t\in[a,a+\delta]$ or $t\in[b-\delta,b]$ from (\ref{item:psi(a)}) and (\ref{item:psi(b)}) for such $N$.
\par
Let $\beta\colon\R\to[0,1]$ be a $C^{\infty}$-function such that
\begin{equation*}
  \beta(t)
  =
  \begin{cases}
    1&\text{if $t\in[a+\delta,b-\delta]$},
    \\
    0&\text{if $t<a$ or $t>b$}.
  \end{cases}
\end{equation*}
We also assume that $\beta(t)$ is in the Schwartz space $\mathcal{S}(\R)$, that is, $\sup_{x\in\R}\left|x^{m}f^{(n)}(x)\right|<\infty$ for any non-negative integers $m$ and $n$.
Put $\Psi_N(x):=\beta(x/N)e^{N\psi_N(x/N)}$.
\par
We have
\begin{equation}\label{eq:a_delta_f}
  \left|\sum_{a\le k/N<a+\delta}e^{N\psi_N(k/N)}\right|
  \le
  \sum_{a\le k/N<a+\delta}e^{N\Re{\psi_N(k/N)}}
  <
  \delta Ne^{-\varepsilon N},
\end{equation}
where the second inequality follows since $\Re{\psi_N(k/N)}<-\varepsilon$ when $a\le k/N\le a+\delta$.
Similarly we have
\begin{equation}\label{eq:b_delta_f}
  \left|\sum_{b-\delta\le k/N<b}e^{N\psi_N(k/N)}\right|
  <
  \delta Ne^{-\varepsilon N}.
\end{equation}
We also have
\begin{equation}\label{eq:a_delta_F}
  \left|\sum_{k/N<a+\delta}\Psi_N(k)\right|
  \le
  \sum_{a\le k/N<a+\delta}\beta(k/N)e^{N\Re{\psi_N(k/N)}}
  <
  \delta Ne^{-\varepsilon N}
\end{equation}
and
\begin{equation}\label{eq:b_delta_F}
  \left|\sum_{k/N>b-\delta}\Psi_N(k)\right|
  <
  \delta Ne^{-\varepsilon N}.
\end{equation}
Since $\Psi_N(k)=e^{N\psi_N(k/N)}$ if $a+\delta\le k/N\le b-\delta$, we have
\begin{equation}\label{eq:F_f}
\begin{split}
  &
  \left|\sum_{k\in\Z}\Psi_N(k)-\sum_{a\le k/N\le b}e^{N\psi_N(k/N)}\right|
  \\
  \le&
  \left|\sum_{k/N<a+\delta}\Psi_N(k)\right|
  +
  \left|\sum_{a\le k/N<a+\delta}e^{N\psi_N(k/N)})\right|
  \\
  &+
  \left|\sum_{b-\delta<k/N\le b}\Psi_N(k)\right|
  +
  \left|\sum_{k/N>b-\delta}e^{N\psi_N(k/N)})\right|
  \\
  <&
  4\delta Ne^{-\varepsilon N}
\end{split}
\end{equation}
from \eqref{eq:a_delta_f}--\eqref{eq:b_delta_F}.
\par
Since $\Psi_N(t)$ is also in $\mathcal{S}(\R)$, we can apply the Poisson summation formula (see for example \cite[Theorem~3.1]{Stein/Shakarchi2003}):
\begin{equation}\label{eq:F_Fhat}
  \sum_{k\in\Z}\Psi_N(k)
  =
  \sum_{l\in\Z}\hat\Psi_N(l),
\end{equation}
where $\hat\Psi_N$ is the Fourier transform of $\Psi_N$, that is, $\hat\Psi_N(l):=\int_{-\infty}^{\infty}\Psi_N(t)e^{-2l\pi\i t}\,dt$.
\par
Putting $s:=t/N$, we have
\begin{equation}\label{eq:Fhat}
  \hat\Psi_N(l)
  =
  N
  \int_{-\infty}^{\infty}
  \beta(s)e^{N\bigl(\psi_N(s)-2l\pi\i s\bigr)}\,ds.
\end{equation}
\par
From the properties of $\beta(s)$, we have
\begin{equation}\label{eq:Fhat_0}
\begin{split}
  &\left|
    \frac{1}{N}\hat\Psi_N(0)-\int_{a}^{b}e^{N\psi_N(s)}\,ds
  \right|
  \\
  \le&
  \left|\int_{a}^{a+\delta}\bigl(\beta(s)-1\bigr)e^{N\psi_N(s)}\,ds\right|
  +
  \left|\int_{b-\delta}^{b}\bigl(\beta(s)-1\bigr)e^{N\psi_N(s)}\,ds\right|
  \\
  \le&
  \int_{a}^{a+\delta}\bigl(1-\beta(s)\bigr)e^{N\Re{\psi_N(s)}}\,ds
  +
  \int_{b-\delta}^{b}\bigl(1-\beta(s)\bigr)e^{N\Re{\psi_N(s)}}\,ds
  \\
  <&
  2\delta e^{-\varepsilon N}.
\end{split}
\end{equation}
\par
Therefore we have
\begin{equation}\label{eq:f_sum_int}
\begin{split}
  &\left|
    \frac{1}{N}\sum_{a\le k/N\le b}e^{N\psi_N(k/N)}
    -
    \int_{a}^{b}e^{N\psi_N(s)}\,ds
  \right|
  \\
  \le&
  \left|
    \frac{1}{N}\sum_{a\le k/N\le b}e^{N\psi_N(k/N)}
    -
    \frac{1}{N}
    \sum_{l\in\Z}\hat\Psi_N(l)
  \right|
  +
  \left|
    \frac{1}{N}
    \sum_{l\in\Z}\hat\Psi_N(l)
    -
    \int_{a}^{b}e^{N\psi_N(s)}\,ds
  \right|
  \\
  &\text{(from \eqref{eq:F_Fhat})}
  \\
  \le&
  \left|
    \frac{1}{N}\sum_{a\le k/N\le b}e^{N\psi_N(k/N)}
    -
    \frac{1}{N}
    \sum_{m\in\Z}\Psi_N(k)
  \right|
  +
  \left|
    \frac{1}{N}\hat\Psi_N(0)
    -
    \int_{a}^{b}e^{N\psi_N(s)}\,ds
  \right|
  \\
  &+
  \frac{1}{N}
  \sum_{l\in\Z,l\ne0}|\hat\Psi_N(l)|
  \\
  &\text{(from \eqref{eq:F_f} and \eqref{eq:Fhat_0})}
  \\
  <&
  \frac{1}{N}
  \sum_{l\in\Z,l\ne0}|\hat\Psi_N(l)|
  +
  6\delta e^{-\varepsilon N}.
\end{split}
\end{equation}
\par
Next we calculate $\hat\Psi_N(l)$ for $l\ne0$.
Integrating the right hand side of \eqref{eq:Fhat} by parts twice, we have
\begin{align*}
  \hat\Psi_N(l)
  =&
  \frac{1}{2l\pi\i}
  \int_{-\infty}^{\infty}
  \frac{d}{d\,s}\left(\beta(s)e^{N \psi_N(s)}\right)e^{-2l\pi\i Ns}\,ds
  \\
  =&
  -\frac{1}{4l^2\pi^2 N}
  \int_{-\infty}^{\infty}
  \frac{d^2}{d\,s^2}\left(\beta(s)e^{N \psi_N(s)}\right)e^{-2l\pi\i Ns}\,ds.
\end{align*}
Putting $B_N(s):=\beta''(s)+2N\beta'(s)\psi'_N(s)+N\beta(s)\psi''_N(s)+N^2\beta(s)\bigl(\psi'(s)\bigr)^{2}$ and $\tilde{B}_N(s):=N\psi''_N(s)+N^2\bigl(\psi'_N(s)\bigr)^{2}$, we have
\begin{align*}
  &-4l^2\pi^2N\hat\Psi_N(l)
  \\
  =&
  \int_{a}^{b}
  B_N(s)e^{N\bigl(\psi_N(s)-2l\pi\i s\bigr)}\,ds
  \\
  =&
  \int_{a+\delta}^{b-\delta}
  \tilde{B}_N(s)e^{N\bigl(\psi_N(s)-2l\pi\i s\bigr)}\,ds
  \\
  &+
  \int_{a}^{a+\delta}
  B_N(s)e^{N\bigl(\psi_N(s)-2l\pi\i s\bigr)}\,ds
  +
  \int_{b-\delta}^{b}
  B_N(s)e^{N\bigl(\psi_N(s)-2l\pi\i s\bigr)}\,ds
  \\
  =&
  \int_{a}^{b}
  \tilde{B}_N(s)e^{N\bigl(\psi_N(s)-2l\pi\i s\bigr)}\,ds
  \\
  &-
  \int_{a}^{a+\delta}
  \tilde{B}_N(s)e^{N\bigl(\psi_N(s)-2l\pi\i s\bigr)}\,ds
  -
  \int_{b-\delta}^{b}
  \tilde{B}_N(s)e^{N\bigl(\psi_N(s)-2l\pi\i s\bigr)}\,ds
  \\
  &+
  \int_{a}^{a+\delta}
  B_N(s)e^{N\bigl(\psi_N(s)-2l\pi\i s\bigr)}\,ds
  +
  \int_{b-\delta}^{b}
  B_N(s)e^{N\bigl(\psi_N(s)-2l\pi\i s\bigr)}\,ds,
\end{align*}
where the second equality follows because $B_N(s)=\tilde{B}_N(s)$ when $s\in[a+\delta,b-\delta]$.
So we have
\begin{equation}\label{eq:Fhat_int}
\begin{split}
  &\left|
    4l^2\pi^2 N
    \hat\Psi_N(l)
    +
    \int_{a}^{b}\tilde{B}_N(s)e^{N\bigl(\psi_N(s)-2l\pi\i s\bigr)}\,ds
  \right|
  \\
  \le&
  \left|
    \int_{a}^{a+\delta}B_N(s)e^{N\bigl(\psi_N(s)-2l\pi\i s\bigr)}\,ds
  \right|
  +
  \left|
    \int_{b-\delta}^{b}B_N(s)e^{N\bigl(\psi_N(s)-2l\pi\i s\bigr)}\,ds
  \right|
  \\
  &+
  \left|
    \int_{a}^{a+\delta}\tilde{B}_N(s)e^{N\bigl(\psi_N(s)-2l\pi\i s\bigr)}\,ds
  \right|
  +
  \left|
    \int_{b-\delta}^{b}\tilde{B}_N(s)e^{N\bigl(\psi_N(s)-2k\pi\i s\bigr)}\,ds
  \right|.
\end{split}
\end{equation}
Since $\Re{\psi_N(s)}<-\varepsilon$ if $a\le s\le a+\delta$, we have
\begin{equation}\label{eq:Ka}
\begin{split}
  \left|
    \int_{a}^{a+\delta}
    B_N(s)e^{N\bigl(\psi_N(s)-2l\pi\i s\bigr)}\,ds
  \right|
  \le&
  \int_{a}^{a+\delta}
  \left|B_N(s)\right|e^{N\Re{\psi_N(s)}}\,ds
  \\
  <&
  \delta e^{-\varepsilon N}\times\max_{s\in[a,a+\delta]}|B_N(s)|
  \le
  K_aN^2e^{-\varepsilon N},
\end{split}
\end{equation}
where we put
\begin{align*}
  K_a
  :=&
  \max_{s\in[a,a+\delta]}|\beta''(s)|
  +
  \max_{s\in[a,a+\delta]}|2\beta'(s)\psi'_N(s)|
  +
  \max_{s\in[a,a+\delta]}|\beta(s)\psi''_N(x)|
  \\
  &+
  \max_{s\in[a,a+\delta]}|\beta(s)\bigl(\psi'_N(s)\bigr)^2|
  \\
  \ge&
  \max_{s\in[a,a+\delta]}
  \left|
    \frac{\beta''(s)}{N^2}
    +\frac{2\beta'(s)\psi'_N(s)}{N}
    +\frac{\beta(s)\psi''_N(s)}{N}
    +\beta(s)\bigl(\psi'_N(s)\bigr)^2
  \right|
  \\
  =&
  \max_{s\in[a,a+\delta]}\left|B_N(s)\right|\frac{1}{N^2}.
\end{align*}
Similarly, putting
\begin{align*}
  K_b
  :=&
  \max_{s\in[b-\delta,b]}|\beta''(s)|
  +
  \max_{s\in[b-\delta,b]}|2\beta'(s)\psi'_N(s)|
  +
  \max_{s\in[b-\delta,b]}|\beta(s)\psi''_N(x)|
  \\
  &+
  \max_{s\in[b-\delta,b]}|\beta(s)\bigl(\psi'_N(s)\bigr)^2|,
  \\
  \tilde{K}_a
  :=&
  \max_{s\in[a,a+\delta]}|\psi''_N(s)|
  +
  \max_{s\in[a,a+\delta]}|\bigl(\psi'_N(s)\bigr)^2|,
  \\
  \tilde{K}_b
  :=&
  \max_{s\in[b-\delta,b]}|\psi''_N(s)|
  +
  \max_{s\in[b-\delta,b]}|\bigl(\psi'_N(s)\bigr)^2|,
\end{align*}
we have
\begin{align}
  \left|\int_{b-\delta}^{b}B_N(s)e^{N\bigl(\psi_N(s)-2k\pi\i s\bigr)}\,ds\right|
  <&
  K_bN^2e^{-\varepsilon N},
  \label{eq:Kb}
  \\
  \left|\int_{a}^{a+\delta}\tilde{B}_N(s)e^{N\bigl(\psi_N(s)-2k\pi\i s\bigr)}\,ds\right|
  <&
  \tilde{K}_aN^2e^{-\varepsilon N},
  \label{eq:tildeKa}
  \\
  \left|\int_{b-\delta}^{b}\tilde{B}_N(s)e^{N\bigl(\psi_N(s)-2k\pi\i s\bigr)}\,ds\right|
  <&
  \tilde{K}_bN^2e^{-\varepsilon N}.
  \label{eq:tildeKb}
\end{align}
Therefore we have
\begin{equation}\label{eq:hatFl}
  \left|\hat\Psi_N(l)\right|
  <
  \frac{1}{4l^2\pi^2N}
  \left|
    \int_{a}^{b}\tilde{B}_N(s)e^{N\bigl(\psi_N(s)-2l\pi\i s\bigr)}\,ds
  \right|
  +
  \frac{KN}{4l^2\pi^2}e^{-\varepsilon N}
\end{equation}
from \eqref{eq:Ka}--\eqref{eq:tildeKb}, where we put $K:=K_a+K_b+\tilde{K}_a+\tilde{K}_b$.
\par
To evaluate $\int_{a}^{b}\tilde{B}_N(s)e^{N\bigl(\psi_N(s)-2l\pi\i s\bigr)}\,ds$, we consider the paths $C_{\pm}\subset R_{\pm}$.
Note that $\tilde{B}_N$ is defined in $D$.
\par
By replacing the path $[a,b]$ with $C_{\pm}$, we have
\begin{equation}\label{eq:Cpm}
\begin{split}
  \left|\int_{a}^{b}\tilde{B}_N(s)e^{N\bigl(\psi_N(s)-2\pi\i ls\bigr)}\,ds\right|
  =&
  \left|\int_{C_{\pm}}\tilde{B}_N(z)e^{N\bigl(\psi_N(z)-2\pi\i lz\bigr)}\,dz\right|
  \\
  \le&
  \int_{C_{\pm}}\left|\tilde{B}_N(z)\right|e^{N\bigl(\Re{\psi_N(z)}+2l\pi\Im{z}\bigr)}\,|dz|
  \\
  \le&
  \max_{z\in C_{\pm}}\left|\tilde{B}_N(z)\right|
  \times
  \int_{C_{\pm}}e^{N\bigl(\Re{\psi_N(z)}+2l\pi\Im{z}\bigr)}\,|dz|
  \\
  \le&
  K_{\pm}N^2
  \int_{C_{\pm}}e^{N\bigl(\Re{\psi_N(z)}+2l\pi\Im{z}\bigr)}\,|dz|,
\end{split}
\end{equation}
where we put
\begin{align*}
  K_{\pm}
  :=&
  \max_{z\in C_{\pm}}|\psi''_N(z)|
  +
  \max_{z\in C_{\pm}}\left|\bigl(\psi'_N(z)|\bigr)^2\right|
  \\
  \ge&
  \max_{z\in C_{\pm}}
  \left|
    \frac{\psi''_N(z)}{N}+\bigl(\psi'_N(z)|\bigr)^2
  \right|
  =
  \max_{z\in C_{\pm}}\left|\tilde{B}_N(z)\right|\frac{1}{N^2}.
\end{align*}
\par
If $l\ge1$, we use $C_{-}$.
Since $C_{-}\subset R_{-}$, we have $\Im{z}\le0$ and $\Re{\psi_N(z)}+2\pi\Im{z}<-\varepsilon$  from (\ref{item:C-}).
So from \eqref{eq:Cpm}, we have
\begin{equation}\label{eq:F_positive}
  \left|\int_{a}^{b}\tilde{B}_N(s)e^{N\bigl(\psi_N(s)-2\pi\i ls\bigr)}\,ds\right|
  <
  \tilde{K}_{-}N^2e^{-\varepsilon N},
\end{equation}
where $\tilde{K}_{-}:=K_{-}\times\text{(length of $C_{-}$)}$.
\par
Similarly, if $l\le-1$, putting $\tilde{K}_{+}:=K_{+}\times\text{(length of $C_{+}$)}$, we have
\begin{equation}\label{eq:F_negative}
  \left|\int_{a}^{b}\tilde{B}_N(s)e^{N\bigl(\psi_N(s)-2\pi\i ls\bigr)}\,ds\right|
  <
  \tilde{K}_{+}N^2e^{-\varepsilon N}
\end{equation}
from (\ref{item:C+}).
\par
Therefore, from \eqref{eq:hatFl}--\eqref{eq:F_negative}, we have
\begin{align*}
  &\left|\sum_{l\in\Z,l\ne0}\hat\Psi_N(l)\right|
  \\
  <&
  \sum_{l=1}^{\infty}
  \left(
    \frac{\tilde{K}_{-}N}{4l^2\pi^2}e^{-\varepsilon N}
    +
    \frac{KN}{4l^2\pi^2}e^{-\varepsilon N}
  \right)
  +
  \sum_{l=1}^{\infty}
  \left(
    \frac{\tilde{K}_{+}N}{4l^2\pi^2}e^{-\varepsilon N}
    +
    \frac{KN}{4l^2\pi^2}e^{-\varepsilon N}
  \right)
  \\
  =&
  \left(
    \frac{\tilde{K}_{-}}{24}
    +
    \frac{\tilde{K}_{+}}{24}
    +
    \frac{K}{12}
  \right)
  Ne^{-\varepsilon N},
\end{align*}
since $\sum_{l=1}^{\infty}\frac{1}{l^2}=\frac{\pi^2}{6}$.
\par
From \eqref{eq:f_sum_int} we finally have
\begin{equation*}
  \left|
    \frac{1}{N}\sum_{a\le k/N\le b}e^{N\psi(k/N)}
    -
    \int_{a}^{b}e^{N\psi_N(s)}\,ds
  \right|
  <
  \left(
    6\delta
    +
    \frac{\tilde{K}_{-}}{24}
    +
    \frac{\tilde{K}_{+}}{24}
    +
    \frac{K}{12}
  \right)
  e^{-\varepsilon N},
\end{equation*}
proving the proposition.
\end{proof}

%% file: saddle_proof.tex
\section{Proof of the saddle point method of order two}\label{sec:saddle_proof}
In this appendix, we give a proof of Proposition~\ref{prop:saddle}.
\par
Let $c:=re^{\theta\i}$ be a complex number with $r>0$ and $-\pi<\theta\le\pi$, and put $U:=\{z\in\C\mid\Re(cz^3)<0\}$.
If we write $z:=se^{\tau\i}$ with $s>0$ and $\tau\in\R$, then since $cz^3=rs^3e^{(\theta+3\tau)\i}$, the region $U$ has three connected components $U_{k}$ ($k=1,2,3$):
\begin{equation}\label{eq:valley}
  U_{k}
  :=
  \left\{
    w\in\C\Bigm|
    w=se^{\tau\i},
    s>0,
    \left|\tau+\theta/3-(2k-1)\pi/3\right|<\pi/6
  \right\}.
\end{equation}
Note that $U_{k}$ ($k=1,2,3$) is obtained from $U_{k-1}$ by the $2\pi/3$-rotation around the origin $O$, where $U_{0}$ means $U_{3}$.
The origin $O$ is a saddle point of order two for the function $\Re(cz^3)$, and the regions $U_{k}$ are called valleys.
\par
First of all, we study the asymptotic behavior of the integral $\int_{C}h_N(z)e^{N\times cz^3}\,dz$ as $N\to\infty$, where $C$ is a path start at the origin into a valley, and $h_N(z)$ is a holomorphic function depending on $N$.
The following lemma follows from the techniques described in \cite[II.4]{Wong:1989}.
\begin{lem}\label{lem:cz^3}
Let $D$ be an open bounded region in $\C$ containing $O$, $h_N(z)$ be a holomorphic function in $D$ depending on a positive integer $N$, and $U_{k}$ be as above.
We assume that $h_N(z)$ uniformly converges to a holomorphic function $h(z)$ with $h(0)\ne0$ and that $\left|h_N(z)\right|$ is bounded irrelevant to $z$ or $N$.
We also assume that $U_{k}\cap D$ is connected and simply-connected for each $k$.
For a point $a\in U_{k}\cap D$, let $C\subset(U_{k}\cap D)\cup\{O\}$ be a path from $O$ to $a$.
Then we have
\begin{equation*}
  \int_{C}h_N(z)e^{N\times cz^3}\,dz
  =
  \frac{e^{\bigl((2k-1)\pi-\theta\bigr)\i/3}h(0)\Gamma(1/3)}{3r^{1/3}N^{1/3}}
  \left(1+O(N^{-1/3})\right)
\end{equation*}
as $N\to\infty$, where $\Gamma(x):=\int_{0}^{\infty}t^{x-1}e^{-t}\,dt$ is the gamma function.
\end{lem}
\begin{proof}
Let $\tilde{U}_{k}$ be the region obtained from $U_{k}$ by the $\bigl(\theta-(2k-1)\pi\bigr)/3$-rotation around $O$, that is,
\begin{equation}\label{eq:tildeU_k}
  \tilde{U}_k
  :=
  \left\{w\in\C\mid w=se^{\tau\i},s>0,|\tau|<\pi/6\right\}.
\end{equation}
The same rotation sends $D$ to $\tilde{D}$, $C$ to $\tilde{C}\subset(\tilde{U}_{k}\cap\tilde{D})\cup\{O\}$, and $a$ to $\tilde{a}:=e^{\bigl(\theta-(2k-1)\pi\bigr)\i/3}a\in\tilde{U}_{k}\cap\tilde{D}$.
See Figure~\ref{fig:U}.
\begin{figure}[h]
  \begin{center}
    \pic{0.24}{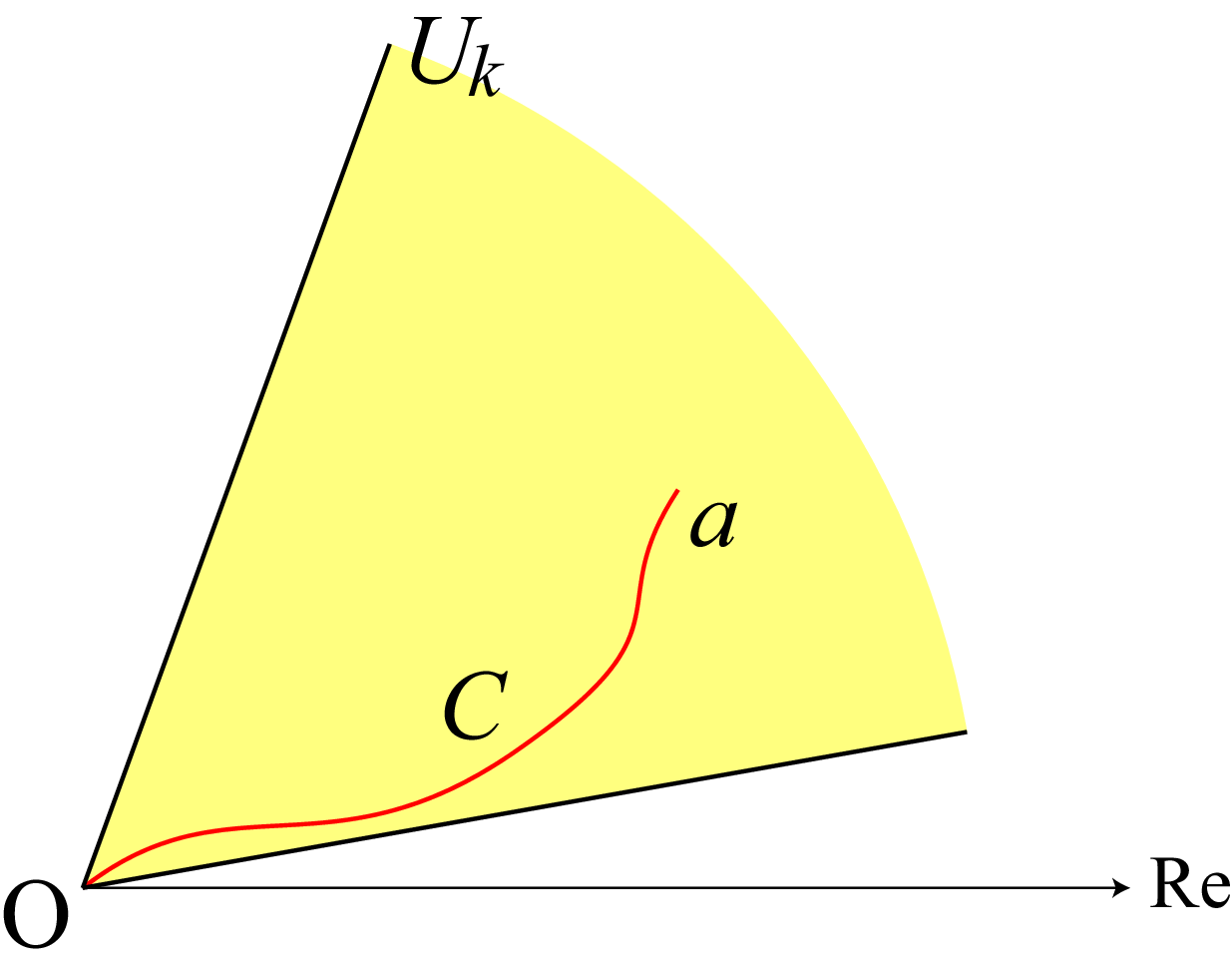}
    $\xrightarrow[\text{-rotation}]{\bigl(\theta-(2k-1)\pi\bigr)/3}$
    \pic{0.24}{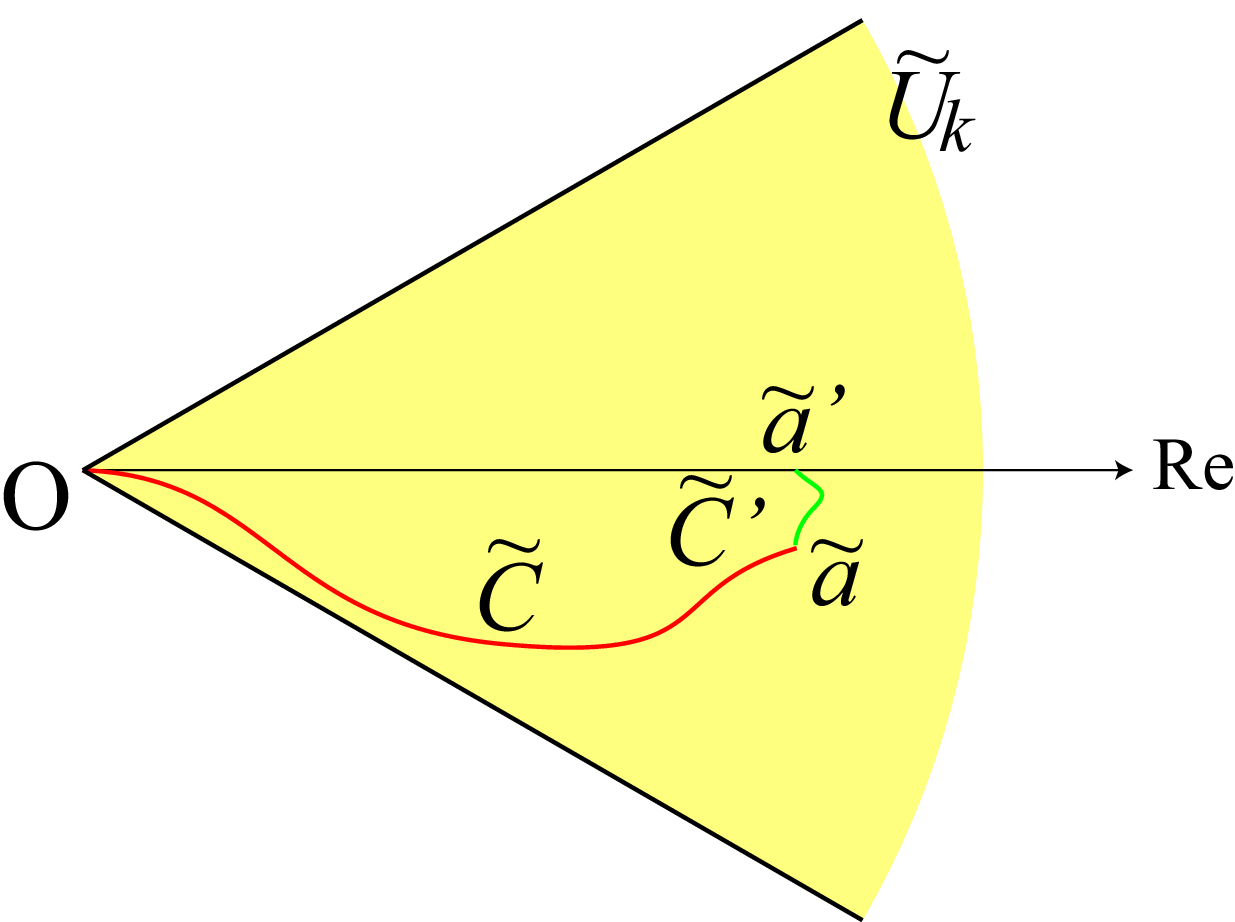}
  \end{center}
  \caption{The yellow regions are $U_k$ and $\tilde{U}_k$,
           the red curves are $C$ and $\tilde{C}$, and the green curve is $\tilde{C}'$.}
  \label{fig:U}
\end{figure}
\par
Putting
\begin{align*}
  w
  &:=
  e^{\bigl(\theta-(2k-1)\pi\bigr)\i/3}z,
  \\
  \intertext{and}
  \tilde{h}_N(w)
  &:=
  h_N\left(e^{\bigl((2k-1)\pi-\theta\bigr)\i/3}w\right),
\end{align*}
we have
\begin{equation}\label{eq:rotate}
  \int_{C}h_N(z)e^{N\times cz^3}\,dz
  =
  e^{\bigl((2k-1)\pi-\theta\bigr)\i/3}
  \int_{\tilde{C}}\tilde{h}_N(w)e^{-Nrw^3}\,dw.
\end{equation}
\par
Since $\tilde{U}_k\cap\tilde{D}$ is connected, we can choose $\tilde{a}'>0$ in $\R\cap\tilde{U}_{k}\cap\tilde{D}$ and connect $\tilde{a}$ to $\tilde{a}'$ by a path $\tilde{C}'\subset\tilde{U}_{k}\cap\tilde{D}$.
Now the function $\tilde{h}_N$ is defined in $\tilde{D}$, and we will extend $\tilde{h}_N\Bigl|_{\tilde{U}_k\cap\tilde{D}\cap\R}$ to a $C^{\infty}$ function $h_N^{\ast}(t)$ for any $t\ge0$.
Here we assume the following:
\begin{enumerate}
\item
$h_N^{\ast}(t)$ is bounded,
\item
$h_N^{\ast}(t)$ converges uniformly to a $C^{\infty}$ function $h^{\ast}(t)$,
\item$h_N^{\ast}(t)=\tilde{h}_N(t)$ and $h^{\ast}(t)=\tilde{h}(t):=h\left(e^{\bigl((2k-1)\pi-\theta\bigr)\i/3}t\right)$ for $t\in\tilde{U}_k\cap\tilde{D}\cap\R$.
\end{enumerate}
Then since $\tilde{U}_{k}\cap\tilde{D}$ is simply-connected, by Cauchy's theorem we have
\begin{equation}\label{eq:Cauchy}
  \int_{\tilde{C}}\tilde{h}_N(w)e^{-Nrw^3}\,dw
  =
  I_1-I_2-I_3,
\end{equation}
where we put
\begin{align*}
  I_1
  &:=
  \int_{0}^{\infty}h_N^{\ast}(w)e^{-Nrw^3}\,dw,
  \\
  I_2
  &:=
  \int_{\tilde{a}'}^{\infty}h_N^{\ast}(w)e^{-Nrw^3}\,dw,
  \\
  I_3
  &:=
  \int_{\tilde{C}'}\tilde{h}_N(w)e^{-Nrw^3}\,dw.
\end{align*}
\par
We use Watson's lemma \cite{Watson:PROLM1918} to evaluate $I_1$.
Putting $t:=w^3$, we have
\begin{equation*}
  I_1
  =
  \int_{0}^{\infty}
  h_N^{\ast}\left(t^{1/3}\right)\frac{1}{3t^{2/3}}e^{-Nrt}\,dt.
\end{equation*}
Since $h_N^{\ast}(s)$ uniformly converges to an analytic function $h^{\ast}(s)$ in $\tilde{D}\cap\R$, we conclude that
\begin{equation*}
  h_N^{\ast}(s)
  =
  h^{\ast}(s)
  +
  \frac{g_N(s)}{N}
\end{equation*}
with $|g_N(s)|<c$, where $c$ is a constant independent of $s$.
Since $h^{\ast}(0)=h(0)$, $h_N^{\ast}(s)$ is of the form
\begin{equation*}
  h_N^{\ast}(s)
  =
  h(0)+\frac{g_N(s)}{N}+\sum_{j=1}^{\infty}b_{j}s^j
\end{equation*}
near $0$, where $b_{j}:=\frac{1}{j!}\frac{d^j}{d\,s^j}h(0)$.
So we have
\begin{equation*}
  h_N^{\ast}\bigl(t^{1/3}\bigr)\frac{1}{3t^{2/3}}
  =
  \frac{h(0)}{3}t^{-2/3}
  +
  \frac{g_N\bigl(t^{1/3}\bigr)}{3t^{2/3}N}
  +
  \sum_{j=1}^{\infty}\frac{b_{j}}{3}t^{(j-2)/3}.
\end{equation*}
Since $|g_N(s)|<c$, we have
\begin{equation*}
  \left|
    \int_{0}^{\infty}\frac{g_N(t^{1/3})}{3t^{2/3}N}e^{-Nrt}\,dt
  \right|
  <
  \frac{c}{3N}
  \int_{0}^{\infty}t^{-2/3}e^{-Nrt}\,dt
  =
  \frac{c\Gamma(1/3)}{3r^{1/3}N^{4/3}}.
\end{equation*}
Therefore from Watson's lemma \cite[P.~133]{Watson:PROLM1918} (see also \cite[P.~20]{Wong:1989}), we have
\begin{equation}\label{eq:Watson}
\begin{split}
  I_1
  =&
  \frac{h(0)\Gamma(1/3)}{3(rN)^{1/3}}
  +
  \sum_{j=1}^{\infty}\frac{b_{j}\Gamma\bigl((j+1)/3\bigr)}{3}(rN)^{-(j+1)/3}
  +
  O(N^{-4/3})
  \\
  =&
  \frac{h(0)\Gamma(1/3)}{3(rN)^{1/3}}
  +
  O(N^{-2/3})
\end{split}
\end{equation}
as $N\to\infty$.
\par
As for $I_2$, since $|h_N^{\ast}(w)|<M$ if $w\in\R$ for some $M>0$, we have
\begin{equation}\label{eq:I2}
\begin{split}
  \left|I_2\right|
  &\le
  \int_{\tilde{a}'}^{\infty}\left|h_N^{\ast}(w)\right|e^{-rN\tilde{a}^{\prime2}w}\,dw
  =
  \frac{Me^{-{\tilde{a}^{\prime3}rN}}}{\tilde{a}^{\prime2}rN}
  <
  M_1e^{-\varepsilon_1N}
\end{split}
\end{equation}
if $N>1$, where we put $M_1:=\frac{M}{r\tilde{a}^{\prime2}}$ and $\varepsilon_1:=r\tilde{a}^{\prime3}>0$.
\par
As for $I_3$, we note that if $w\in\tilde{C}'\subset\tilde{U}_{k}$, then $\Re{w^3}>\varepsilon_2$ for some $\varepsilon_2>0$, since $\left|\arg\left(w^3\right)\right|<\pi/2$ from \eqref{eq:tildeU_k}.
So we have
\begin{equation}\label{eq:I3}
\begin{split}
  \left|I_3\right|
  &<
  \max_{w\in\tilde{C}'}\left|\tilde{h}_N(w)\right|
  \int_{\tilde{C}'}e^{-Nr\varepsilon_2}\,dw
  \le
  M_2e^{-r\varepsilon_2N},
\end{split}
\end{equation}
where we put $M_2:=\max_{w\in\tilde{C}'}\left|\tilde{h}_N(w)\right|\times\text{(length of $\tilde{C}'$)}$.
\par
From \eqref{eq:Cauchy}, \eqref{eq:I2} and \eqref{eq:I3}, we have
\begin{equation*}
  \left|
    \int_{\tilde{C}}\tilde{h}_N(z)e^{-Nrw^3}\,dw
    -
    I_1
  \right|
  \le
  \left|I_2\right|+\left|I_3\right|
  =
  O(e^{-\varepsilon_3N})
\end{equation*}
with $\varepsilon_3:=\min\{\varepsilon_1,r\varepsilon_2\}$.
Therefore from \eqref{eq:rotate} and \eqref{eq:Watson} we finally have
\begin{equation*}
  \int_{C}h_N(z)e^{N\times cz^3}\,dz
  =
  \frac{e^{\bigl((2k-1)\pi-\theta\bigr)\i/3}h(0)\Gamma(1/3)}{3r^{1/3}N^{1/3}}
  \left(1+O(N^{-1/3})\right).
\end{equation*}
This completes the proof.
\end{proof}
\begin{cor}\label{cor:cz^3}
Let $c:=re^{\theta\i}$, $D$, $h_N(z)$, $h(z)$, and $U_{k}$ be as in Lemma~\ref{lem:cz^3}.
Let $C\subset D$ be a path from $a_{k}\in U_{k}\cap D$ to $a_{k+1}\in U_{k+1}\cap D$, where $U_{4}$ means $U_{1}$.
We also assume that there exist paths $C_{k}$ from $a_{k}$ to $O$ and $C_{k+1}$ from $O$ to $a_{k+1}$ with the following properties:
\begin{enumerate}
\item
$C_{k}\setminus\{O\}\subset U_{k}\cap D$,
\item
$C_{k+1}\setminus\{O\}\subset U_{k+1}\cap D$,
\item
the path $C_{k}\cup C_{k+1}$ is homotopic to $C$ in $D$ keeping $a_{k}$ and $a_{k+1}$ fixed.
\end{enumerate}
Then we have
\begin{equation*}
  \int_{C}h_N(z)e^{N\times cz^3}\,dz
  =
  \frac{h(0)\Gamma(1/3)}{\sqrt{3}r^{1/3}N^{1/3}}\i\omega^{k}e^{-\theta\i/3}
  \left(1+O(N^{-1/3})\right),
\end{equation*}
where we put $\omega:=e^{2\pi\i/3}$.
\end{cor}
\begin{proof}
By Cauchy's theorem, we have $\int_{C}h_N(z)e^{N\times cz^3}\,dz=\int_{C_{k}\cup C_{k+1}}h_N(z)e^{N\times cz^3}\,dz$.
Then from Lemma~\ref{lem:cz^3} we have
\begin{align*}
  &\int_{C_{k}\cup C_{k+1}}h_N(z)e^{N\times cz^3}\,dz
  \\
  =&
  \frac{e^{-\theta\i/3}h(0)\Gamma(1/3)}{3r^{1/3}N^{1/3}}
  \left(e^{(2k+1)\pi\i/3}-e^{(2k-1)\pi\i/3}\right)
  \left(1+O(N^{-1/3})\right)
  \\
  =&
  \frac{e^{-\theta\i/3}h(0)\Gamma(1/3)}{\sqrt{3}r^{1/3}N^{1/3}}\i\omega^k
  \left(1+O(N^{-1/3})\right),
\end{align*}
completing the proof.
\end{proof}
\par
Now, we prove Proposition~\ref{prop:saddle}.
\begin{proof}[Proof of Proposition~\ref{prop:saddle}]
By Cauchy's theorem we study the integral $\int_{P_{k}\cup P_{k+1}}h_N(z)e^{N\eta(z)}\,dz$.
Since any point on $P_{k}$ or $P_{k+1}$ outside $\hat{D}$ satisfies the inequality $\Re\eta(z)<-\varepsilon$ for some $\varepsilon>0$, the integrals along $P_{k}$ and $P_{k+1}$ outside $\hat{D}$ is of order $O(e^{-\varepsilon N})$.
So it is enough to show that the integral $\int_{P}h_N(z)e^{N\eta(z)}\,dz$ equals the right-hand side of \eqref{eq:saddle}, where we put $P:=\bigl(-P_{k}\cup P_{k+1}\bigr)\cap\hat{D}$.
\par
Define the function $G$ so that $\eta(z)=re^{\theta\i}G(z)^3$ and that $G$ is a bijection from $\hat{D}$ to $E:=G(\hat{D})$, as described at the beginning of Section~\ref{sec:saddle}.
Let $\hat{P}$ be the image of $P$ by $G$, and $a_{k}$ and $a_{k+1}$ be the endpoints of $\hat{P}$ with $a_{k}\in V_{k}$ and $a_{k+1}\in V_{k+1}$.
Putting $w:=G(z)$ and $c:=re^{\theta\i}$, we have
\begin{equation*}
  \int_{P}h_{N}(z)e^{N\eta(z)}\,dz
  =
  \int_{\hat{P}}\gamma_N(w)e^{Ncw^3}\,dw,
\end{equation*}
since $\eta(z)=re^{\theta\i}G(z)^3$, where we put $\gamma_N(w):=h_N\bigl(G^{-1}(w)\bigr)\frac{d\,G^{-1}(w)}{d\,w}$.
Since $\frac{d}{d\,z}G(0)=1$ and $\gamma_N(w)$ converges to $\gamma(w):=h\bigl(G^{-1}(w)\bigr)\frac{d\,G^{-1}(w)}{d\,w}$, we have $\gamma(0)=h(0)$.
So from Corollary~\ref{cor:cz^3}, we conclude
\begin{equation*}
  \int_{\hat{P}}h_N(z)e^{N\eta(z)}\,dz
  =
  \frac{h(0)\Gamma(1/3)}{\sqrt{3}r^{1/3}N^{1/3}}\i\omega^{k}e^{-\theta\i/3}
  \bigl(1+O(N^{-1/3})\bigr),
\end{equation*}
completing the proof.
\end{proof}

%% file: stevedore.tex
\section{Some computer calculations on the stevedore knot}\label{sec:stevedore}
Theorem~\ref{thm:main} states that the colored Jones polynomial of the figure-eight knot $\FE$ evaluated at $(2\pi\i+\kappa)/N$ grows exponentially with growth rate determined by the Chern--Simons invariant of an affine representation associated with the pair $(\kappa,2\pi\i)$, where $e^{\kappa}=(3+\sqrt{5})/2$ is a zero of the Alexander polynomial $\Delta(\FE;t)=-t+3-t^{-1}$.
Corollary~\ref{cor:conjugate} also says that the same is true for $-\kappa$.
\par
In this appendix, we use the computer programs Mathematica \cite{Mathematica} and PARI/GP \cite{PARI2} to study the asymptotic behavior of $J_N(\St;e^{(2\pi\i\pm\tilde{\kappa})/N})$ for the stevedore knot $\St$ with $\tilde{\kappa}:=\log{2}$ , expecting a similar asymptotic behavior as $\FE$.
Note that $e^{\pm\tilde{\kappa}}=2^{\pm1}$ annihilates the Alexander polynomial $\Delta(\St;t):=-2t+5-2t^{-1}$ of $\St$.
\par
The stevedore knot $\St$ is the mirror image of the $6_1$ knot in Rolfsen's book \cite{Rolfsen:1990} (see also the Knot Atlas \cite{BarNatan/Morrison:knotatlas}) as depicted in Figure~\ref{fig:stevedore}.
Note that in the Knot Info it is denoted by $6_1$ \cite{Livingston/Moore:knotInfo}.
\begin{figure}[h]
  \pic{0.3}{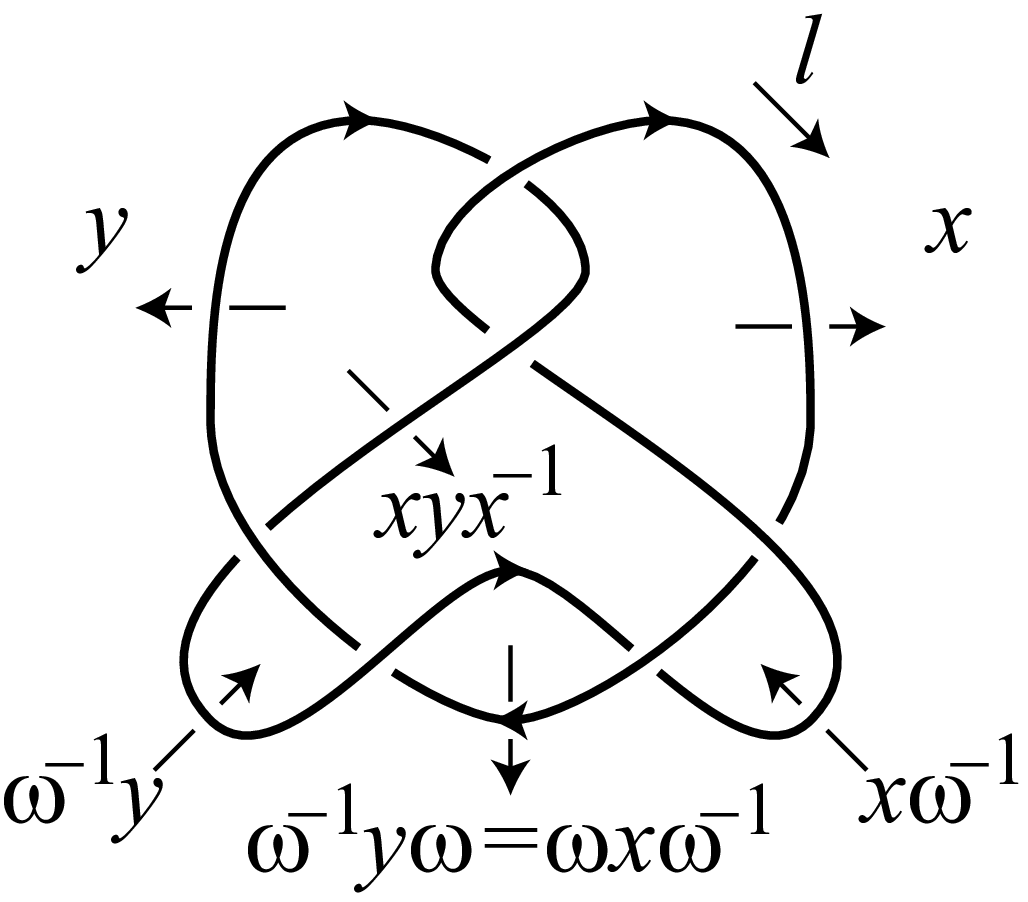}
  \caption{The stevedore knot}
  \label{fig:stevedore}
\end{figure}
\par
Due to \cite[Theorem~5.1]{Masbaum:ALGGT12003}, we obtain
\begin{equation*}
  J_N(\St;q)
  =
  \sum_{k=0}^{N-1}
  q^{-k(N+k+1)}
  \prod_{a=1}^{k}\bigl((1-q^{N+a})(1-q^{N-a})\bigr)
  \sum_{l=0}^{k}
  q^{l(k+1)}
  \frac{\prod_{b=l+1}^{k}(1-q^b)}{\prod_{c=1}^{k-l}(1-q^c)}.
\end{equation*}
Put $J_N^{\pm}:=J_N\left(\St;e^{(2\pi\i\pm\tilde{\kappa})/N}\right)$.
By using PARI/GP \cite{PARI2}, we calculate $(2\pi\i\pm\tilde{\kappa})\log\left(J_{N+1}^{\pm}/J_N^{\pm}\right)$ for $N=2,3,4,\dots,200$, and plot them by using Mathematica \cite{Mathematica} in Figures~\ref{fig:stevedore_plus_kappa} and \ref{fig:stevedore_minus_kappa}.
\begin{figure}[h]
\pic{0.6}{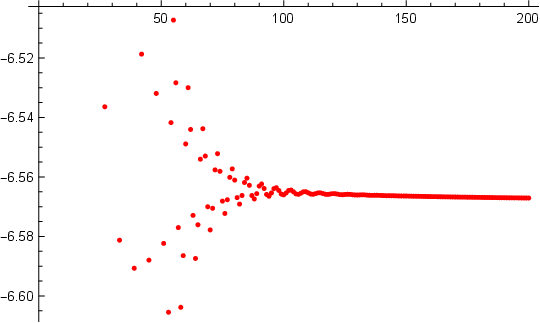}
\hfil
\pic{0.6}{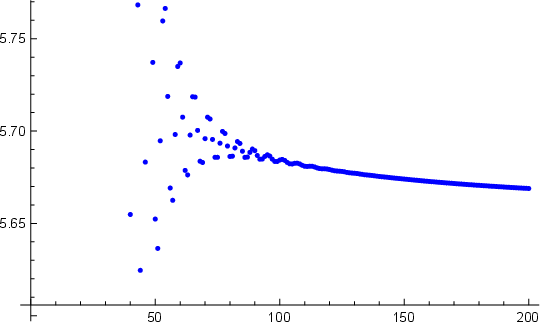}
\caption{Plots of the real part (left) and the imaginary part (right) of $(2\pi\i+\tilde{\kappa})\log\left(J_{N+1}^{+}/J_N^{+}\right)$ with $N=2,3,4,\dots,200$.}
\label{fig:stevedore_plus_kappa}
\end{figure}
\begin{figure}[h]
\pic{0.6}{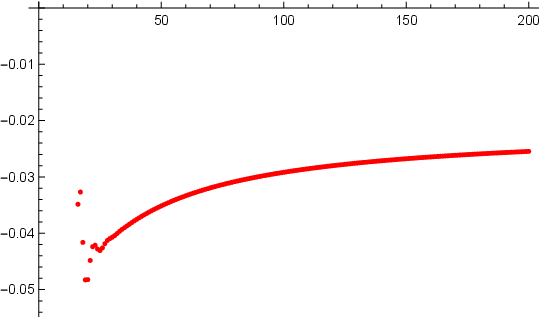}
\hfil
\pic{0.6}{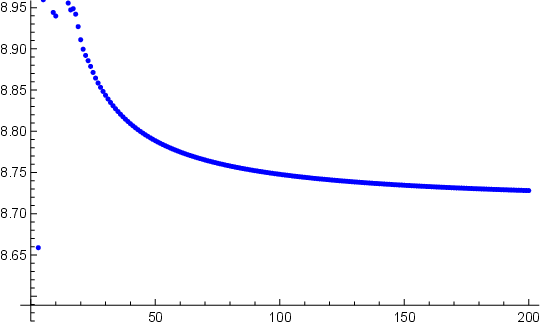}
\caption{Plots of the real part (left) and the imaginary part (right) of $(2\pi\i-\tilde{\kappa})\log\left(J_{N+1}^{-}/J_N^{-}\right)$ with $N=2,3,4,\dots,200$.}
\label{fig:stevedore_minus_kappa}
\end{figure}
The plots indicate that $J_N^{+}$ grows like $\exp\left(\frac{S_{+}}{2\pi\i+\tilde{\kappa}}N\right)\times\text{(polynomial in $N$)}$ with
\begin{equation}\label{eq:S+}
  S_{+}:=-6.485+5.697\i,
\end{equation}
and that $J_N^{-}$ grows like $\exp\left(\frac{S_{-}}{2\pi\i-\tilde{\kappa}}N\right)\times\text{(polynomial in $N$)}$ with
\begin{equation}\label{eq:S-}
  S_{-}:=-0.06880+8.747\i.
\end{equation}
Here we use Mathematica \cite{Mathematica} again to find the constants $S_{\pm}$ so that $S_{\pm}+c_{\pm,1}N^{-1}+c_{\pm,2}N^{-2}$ best fits the data.
Note that the constants $S_{\pm}$ are defined modulo integral multiples of $2\pi\i(2\pi\i\pm\tilde{\kappa})$, and that they may also be defined modulo integral multiples of $\pi^2$ because of the definition of the $\SL(2;\C)$ Chern--Simons invariant (see \S~\ref{sec:CS}).
\par
From Theorem~\ref{thm:main} we expect that $S_{\pm}=\pm2\tilde{\kappa}\pi\i$.
However, since $\pm2\tilde{\kappa}\pi\i=\pm4.355\i$, neither $S_{+}$ nor $S_{-}$ fits with $\pm2\tilde{\kappa}\pi\i$ even modulo $2\pi\i(\pm2\tilde{\kappa}+2\pi\i)$ nor $\pi^2$.
\par
Now, let us seek for other interpretations of $S_{\pm}$.
\par
Let $x$, $y$ be elements in the fundamental group $G_{\St}:=\pi_1(S^3\setminus\St)$ as indicated in Figure~\ref{fig:stevedore}.
Then the group $G_{\St}$ has the following presentation
\begin{equation*}
  G_{\St}
  =
  \langle
    x,y\mid\omega^2x=y\omega^2
  \rangle,
\end{equation*}
where we put $\omega:=xy^{-1}x^{-1}y$ as in the case of the figure-eight knot.
The preferred longitude $l$ is given as $x^{3}\omega^{-2}{\overleftarrow\omega}^{-2}x^{-3}$, where $\overleftarrow\omega:=yx^{-1}y^{-1}x$ as before.
\par
Let $\rho\colon G_{\St}\to\SL(2;\C)$ be a non-Abelian representation.
Due to R.~Riley it is of the form
\begin{equation*}
  \rho(x)
  =
  \begin{pmatrix}
    m^{1/2}&1 \\
    0      &m^{-1/2}
  \end{pmatrix},
  \quad
  \rho(y)
  =
  \begin{pmatrix}
    m^{1/2}&0 \\
    d     &m^{-1/2}
  \end{pmatrix}
\end{equation*}
up to conjugation, for some $m\ne0$ and $d$.
\par
From the relation $\omega^2x=y\omega^2$, $d$ and $m$ should satisfy the following equation known as Riley's equation.
\begin{multline*}
  d^4
  +\left(2(m+m^{-1})-5\right)d^3
  +\left((m^2+m^{-2})-6(m+m^{-1})+13\right)d^2
  \\
  -\left(m^2+m^{-2}-7(m+m^{-1})+14\right)d
  -\left(2(m+m^{-1})-5\right)
  =0.
\end{multline*}
We call the left hand side of this equation the Riley polynomial.
\par
If $(m,d)=(1,0.1049+1.552\i)$, then $\rho$ is the holonomy representation of $G_{\St}$ and defines the complete hyperbolic structure of $S^3\setminus\St$.
If $(m,d)=(2,0)$ or $(1/2,0)$, then $\rho$ gives an affine representation.
\par
Let us consider irreducible representations corresponding to $1\le m\le2$.
\par
The Riley polynomial is a quartic equation with respect to $d$, and there are four solutions $d_1(m), d_2(m), d_3(m)$, and $d_4(m)$.
To describe them we introduce the following functions.
Let $D(m)$ be the discriminant of the Riley polynomial with respect to $d$, that is,
\begin{multline*}
  D(m)
  :=
  5(m^{6}+m^{-6})
  -32(m^{5}+m^{-5})
  +56(m^{4}+m^{-4})
  -118(m^3+m^{-3})
  \\
  +124(m^2+m^{-2})
  +32(m+m^{-1})
  +123.
\end{multline*}
We also put
\begin{equation*}
\begin{split}
  A(m)
  :=&
  4B(m)C(m)^{-1/3}+4C(m)^{1/3}
  +3\left(2(m+m^{-1})-5\right)^2
  \\
  &
  -8\left(m^2+m^{-2}-6(m+m^{-1})+13\right),
\end{split}
\end{equation*}
where
\begin{align*}
  B(m)
  :=&
  m^4+m^{-4}-6(m^3+m^{-3})+5(m^2+m^{-2})+3(m+m^{-1})+9,
  \\
  C(m)
  :=&
  \frac{3\sqrt{3}}{2}\sqrt{-D(m)}+m^6+m^{-6}-9(m^5+m^{-5})+21(m^4+m^{-4})
  \\
  &-\frac{9}{2}(m^{3}+m^{-3})+6(m^2+m^{-2})-27(m+m^{-1})-\frac{31}{2}.
\end{align*}
We also put
\begin{align*}
  J_{\pm}(m)
  :=&
  \pm3\sqrt{3}\bigr(2(m+m^{-1})+1\bigr)A(m)^{-1/2}
  -2B(m)C(m)^{-1/3}-2C(m)^{1/3}
  \\
  &
  -8\bigl(m^2+m^{-2}-6(m+m^{-1})+13\bigr)
  +3\bigl(2(m+m^{-1})-5\bigr)^2.
\end{align*}
\par
Now define the following four functions for $1\le m\le2$.
\begin{align*}
  d_1(m)
  &:=
  \frac{-1}{12}
  \left(6(m+m^{-1})-15+\sqrt{3}\sqrt{A(m)}+\sqrt{6}\sqrt{J_{-}(m)}\right),
  \\
  d_2(m)
  &:=
  \frac{-1}{12}
  \left(6(m+m^{-1})-15+\sqrt{3}\sqrt{A(m)}-\sqrt{6}\sqrt{J_{-}(m)}\right),
  \\
  d_3(m)
  &:=
  \frac{-1}{12}
  \left(6(m+m^{-1})-15-\sqrt{3}\sqrt{A(m)}+\sqrt{6}\sqrt{J_{+}(m)}\right),
  \\
  d_4(m)
  &:=
  \frac{-1}{12}
  \left(6(m+m^{-1})-15-\sqrt{3}\sqrt{A(m)}-\sqrt{6}\sqrt{J_{+}(m)}\right).
\end{align*}
Note the following:
\begin{itemize}
\item
$A(m)$, $B(m)$, $D(m)$, $J_{+}(m)$, and $J_{-}(m)$ are in $\R$.
\item
$A(m)>0$, $B(m)>0$, and $J_{-}(m)<0$ for $1\le m\le2$,
\item
$D(m)>0$ for $1\le m<m_0$, $D(m_0)=0$, and $D(m)<0$ for $m_0<m\le2$, where $m_0=1.950$ is the unique solution to the equation $D(m)=0$ between $1$ and $2$,
\item
$J_{+}(m)<0$ for $1\le m<m_0$, $J_{+}(m_0)=0$, and $J_{+}(m)>0$ for $m_0<m\le2$,
\item
$\Im{C(m)}=0$ and $\Re{C(m)}>0$ for $m_0\le m\le2$, and $\Im{C(m)}>0$ for $1\le m<m_0$,
\end{itemize}
which are checked by Mathematica (the author does not have proofs).
\par
We plot, by using Mathematica, the complex valued functions $d_i(m)$ ($i=1,2,3,4$) for $1\le m\le2$ on the complex plane as in Figure~\ref{fig:d}.
\begin{figure}[h]
  \pic{0.7}{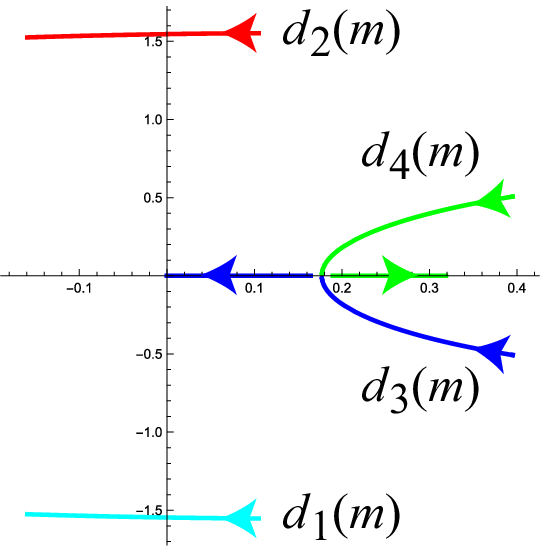}
  \caption{The cyan, red, blue, and green curves indicate
  $d_1(m)$, $d_2(m)$, $d_3(m)$, and $d_4(m)$, respectively.
  The arrows indicate the directions of increase with respect to $m$. }
  \label{fig:d}
\end{figure}
The following facts are also suggested by Mathematica (see Figure~\ref{fig:d}):
\begin{itemize}
\item
$d_2(m)=\overline{d_1(m)}$, and $d_4(m)=\overline{d_3(m)}$ for $1\le m\le2$,
\item
$d_2(1)=0.1049+1.552\i$ and $d_2(2)=-0.1595+1.525\i$.
\item
$d_3(m)\in\R$ and $d_4(m)\in\R$ for $m_0\le m\le2$,
\item
$d_3(m_0)=d_4(m_0)=0.1770$, $d_3(2)=0$, and $d_4(2)=0.3189$,
\item
$d_3(1)=0.3951-0.5068\i$.
\end{itemize}
Therefore, for each $i$, $d_i(m)$ gives an irreducible representation $\rho_m\colon G_{\St}\to\SL(2;\C)$ except for $d_3(2)$, and if $m\ne m_0$, they are mutually distinct.
\par
If we write $\rho_{d_i(m)}$ for the irreducible representation corresponding to $d_i(m)$, then we have the following:
\begin{itemize}
\item
$\rho_{d_i(1)}$ is a parabolic representation for $i=1,2,3,4$,
\item
$\rho_{d_3(2)}$ is an affine representation since $d_3(2)=0$,
\item
$\rho_{d_2(1)}$ is the holonomy representation, and $\rho_{d_1(1)}$ gives the holonomy representation for the mirror image of $\St$, because we have
\begin{align*}
  \rho_{d_2(1)}(l)
  &=
  \begin{pmatrix}
    -1&-1.827-2.565\i \\
     0&-1
  \end{pmatrix},
  \\
  \rho_{d_1(1)}(l)
  &=
  \begin{pmatrix}
    -1&-1.827+2.565\i \\
     0&-1
  \end{pmatrix}.
\end{align*}
\end{itemize}
\par
Let $\lambda(m)$ be the $(1,1)$-entry of $\rho_{d_2(m)}(l)$, and put $v(u):=2\log\lambda\bigl(e^{u/2}\bigr)$, where we choose the logarithm branch so that $v(0)=0$.
Then the $\SL(2;\C)$ Chern--Simons invariant of $\rho_{d_2(e^{u/2})}$ associated with $(u,v(u))$ is given as
\begin{equation*}
  \CS_{u,v(u)}(\rho_{d_2(e^{u/2})})
  =
  \cv(S^3\setminus\St)
  +
  \frac{1}{2}\int_{0}^{u}v(s)\,ds
  -
  \frac{1}{4}uv(u),
\end{equation*}
where $\cv(S^3\setminus\St)=-6.791 + 3.164\i$ is the complex volume, which is defined to be $\i\Vol(S^3\setminus\St)-2\pi^2\CS^{\rm{SO}(3)}(S^3\setminus\St)\pmod{\pi^2}$ with $\CS^{\rm{SO}(3)}$ the $\rm{SO}(3)$ Chern--Simons invariant of the Levi--Civita connection.
Here the complex volume and $\rm{SO}(3)$ Chern--Simons invariant are taken from the Knot Info, where $\Vol(S^3\setminus\St)=3.163963229$ and $\CS^{\rm{SO}(3)}(S^3\setminus\St)=0.155977017$.
Observe that $-0.155977017\times2\pi^2+\pi^2=6.79074$.
Note that $\cv(S^3\setminus\St)$ coincides with the $\SL(2;\C)$ Chern--Simons invariant $\CS_{(0,0)}(\rho_{2}(1))$.
See \cite[Chapter~5]{Murakami/Yokota:2018}.
\begin{figure}[h]
\pic{0.6}{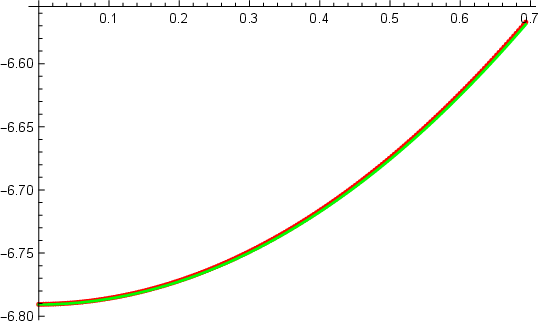}
\hfil
\pic{0.6}{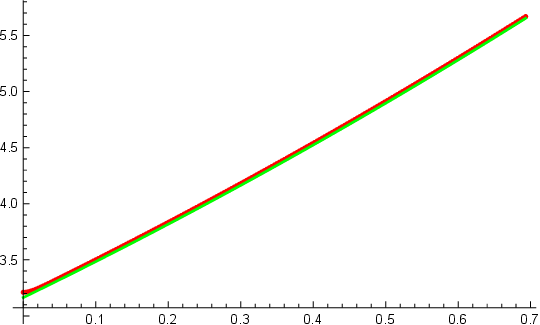}
\caption{The left picture shows the plots of the real parts of $(2\pi\i+u)\log\left(J_{201}(u)/J_{200}(u)\right)$ (red) and $S_{u}(\St)$ (green) for $0\le u\le\log{2}$ (left), and the right picture shows the plots of the imaginary parts of $(u+2\pi\i)\log\left(J_{201}(u)/J_{200}(u)\right)$ (red) and $S_{u}(\St)$ (green) for $0\le u\le\log{2}$ (left), where we put $J_N(u):=J_N(\St;e^{(u+2\pi\i)/N})$ and we use PARI/GP and Mathematica.}
\label{fig:CS}
\end{figure}
\par
Putting
\begin{equation}\label{eq:S_stevedore}
  S_{u}(\St)
  :=
  \CS_{u,v(u)}(\rho_{d_2(e^{u/2})})
  +u\pi\i+\frac{1}{4}uv(u),
\end{equation}
the graphs depicted in Figure~\ref{fig:CS} indicate that
\begin{equation*}
  J_{N}(\St;e^{(u+2\pi\i)/N})
  \underset{N\to\infty}{\sim}
  \text{(polynomial in $N$)}\times
  \exp\left(\frac{S_{u}(\St)}{u+2\pi\i}N\right)
\end{equation*}
for $0\le u\le\tilde{\kappa}=\log{2}$.
When $u=\tilde{\kappa}$, Mathematica calculates $S_{\tilde{\kappa}}(\St)=-6.569+5.653\i$, which is close to $S_{+}$ appearing in \eqref{eq:S+}.
\par
Note that the case $u=\tilde{\kappa}$ does not correspond to an affine representation.
This also suggests that for $0<u\le\tilde{\kappa}$ the representation $d_2(e^{u/2})$ induces an incomplete hyperbolic structure of $S^3\setminus\St$ but the author does not know whether it is correct or not.
The author does not know either any topological/geometric interpretation about the asymptotic behavior of $J_N(\St;e^{(u+2\pi\i)/N})$ for $u<0$.
\par
Compare this with Theorem~\ref{thm:main} and Corollary~\ref{cor:conjugate}, where $S_{\pm\kappa}(\FE)=\pm2\kappa\pi\i$ are the Chern--Simons invariants of affine representations, which correspond to the fact that when $u=\pm\kappa$ the hyperbolic structure collapses.